\documentclass[final,onefignum,onetabnum]{siamart190516}

\usepackage{colordvi,epsfig,amsmath,mathtools,amssymb,url}
\usepackage{amsfonts}
\usepackage{graphicx}
\usepackage{epstopdf}
\usepackage{algorithmic}
\usepackage{multirow}
\usepackage{siunitx,booktabs}
\usepackage{hyperref}
\usepackage{mathrsfs}

\graphicspath{{./Images/}}
\ifpdf
\DeclareGraphicsExtensions{.eps,.pdf,.png,.jpg}
\else
\DeclareGraphicsExtensions{.eps}
\fi

\usepackage[textwidth=.8\marginparwidth,
  backgroundcolor=lightgray,
  linecolor=gray,
  textsize=scriptsize,
  colorinlistoftodos]{todonotes}
\renewcommand{\Re}{\mathbb{R}}

\newsiamremark{remark}{Remark}
\newsiamremark{hypothesis}{Hypothesis}
\crefname{hypothesis}{Hypothesis}{Hypotheses}
\newsiamthm{claim}{Claim}

\headers{Sparse Approximations with Interior Point Methods}{V. De Simone, D. di Serafino, J. Gondzio, S. Pougkakiotis, and M. Viola}

\title{
Sparse Approximations with Interior Point Methods\thanks{
\textbf{Funding:} {this work was funded by various institutes and research programs. V.~De Simone, D.~di  Serafino and M.~Viola were supported by the Istituto Nazionale di Alta Matematica, Gruppo Nazionale per il Calcolo Scientifico (INdAM-GNCS), and by the V:ALERE Program of the University of Campania ``L. Vanvitelli'', 
Italy. S.~Pougkakiotis was supported by a Principal's Career Development scholarship from the University of Edinburgh, as well as a scholarship from A. G. Leventis Foundation. J.~Gondzio and S.~Pougkakiotis were also supported by the Google project ``Fast (1 + $x$)-order Methods for Linear Programming".
We wish to remark that this study does not have any conflict of interest to disclose. }}
}

\author{Valentina De Simone\thanks{Department of Mathematics and Physics, University of Campania ``L. Vanvitelli'', Caserta, Italy         (\email{valentina.desimone@unicampania.it}, \email{marco.viola@unicampania.it}).}
    \and Daniela di Serafino\thanks{Department of Mathematics and Applications, University of Naples Federico II, Napoli, Italy (\email{daniela.diserafino@unina.it}).}
    \and Jacek Gondzio\thanks{School of Mathematics and Maxwell Institute for Mathematical Sciences, University of Edinburgh, Edinburgh, United Kingdom
    (\email{J.Gondzio@ed.ac.uk}, \email{S.Pougkakiotis@sms.ed.ac.uk}).}
    \and Spyridon Pougkakiotis\footnotemark[4]
    \and Marco Viola\footnotemark[2]
}

\date{November 24, 2021}

\usepackage{amsopn}
\DeclareMathOperator{\diag}{diag}
\DeclareMathOperator{\rank}{rank}
\DeclareMathOperator{\sign}{sign}

\def\R{\mathbb R}

\ifpdf
\hypersetup{
  pdftitle={Sparse Approximations with Interior Point Methods},
  pdfauthor={V. De Simone, D. di Serafino, J. Gondzio, S. Pougkakiotis, and M. Viola}
}
\fi

\begin{document}

\maketitle
\centerline{\footnotesize VERSION 3 -- November 24, 2021}

\begin{abstract}
Large-scale optimization problems that seek sparse solutions have become ubiquitous. 
They are routinely solved with various specialized first-order methods.
Although such methods are often fast, they usually struggle with not-so-well 
conditioned problems. In this paper, specialized variants 
of an interior point-proximal method of multipliers
are proposed and analyzed for problems of this class. Computational experience
on a variety of problems, namely, multi-period portfolio optimization, classification of data
coming from functional Magnetic Resonance Imaging,  restoration of images corrupted by Poisson noise,
and classification via regularized logistic regression, provides substantial evidence that interior point methods,
equipped with suitable linear algebra, can offer a noticeable advantage over first-order approaches.
\end{abstract}




\section{Introduction\label{sec:Intro}}
We are concerned with the efficient solution of a class of problems 
which are very large and are expected to yield sparse solutions. 
In practice, the sparsity is often induced by the presence of $\ell_1$ 
norm terms in the objective. We assume that a general problem of the following form
\begin{eqnarray}
    \begin{array}{rl}
        \displaystyle \min_{x} & f(x) + \tau_1 \| x \|_1 + \tau_2 \| L x \|_1 \\
        \text{s.t.}  & A x = b,  
    \end{array}
    \label{SparseZero}
\end{eqnarray}
needs to be solved, where $f:\Re^n\mapsto\Re$ is a twice continuously differentiable convex function,
$L\in\Re^{l\times n}$, $A\in\Re^{m\times n}$, $b\in\Re^{m}$, $m \le n$, and $\tau_1,\tau_2>0$. We are particularly
interested in problems for which $f(x)$ displays some level of separability. 
The terms $\| x \|_1$ and $\| L x \|_1$ induce sparsity in the vector 
$x$ and/or in some (possibly redundant) dictionary $L x$. 
Numerous real-life problems can be recast into the form \eqref{SparseZero}.
Among the various application areas, one can find portfolio optimization \cite{Markowitz1959},
signal and image processing \cite{ChenDonohoSaunders,rudin:1992}, classification in statistics \cite{tibshirani:2005} 
and machine learning \cite{mybib:Vapnik}, inverse problems \cite{TroppWrightIEEE} 
and compressed sensing \cite{CandesRombergTao}, to mention just a few.

Optimization problems arising in these applications are usually solved 
by different specialized variants of first-order methods. Indeed, 
highly specialized and tuned to a narrow class of problems, first-order 
methods often outperform standard off-the-shelf second-order techniques; 
the latter might be too expensive or might struggle with excessive 
memory requirements. Such comparisons are not fair though. 
With this paper we hope to change the incorrect opinion on the second-order 
methods. 
\par Various second-order approaches have been proposed in the literature for problems of the form of \eqref{SparseZero}. In particular, one might employ proximal (projected) Newton-type methods (see \cite{LeeSunSaundersSIAMOpt,SchmidtKinSraProjNewton}) in which a proximal term is added to deal with the non-smooth part of the objective function (unless its minimum-norm subgradient has a closed form solution, in which case the proximal term can be excluded). Alternatively, such problems can be solved by means of standard semi-smooth Newton methods (see \cite{KlatteKummerRegCalcMeth,KojimaShindoOperRes} and the references therein) or semi-smooth Newton methods combined with the augmented Lagrangian method (see, e.g., \cite{LiSunTohSIAMOpt}). The aforementioned approaches employ line-search schemes that allow one to show linear or local superlinear convergence, given certain assumptions. For methods involving proximal terms, superlinear convergence is only guaranteed when the associated penalty parameters increase to infinity.
\par Here we consider Interior Point Methods (IPMs), which exhibit better convergence (in practice and in theory), at the expense of worse conditioning of the associated linear systems that have to be solved at every IPM iteration.
When efficiently implemented and specialized to a particular 
application, interior point methods offer an attractive alternative. 
They can be equally (or more) efficient than the best first-order methods available, 
and they deliver unmatched robustness and reliability. 
\par The specializations of interior point methods proposed in this paper do not 
go beyond what has been commonly exploited by first-order methods. 
Namely we propose: 
\begin{itemize}
   \item 
to exploit special features of the problems in the linear algebra of IPMs, and 
   \item 
to take advantage of the expected sparsity of the optimal solution.
\end{itemize}

In order to achieve our goals we propose to convert sparse approximation problems 
to standard smooth nonlinear convex 
programming ones by replacing the $\ell_1$ norm terms with a usual modeling trick 
in which an absolute value $|a|$ is substituted with the sum 
of two non-negative parts, $|a| = a^+ + a^-$, where $a^+ = \max\{a,0\}$ 
and $a^- = \max\{-a,0\}$.
%
By introducing the auxiliary variable \mbox{$d=Lx\in\Re^l$,} 
problem~(\ref{SparseZero}) is then transformed into the following one:
\begin{eqnarray}
    \begin{array}{cl}
       \displaystyle \min_{x^+, x^-, d^+, d^-} & f(x^+ - x^-) + \tau_1 (e_n^\top x^+ + e_n^\top x^-) 
        + \tau_2 (e_l^\top d^+ + e_l^\top d^-), \\
        \text{s.t.}  & A (x^+ - x^-) = b, \\ 
        & L (x^+ - x^-) = d^+ - d^-, \\
        & x^+, x^-, d^+, d^- \geq 0, 
    \end{array}
    \label{SparseOne}
\end{eqnarray}
where $x^+,x^-\in\Re^n$ are such that $x = x^+ - x^-$, $d^+,d^-\in\Re^l$ 
are such that $d = d^+ - d^-$, and $e_n\in\Re^n$, $e_l\in\Re^l$ 
are vectors with all entries equal to 1.
It is worth observing that \eqref{SparseZero} 
and \eqref{SparseOne} are equivalent; indeed the presence of linear 
terms which penalize for the sum of positive and negative parts 
of vectors $x$ and $d$ guarantees that at optimality only one 
of the split variables can take a nonzero value. We also note that
the number of variables is greater than or equal to the number of equality constraints in \eqref{SparseOne}.
Although \eqref{SparseOne} is larger than~\eqref{SparseZero} 
because the variables have been replicated and new constraints 
have been added, it is in a form eligible to a straightforward application 
of an interior point method. We expect that the well-known 
ability of IPMs to handle large sets of linear equality 
and non-negativity constraints will compensate for this increase 
of the problem dimension.

IPMs employ Newton method to solve a sequence of so-called logarithmic barrier subproblems (see Section~\ref{sec:IPMsBack} for details).
In standard implementations of IPMs this requires many involved linear algebra 
operations (building and inverting the Hessian matrix), and for large problems 
it might become prohibitively expensive. In this paper we demonstrate 
that the use of inexact Newton method \cite{mybib:Bellavia,iter:CdAdSdST,Gondzio-SIOPT2013} 
combined with a knowledgeable choice and appropriate tuning of linear algebra 
solvers (see~\cite{iter:dAdSdS,iter:dSO,IPM_25_years_later,iter:GT-COAP} 
and the references therein) is the key to success when developing an IPM 
specialized to a particular class of problems. 
We also demonstrate an attractive ability of IPMs to select important 
variables and prematurely drop the irrelevant ones, a feature which is very 
well suited to solving sparse approximation problems in which the majority 
of variables are expected to be zero at optimality. It is worth mentioning 
at this point that our understanding of the features of IPMs applied to sparse 
approximation problems benefitted from the earlier studies which focused 
on compressed sensing problems \cite{FountoulakisGondzio2016,FountoulakisEtAl2013}.

Ultimately, we provide computational evidence that IPMs can be more 
efficient than methods which are routinely used for the solution of sparse 
approximation problems by exploiting only first-order information.

\subsection*{Notation}
Throughout this paper we use lowercase Roman and Greek fonts to indicate scalars and vectors
(the nature is clear from the context). Capital italicized Roman fonts are used to indicate matrices.
Superscripts are used to denote the components of a vector/matrix. As an example, given
$M\in\Re^{m \times n}$, $v\in\Re^n$, $\mathcal{R}\subseteq \left\lbrace 1, \ldots, m \right\rbrace$,
and $\mathcal{C} \subseteq \left\lbrace 1, \ldots, n \right\rbrace$, we set
$v^{\mathcal{C}} := (v^i)_{i\in\mathcal{C}}$ and $M^{\mathcal{R},\mathcal{C}} := \left( m^{ij} \right)_{i\in\mathcal{R}, j\in\mathcal{C}}$,
where $v^i$ is the $i$-th entry of $v$ and $m^{ij}$ the $(i,j)$-th entry of $M$.
We use $\lambda_{\min}(B)$ ($\lambda_{\max}(B)$, respectively) to denote the minimum (maximum) eigenvalue
of an arbitrary square matrix $B$ with real eigenvalues. Similarly, $\sigma_{\min}(B)$ ($\sigma_{\max}(B)$, respectively) denotes
the minimum (maximum) singular value of an arbitrary rectangular matrix $B$. We use $B \succ 0$ to indicate that
a square matrix $B$ is symmetric positive definite.
We use $e_n$ and $0_n$ to denote a column vector of size $n$ with all entries equal to $1$ and $0$, respectively.
Moreover, we use $I_n$ to indicate the identity matrix of size $n$ and $0_{m,n}$ to denote the zero matrix
of size $m \times n$. We use subscripts to denote the elements of a sequence, e.g., $\left\lbrace x_k \right\rbrace $.
Norms $\| \cdot \|$ are $\ell_2$. Other norms are identified by adding suitable subscripts. For any finite set $\mathcal{A}$,
we denote by $\vert \mathcal{A} \vert$ its cardinality. Finally, when referring to convex programming problems,
we implicitly assume that the problems are linearly constrained.

\subsection*{Structure of the article}
The rest of this article is organized as follows. In Section~\ref{sec:IPMsBack} we briefly describe IPMs for convex programming,
focusing in particular on the \textit{Interior Point-Proximal Method of Multipliers} (IP-PMM), which is used in the subsequent sections. The choice of
IP-PMM is motivated by the fact that it merges an infeasible IPM with the Proximal Method of Multipliers (PMM), in order to keep the fast and
reliable convergence properties of IPMs and the strong convexity of the PMM subproblems, thus achieving better efficiency and robustness than both methods.
In this section we also outline the testing environment used throughout the paper. In Sections~\ref{sec:Portfolio} to~\ref{sec:ClassificationRLR}
we present four applications formulated as optimization problems with sparsity sought in the solutions, and recast them in the form~\eqref{SparseOne}.
In detail, in Section~\ref{sec:Portfolio} we focus on a multi-period portfolio selection strategy, in Section~\ref{sec:fMRI} on the classification of data
coming from functional Magnetic Resonance Imaging (fMRI), in Section~\ref{sec:ImgRestor} on the restoration of images corrupted by Poisson noise,
and in Section~\ref{sec:ClassificationRLR} on linear classification through regularized logistic regression. The first two applications yield convex quadratic programming
problems, while the remaining ones yield general nonlinear convex programming problems. For each application, we provide a brief
description of its mathematical model and explain how IP-PMM is specialized for that case in terms of linear algebra solvers, including variable dropping
strategies to help sparsification; we also show the results of computational experiments, including comparisons with state-of-the art methods widely used by
the scientific community on the selected problems.


\section{Interior Point Methods for Convex Programming\label{sec:IPMsBack}}

\noindent We consider the following convex programming problem:
\begin{equation} \label{non-regularized primal} 
    \underset{x}{\text{min}} \  f(x), \quad \text{s.t.} \ \  Ax = b, \ \   x \geq 0, 
\end{equation}
\noindent where $x \in \Re^n$, $A \in \Re^{m\times n}$, and $f \colon \Re^n \mapsto \Re$ is a twice differentiable convex function. Using the
Lagrangian duality theory \cite{mybib:Bertsekas}, and defining a function $F(w): \Re^{2n + m} \mapsto \Re^{2n+m}$, we write the KKT
(optimality) conditions as follows:
\begin{equation} \label{non-regularized FOC}
    F(w) = \begin{bmatrix}
        \nabla f(x) - A^\top y - z \\
        Ax - b\\
        XZe_n\\
    \end{bmatrix}
    = 
    \begin{bmatrix}
        0_n\\
        0_m\\
        0_n\\
    \end{bmatrix},
\end{equation} 
\noindent where $y \in \Re^m$ and $z \in \Re^n$ are the Lagrange multipliers corresponding to the equality and inequality constraints respectively, while $X,\ Z \in \Re^{n \times n}$ denote the diagonal matrices with diagonal entries $x^i$ and $z^i$ (respectively), $\forall\ i \in \{1,\ldots,n\}$.
\par Problem~\eqref{non-regularized primal} can be solved using a primal-dual IPM. There are numerous variants of IPMs and the reader is referred to \cite{IPM_25_years_later} for an extended literature review. IPMs handle the non-negativity constraints of the problems with logarithmic barriers in the objective. That is, at each iteration $k$, we choose a \textit{barrier parameter} $\mu_k$ and form the \textit{logarithmic barrier problem}:
\begin{equation} \label{Primal Barrier Problem}
        \min_{x} \ f(x) - \mu_k \sum_{j=1}^n \ln x^j , \quad \text{s.t.} \  \ Ax = b.
\end{equation}
\par Then, a damped Newton method (or possibly an inexact variant of it \cite{mybib:Bellavia,Gondzio-SIOPT2013,Nester_Nemir_IPM_Convex}) is usually employed in order to approximately solve problem \eqref{Primal Barrier Problem}. Applying it to the optimality conditions of \eqref{Primal Barrier Problem}, and further forming the augmented system (as is done in Section \ref{sec:Interior Point-Proximal Method of Multipliers}), we obtain a system of the following form:
\begin{equation} \label{non-regularized Augmented System}
    \begin{bmatrix} 
        -(\nabla^2 f(x_k) +\Theta_k^{-1}) &   A^\top \\
        A & 0_{m,m}
    \end{bmatrix}
    \begin{bmatrix}
        \Delta x_k\\ 
        \Delta y_k
    \end{bmatrix}
    = 
    \begin{bmatrix}
        \nabla f(x_k) - A^\top y_k - \sigma_{k} \mu_k X_k^{-1} e\\
        b-Ax_k
    \end{bmatrix},
\end{equation}
\noindent where $\Theta_k = X_k Z_k^{-1}$ and the entries of $x_k$ and $z_k$ are maintained positive throughout the algorithm, allowing the use of the logarithmic barrier (see Section~\ref{sec:Interior Point-Proximal Method of Multipliers} for details). One can observe that the matrix $\Theta_k$ contains some very large and some very small elements close to optimality. Hence, the matrix in (\ref{non-regularized Augmented System}) becomes increasingly ill-conditioned, as the method progresses. Notice that as $\mu_k \rightarrow 0$, an optimal solution of \eqref{Primal Barrier Problem} converges to an optimal solution of \eqref{non-regularized primal}. Polynomial convergence of such methods (with respect to the number of variables $n$), for various classes of problems, has been proved multiple times in the literature (see for example \cite{Nester_Nemir_IPM_Convex,Zhang_Complex_IPM}). 
\par A system like \eqref{non-regularized Augmented System} can either be solved directly (using an appropriate factorization, as in \cite{Gondz_Altman_Reg_IPM,Orb_Fried_Reg_IPM,paper_IP_PMM}) or iteratively (using an appropriate Krylov subspace method, as in \cite{paper_IP_PMM_inexact,iter:CdAdSdS-2007b,iter:dAdSdS,iter:GT-COAP}). While the former approach is very general, it becomes problematic as the problem size increases. On the other hand, iterative methods (accompanied by appropriate preconditioners) may be difficult to generalize. However, if applied to specific classes of problems, they make possible solving huge-scale instances, by avoiding the explicit storage of the problem matrices.


\subsection{Regularization in IPMs\label{sec:Regularization}}
\par In the context of IPMs, it is often beneficial to include some regularization, in order to improve the spectral properties of the system matrix in \eqref{non-regularized Augmented System}. For example, notice that if the constraint matrix $A$ is rank deficient, then the matrix in \eqref{non-regularized Augmented System} might not be invertible. The latter can be immediately addressed by the introduction of a dual regularization, say $\delta > 0$, ensuring that $\rank([A\; \delta I_m] ) = m$. The introduction of a primal regularization, say $\rho > 0$, can ensure that the matrix $\nabla^2 f(x_k) +\Theta_k^{-1} + \rho I_n$, has eigenvalues that are bounded away from zero, and hence a significantly better worst-case conditioning than that of $\nabla^2 f(x_k)+\Theta_k^{-1}$. To produce a diagonal term in the (2,2) block Vanderbei added artificial variables to all the constraints \cite{mybib:Vanderbei-SQM}.
Saunders and Tomlin \cite{mybib:ST-Reg,mybib:S-pdReg} achieved a similar 
result for the (1,1) and (2,2) blocks, by adding Tikhonov-type regularization terms to the original problem. In the aforementioned approaches, in order to guarantee that the solution of the original problem is retrieved, one has to ensure that the regularization parameters are smaller than some unknown (nonzero) value, in which case the regularization can be shown to be \emph{exact} (see \cite{friedlanderSIAMOpt}, and the references therein).
\par In later works, these Tikhonov-type regularization methods were replaced by algorithmic regularization schemes. In particular, one can observe that a very natural way of introducing primal regularization to problem (\ref{non-regularized primal}), is through the application of the primal proximal point method. Similarly, dual regularization can be incorporated through the application of the dual proximal point method. This is a well-known fact. The authors in~\cite{Gondz_Altman_Reg_IPM} presented a primal-dual regularized IPM for convex Quadratic Programming (QP), and interpreted this regularization as the application of the proximal point method. Subsequently, the authors in~\cite{Orb_Fried_Reg_IPM} developed a primal-dual regularized IPM, which applies PMM to solve convex QP problems, and employs a single IPM iteration for approximating the solution of each PMM subproblem. There, global convergence of the method was proved, under some assumptions. A variation of the method proposed in~\cite{Orb_Fried_Reg_IPM} is given in~\cite{Pougk_Gondz_non_diag_reg}, where general non-diagonal regularization matrices are employed, as a means of further improving factorization efficiency. Then, the authors in~\cite{paper_IP_PMM} proposed an IP-PMM and proved (under standard assumptions) that it achieves convergence to an $\epsilon$-optimal solution in a polynomial (in $n$) number of iterations for convex QP problems and for linear semidefinite programming problems (see~\cite{paper_IP_PMM_SDP}), 
including cases when the associated Newton systems are solved iteratively. Finally, a similar algorithm was proposed for general nonlinear programming problems in \cite{ArmandOhmeniJOTA}, and was shown to be convergent under standard assumptions. In all these cases, algorithmic regularization ensures stability, while allowing one to retrieve the solution of the original problem.

\subsection{Interior Point-Proximal Method of Multipliers\label{sec:Interior Point-Proximal Method of Multipliers}}

\par In this subsection, we derive an IP-PMM suitable for solving convex programming problems. The method is based on the developments in \cite{paper_IP_PMM}. We consider the following primal problem (which can be equivalently formulated as \eqref{non-regularized primal}, by adding some additional constraints):
\begin{equation}  \label{primal problem} 
        \min_x \ f(x), \quad \text{s.t.} \ \ Ax = b, \ \ x^{\mathcal{I}} \geq 0, \ \ x^{\mathcal{F}}\ \text{free},
\end{equation}
\noindent where $\mathcal{I}  \subseteq \{1,\ldots,n\}$, $\mathcal{F} = \{1,\ldots,n\}\setminus\mathcal{I}$. In the above problem, we assume that the dimensions of the involved matrix are the same as those in \eqref{non-regularized primal}. Effectively, an IP-PMM arises by merging PMM with an infeasible IPM. For that purpose, assume that, at some iteration $k$ of the method, we have available an estimate $\eta_k$ for an optimal Lagrange multiplier vector $y^*$ associated to the equality constraints of \eqref{primal problem}.  Similarly, we denote by $\zeta_k$ an estimate of a primal solution $x^*$. Now, we define the proximal penalty function that has to be minimized at the $k$-th iteration of the PMM, for solving \eqref{primal problem}, given the estimates $\eta_k,\ \zeta_k$:
\begin{equation*}
    \mathcal{L}^{PMM}_{\rho_k,\delta_k} (x; \zeta_k,\eta_k) = f(x)-\eta_k^\top (Ax - b) + \frac{1}{2\delta_k}\|Ax-b\|_2^2 + \frac{\rho_k}{2}\|x-\zeta_k\|_2^2,
\end{equation*}
\noindent with $\{\delta_k\},\ \{\rho_k\}$ two positive non-increasing sequences of penalty parameters. Following~\cite{paper_IP_PMM}, we require that these parameters decrease at the same rate as $\mu_k$; however, in practice we never allow these values to be reduced below a certain appropriately chosen threshold. For more details on how to choose these constants for general problems, the reader is referred to~\cite{Pougk_Gondz_non_diag_reg}, where a perturbation analysis of regularization is performed. In order to solve the PMM subproblem, we apply one (or a few) iterations of an infeasible IPM. To that end, we alter the previous penalty function, by including logarithmic barriers, that is
\begin{equation} \label{Proximal IPM Penalty}
    \mathcal{L}^{IP-PMM}_{\rho_k,\delta_k} (x; \zeta_k,\eta_k) = \mathcal{L}^{PMM}_{\rho_k, \delta_k} (x; \zeta_k,\eta_k) - \mu_k \sum_{j \in \mathcal{I}}\ln x^j,
\end{equation}
\noindent where $\mu_k > 0$ is the barrier parameter. In order to form the optimality conditions of this subproblem, we equate the gradient of $\mathcal{L}^{IP-PMM}_{\rho_k, \delta_k}$ with respect to $x$ to the zero vector, i.e.,
\begin{equation*}
    \nabla f(x) - A^\top \eta_k + \frac{1}{\delta_k}A^\top(Ax - b) + \rho_k (x - \zeta_k) -\mathscr{P}^\top \begin{bmatrix} 0_{|\mathcal{F}|} \\ \mu_k (X^{\mathcal{I}})^{-1}e_{|\mathcal{I}|} \end{bmatrix} = 0_n,
\end{equation*}
\noindent where $\mathscr{P}$ is an appropriate permutation matrix, such that $\mathscr{P} x_k = [(x_k^{\mathcal{F}})^\top, (x_k^{\mathcal{I}})^\top]^\top$. Next, we define the variables $y = \eta_k - \frac{1}{\delta_k}(Ax - b)$ and $z \in \Re^n$, such that $z^{\mathcal{I}} = \mu_k (X^{\mathcal{I}})^{-1}e_{|\mathcal{I}|}$, $z^{\mathcal{F}} = 0$, to obtain the following (equivalent) system of equations:
\begin{equation*} \begin{bmatrix}
        \nabla f(x) - A^\top y - z + \rho_k (x-\zeta_k)\\
        Ax + \delta_k (y - \eta_k) - b\\
        X^{\mathcal{I}} z^{\mathcal{I}} -  \mu_k e_{|\mathcal{I}|}
    \end{bmatrix} = \begin{bmatrix}
        0_n\\
        0_m\\
        0_{|\mathcal{I}|}
    \end{bmatrix}.
\end{equation*}

\noindent To approximately solve the previous mildly nonlinear system of equations, at every iteration $k$, we employ a damped perturbed Newton method (that is, we alter its right-hand side using a centering parameter $\sigma_k \in  (0,1)$). In other words, at every iteration of IP-PMM we have available an iteration triple $(x_k,y_k,z_k)$ and we want to solve the following system of equations:
\begin{multline}\label{Newton system}
    \begin{bmatrix}
        -(\nabla^2 f(x_k)  + \rho_k I_n)  & A^\top & I_n\\
        A & \delta_k I_m  & 0_{m,n}\\
        Z_k & 0_{n,m} & X_k
    \end{bmatrix} \begin{bmatrix}
        \Delta x\\
        \Delta y\\
         \mathscr{P}^\top \begin{bmatrix} 0_{|\mathcal{F}|} \\ \Delta z^{\mathcal{I}} \end{bmatrix}
    \end{bmatrix} \\
    = \begin{bmatrix}
        \nabla f(x_k) - A^\top y_k + \sigma_k\rho_k (x_k - \zeta_k) -  z_k \\
        b - Ax_k - \sigma_k \delta_k (y_k - \eta_k)\\
       \mathscr{P}^\top \begin{bmatrix} 0_{|\mathcal{F}|} \\ \sigma_k \mu_k e_{|\mathcal{I}|} -X_k^{\mathcal{I}} z_k^{\mathcal{I}} \end{bmatrix}
    \end{bmatrix},
\end{multline}
Notice that all penalty parameters in the right-hand side are multiplied by $\sigma_k$. In essence $\sigma_k$ determines how fast (or slow) these parameters are going to decrease in the next IP-PMM iteration. Following a standard development of IPMs, $x_k^{\mathcal{I}}$ is maintained positive and hence the logarithmic barrier in~\eqref{Proximal IPM Penalty} is well defined while also $z_k^{\mathcal{I}}$ remains positive. From the third block-equation of \eqref{Newton system} we have $\Delta z^{\mathcal{F}} = 0$ and 
$$
    \Delta z^{\mathcal{I}} = (X_k^{\mathcal{I}})^{-1}(-Z_k^{\mathcal{I}} \Delta x^{\mathcal{I}} + \sigma_k \mu_k e_{|\mathcal{I}|} -X_k^{\mathcal{I}} z_k^{\mathcal{I}}).$$
\noindent In light of the previous computations, \eqref{Newton system} reduces to:
\begin{equation}\label{regularized augmented system}
    \begin{split}
        \begin{bmatrix}
            -(\nabla^2 f(x_k)+ \Xi_k + \rho_k I_n)  & A^\top\\
            A & \delta_k I_m 
        \end{bmatrix} \begin{bmatrix}
            \Delta x\\
            \Delta y
        \end{bmatrix} = \begin{bmatrix}
            r_{1,k}\\
            r_{2,k}
        \end{bmatrix},\end{split}
\end{equation}
\noindent where
\begin{equation}\label{augmented system residuals}
    \begin{bmatrix}
        r_{1,k}\\
        r_{2,k}
    \end{bmatrix} = \begin{bmatrix}
        \nabla f(x_k) - A^\top y_k + \sigma_k\rho_k (x_k - \zeta_k)- \mathscr{P}^\top \begin{bmatrix} 0_{|\mathcal{F}|}\\ \sigma_k \mu_k (X_k^{\mathcal{I}})^{-1} e_{|\mathcal{I}|} \end{bmatrix} \\
        b - Ax_k - \sigma_k\delta_k (y_k - \eta_k)
    \end{bmatrix},
\end{equation}
\noindent and
\begin{equation*}
    \Xi_k \coloneqq {\mathscr{P}^\top} \begin{bmatrix}  0_{|\mathcal{F}|,|\mathcal{F}|} & 0_{|\mathcal{I}|,|\mathcal{F}|}\\ 0_{|\mathcal{F}|,|\mathcal{I}|} & (X_k^{\mathcal{I}})^{-1}(Z_k^{\mathcal{I}}) \end{bmatrix} {\mathscr{P}}.
\end{equation*}
\noindent For the rest of this paper, we will make use of the notation $\Xi_k$ in cases where only a subset of the primal variables $x$ are constrained to be non-negative. In the case where all the entries of $x$ must satisfy this constraint, we will employ the standard IPM notation $\Theta_k \equiv \Xi_k^{-1}$, since in this case $\mathcal{F}  = \emptyset$. In the special case where $\nabla^2 f(x_k)$ is a diagonal (or zero) matrix, it could be beneficial to further reduce system \eqref{regularized augmented system}, by eliminating variables $\Delta x$. The resulting normal equations yield a positive definite system of equations that reads as follows:
\begin{equation} \label{regularized normal equations}
    \begin{array}{ll}
    \displaystyle \phantom{=} \left(A (\nabla^2 f(x_k) + \Xi_k + \rho_k I_n)^{-1} A^\top + \delta_k I_m \right) \Delta y \\[2mm]
    \displaystyle = r_{2,k} + A (\nabla^2 f(x_k) + {\Xi_k} + \rho_k I_n)^{-1}(r_{1,k}).
    \end{array}
\end{equation}
\noindent The parameters $\eta_k,\ \zeta_k$ are tuned as in \cite{paper_IP_PMM}. In particular, we set $\eta_0 = y_0$ and $\zeta_0 = x_0$, where $(x_0,y_0,z_0)$ is the starting point of IP-PMM. Then, at the end of every iteration $k$, we set $(\zeta_{k+1},\eta_{k+1}) = (x_{k+1},y_{k+1})$ only if the primal and dual residuals are decreased sufficiently. If the latter is not the case, we set $(\zeta_{k+1},\eta_{k+1}) =(\zeta_{k},\eta_{k}) $.

It has been demonstrated in \cite{paper_IP_PMM} that IP-PMM using a single Newton step per iteration converges to an $\epsilon$-optimal solution in a number of iterations that is polynomial with respect to the problem size $n$, if $f$ is a convex quadratic function. Furthermore, the latter holds for linear semidefinite programming problems, even if one solves the Newton system inexactly, i.e., requiring only the residual to be bounded by a suitable multiple of the barrier parameter $\mu_k$ (see \cite{paper_IP_PMM_SDP}). Nevertheless, the previous is not proven to hold for the general convex (nonlinear) case. In the latter case, one would have to employ Newton method combined with a \textit{line-search} or a \textit{trust-region} strategy (see, e.g., \cite{ArmandOhmeniJOTA,WaltzMoralNocedOrban}), in order to guarantee the convergence of the method.
In all the cases analyzed in this work we make use of a simple Mehrotra-type \cite{SIAMJOpt:Mehr} predictor-corrector scheme, 
which in general is sufficient to produce good directions that allow the method to converge quickly to the optimal solution. In the corrector stage, the right-hand side is approximated by a linearization of the function that is being minimized (see~\cite{SIAMJOpt:TapZhaSalWei}).

\subsection{Testing environment\label{sec:TestingEnv}}
The various specializations of IP-PMM discussed in the following sections have been implemented in MATLAB and compared with MATLAB implementations of state-of-the-art methods for each specific problem. All the tests were run with MATLAB R2019b on an Intel Xeon Platinum 8168 CPU with 192 GB of RAM, available from the \textit{magicbox} server at the Department of Mathematics and Physics of the University of Campania ``L.~Vanvitelli".


\section{Portfolio Selection Problem\label{sec:Portfolio}}

Portfolio selection is one of the most central topics in modern financial economics.
It deals with the decision problem of how to allocate resources among
several competing assets in accordance with the investment objectives.
For medium- and long-time horizons, the multi-period strategy is suitable, because it allows the change of the allocation
of the capital across the assets, taking into account the evolution of the available information.
In a multi-period setting, the investment period is partitioned into $m$ sub-periods, delimited by $m+1$ rebalancing dates $t^j$.
The decisions are taken at the beginning of each sub-period $[t^j, t^{j + 1})$,  $j=1,...,m$, and kept within it. 
The optimal portfolio is defined  by the vector
\[w = [ w_1^\top, w_2^\top, \dots, w_m^\top]^\top,\]
where $w_j \in \Re^s$ is the portfolio of holdings at time $t^j$ and $s$ is the number of assets.

The mean-variance formulation proposed by Markowitz \cite{Markowitz1959} was extended to a multi-period portfolio
selection by Li and Ng \cite{LiNg00}, and  in recent years there has been a significant advancement of
both theory and methodologies. In a multi-period mean-variance framework, we fix a final target expected return
and adopt as risk measure the function obtained by summing the single-period
variance terms \cite{ChenLiGuo13}:
\[ \rho(w) = \sum_{j=1}^{m}  w_j^\top C_j w_j, \] 
\noindent where $C_j \in \Re^{s\times s}$ is the covariance matrix, assumed to be positive definite,
estimated at~$t^j$.  
A common strategy to estimate Markowitz model parameters is to use historical data as predictive of the future behavior
of asset returns.  Different regularization techniques have been  proposed to deal with  ill-conditioning due to asset correlation; in the last years 
the $\ell_1$-regularization has been used to promote sparsity in the solution~\cite{corsaro2}. It allows investors to reduce the number of positions to
be monitored and held and the overall transaction costs.  Another useful interpretation of the $\ell_1$ norm is related to the
amount of short positions (i.e., negative components in the solution), which indicate an investment strategy where an investor is selling borrowed stocks in the open market, expecting that the market will drop, in order to realize a profit.
A suitable tuning of the regularization parameter permits short controlling in both the single- and the
multi-period case \cite{corsaro1, corsaro2}. However, in the multi-period case, the sparsity in the solution does not guarantee the control of the transaction 
costs, especially if the pattern of the active positions (i.e., positive components in the solution) completely changes across periods. In this case, sparsity must be introduced in the variation,
e.g., by adding an $\ell_1$ term involving the differences of the wealth values allocated on the assets between two contiguous rebalancing times.
This acts as a penalty on the portfolio turnover, which has the effect of reducing the number of transactions and hence the transaction
costs~\cite{corsaro3, DDV19}.

Thus, we consider the following fused lasso optimization problem for multi-period portfolio selection~\cite{DDV19}:
\begin{equation}\label{eq:FL}
    \begin{array}{rl}
        \displaystyle \min_{w} &  \frac{1}{2} w^\top C w+ \tau_1 \|w\|_1 + \tau_2  \|L w\|_1, \\
        \text{s.t.}    & w_1^\top e_s = \xi_{\mathrm{init}}, \\
        &w_j^\top e_s = (e_s+r_{j-1})^\top w_{j-1}, \quad j = 2,\ldots,m,\\
        &(e_s+r_{m})^\top w_m = \xi_{\mathrm{term}}, \\
    \end{array}
\end{equation}
where $n= m \, s$, $C = \diag(C_1, C_2, \ldots, C_m) \in \Re^{n \times n }$ is a block diagonal symmetric positive definite matrix,
$\tau_1, \tau_2>0$, $L \in \Re^{ (n-s)\times n}$ is the discrete difference operator representing the fused lasso regularizer, $r_j \in \Re^s$ is  the 
expected return vector at time $t^j$, $\xi_{\mathrm{init}}$ is the initial wealth, and $\xi_{\mathrm{term}}$
is the target expected wealth resulting from the overall investment. 
The first constraint is the initial budget constraint.
The strategy is assumed to be self-financing, as constraints from
$2$ to $m$ establish; this means that the value of the portfolio changes only because the asset prices change.
The $(m+1)$-st constraint defines the expected final wealth.
To deal with the non-separability of the objective function in \eqref{eq:FL}, we introduce an auxiliary variable $d$, which is constrained to be equal to $Lw$,
and we equivalently formulate problem~\eqref{eq:FL} as  follows:
\begin{equation}\label{eq:SFL} 
    \begin{array}{rl} 
         \displaystyle \min_{w, d} & \displaystyle \frac{1}{2} w^\top Cw+\tau_1 \|w\|_1+\tau_2\|d\|_1, \\
        \text{s.t.}& \bar A w = \bar{b},\\
        & Lw = d,
    \end{array}
\end{equation}
where the constraint matrix $\bar A \in  \Re^{(m+1) \times n}$ can be interpreted as an $(m+1) \times m$ 
lower bi-diagonal block matrix, with blocks of dimension $1 \times s$ defined by
\[ \bar A^{i,j}= \left \{
\begin{array}{cl}
    e_s^\top & \mbox{if } i=j, \\
    -(e_s +r_{i-1})^\top & \mbox{if } j=i+1, \\
    0_s^\top& \mbox{otherwise},
\end{array} \right. 
\]
and $\bar{b}=(\xi_{\mathrm{init}},0,0,...,\xi_{\mathrm{term}})^\top \in \Re^{m+1}$.


\subsection{Specialized IP-PMM for quadratic portfolio optimization problems}

Using the standard trick described in Section~\ref{sec:Intro}, we split $w$ and $d$ into two vectors  of the same size,
representing the non-negative and non-positive parts of the entries of 
$w$ and $d$ respectively, i.e., $w = w^+ - w^-$ and $d = d^+ - d^-$. Then, problem~\eqref{eq:SFL} is reformulated
as the following QP problem:
\begin{equation}\label{QPPort} 
        \min_{x} \ \frac{1}{2} x^\top Qx+ c^\top x, \quad \text{s.t.} \ \ Ax = b, \ \ x \geq 0
\end{equation}
where we set $l=n-s$, $\overline n = 2(n+l) = 2s(2m-1)$, $\overline m = m+1+l = (m+1)+s(m-1)$,
$$
x=[(w^+)^\top,(w^-)^\top,(d^+)^\top,(d^-)^\top]^\top \in \Re^{\overline n},
$$
\begin{equation}
    Q = \left[
    \begin{array}{cc}
      \begin{bmatrix}
         \phantom{-}C & -C \\ -C & \phantom{-}C
      \end{bmatrix} & 0_{2n,2l} \\
      0_{2l,2n} & 0_{2l,2l}\\
    \end{array} \right] \in \Re^{\overline n \times \overline n },
    \quad
    A = \left[
    \begin{array}{ccc}
        \bar{A} & -\bar{A} & 0_{(m+1),2l}\\
        L & -L & \begin{bmatrix}-I_l & I_l\end{bmatrix}\\
    \end{array} \right] \in \Re^{\overline m \times \overline n},
\end{equation}
$$
c = [\tau_1,\ldots,\tau_1,\tau_2,\ldots,\tau_2]^\top \in \Re^{\overline n}, \quad
b =[\bar{b}^1,\ldots,\bar{b}^{m+1},0,\ldots,0]^\top \in \Re^{\overline m}.
$$


\subsubsection{Dropping Primal Variables\label{sec:DroppingStrategy}}

The optimal solution of problem \eqref{eq:SFL} is expected to be sparse. On the other hand, in light of the reformulation given in~\eqref{QPPort}, we anticipate (and verify in practice) that most of the primal variables $w$ attain a zero value close to optimality. Such variables may significantly contribute to the ill conditioning of the matrix in \eqref{regularized augmented system} (see, e.g., \cite{iter:dAdSdS,IPM_25_years_later} and the references therein). In order to take advantage of this special property displayed by problem \eqref{QPPort}, we employ the following heuristic method, which aims at dropping variables $x^j$ which are sufficiently close zero, their seemingly optimal value. This results in better conditioning of the augmented system, whose dimension is also significantly reduced close to optimality, thus decreasing the computational cost of each IPM iteration. In other words, as IP-PMM progresses, we project the problem onto a smaller space.    
After this reduced problem is solved, its optimal solution is expanded back to the original space by filling all earlier eliminated variables with zeros. This delivers an optimal solution to the original problem.
\par In particular, we set a threshold value $\epsilon_\mathrm{drop} > 0$, and a large constant $\xi > 0$. At iteration $k = 0$, we define a set $\mathcal{V} = \emptyset$. Then, at every iteration $k$ of IP-PMM, we check the following condition, for every $j \in \mathcal{I}\setminus \mathcal{V}$:
\begin{equation}\label{portfolio regularized dropping condition}
    x_k^j \leq \epsilon_\mathrm{drop} \quad \text{and}\quad z_k^j \geq \xi \cdot \epsilon_\mathrm{drop} \quad \text{and}\quad (r_d)_k^j \leq \epsilon_\mathrm{drop},
\end{equation}
\noindent where $(r_d)_k^j = (c - A^\top y_k + Qx_k -z_k)^j$ represents the dual infeasibility corresponding to the $j$-th variable. Any variable that satisfies the latter condition is dropped, that is, we set $x_k^j = 0$, $\mathcal{V} = \mathcal{V} \cup \{j\}$, $\mathcal{G} = \mathcal{F}\cup (\mathcal{I}\setminus \mathcal{V})$, we drop $z_k^j$ and solve
\begin{equation}\label{portfolio regularized augmented system dropping}
    \begin{split}
        \begin{bmatrix}
            -(Q^{\mathcal{G},\mathcal{G}} + \Xi_k^{\mathcal{G},\mathcal{G}} + \rho_k I_{|\mathcal{G}|})  & (A^{\mathcal{H},\mathcal{G}})^\top\\
            A^{\mathcal{H},\mathcal{G}} & \delta_k I_{\overline{m}} 
        \end{bmatrix} \begin{bmatrix}
            {\Delta x^{\mathcal{G}}}\\
            \Delta y
        \end{bmatrix} = \begin{bmatrix}
            r_{1,k}^{\mathcal{G}}\\
            r_{2,k}
        \end{bmatrix},\end{split}
\end{equation}

\noindent where, $\Xi_k$ is defined as in Section \ref{sec:Interior Point-Proximal Method of Multipliers}, $r_{1,k},\ r_{2,k}$ are defined in \eqref{augmented system residuals} (by substituting $A^{\mathcal{H},\mathcal{G}}$ as the constraint matrix), and $\mathcal{H} = \{1, \ldots, \overline{m}\}$. We should note that this is a heuristic, since once a variable is dropped, it is not considered again until the method converges. Hence, one has to make sure that none of the nonzero variables $x_k^j$ is dropped. Nevertheless, at the end of the optimization process we can test whether the variables in $\mathcal{V}$ are indeed nonzero. More specifically, once an optimal solution $(x^*,y^*,z^*)$ is found, we compute:
$$z^{\mathcal{V}} = c^{\mathcal{V}} - (A^{\mathcal{H},\mathcal{V}})^\top y^* + Q^{\mathcal{V},\mathcal{G}}(x^*)^{\mathcal{G}}.$$
\noindent If there exists $j$ such that $(z^{\mathcal{V}})^j \leq 0$, then we would identify that a variable $x^j$ was incorrectly dropped. Otherwise, the optimal solution of the reduced problem coincides with the nonzero part of the optimal solution of \eqref{QPPort}. We notice that this methodology is not new. In particular, a similar strategy was employed in \cite{paper_Drop_Vars}, where a special class of linear programming problems was solved using a primal-dual logarithmic barrier method.


\subsection{Computational Experience}

We test the effectiveness of the IP-PMM applied to the fused lasso model on the following real-world data sets:
\begin{enumerate}
    \item FF48-FF100 (Fama \& French 48-100 Industry portfolios, USA), containing 48-100 portfolios considered as assets, from July 1926 to December 2015. 
    \item ES50 (EURO STOXX 50), containing 50 stocks from 9 Eurozone countries (Belgium, Finland, France, Germany, Ireland, Italy,
    Luxembourg, the Netherlands and Spain), from January~2008 to December~2013.
    \item FTSE100 (Financial Times Stock Exchange, UK), containing 100 assets, from July~2002 to April~2016.
    \item SP500 (Standard \& Poors, USA), containing 500 assets,  from January~2008 to December~2016.
    \item NASDAQC (National Association of Securities Dealers Automated
    Quotation Composite, USA), containing almost all stocks listed on the Nasdaq stock market, from February~2003 to April~2016.
\end{enumerate}
Following \cite{corsaro3,DesimoneAMC}, we generate 10 problems with annual or quarterly rebalancing, after a preprocessing
procedure that eliminates the elements with the highest volatilities. A rolling window (RW) for setting up the model parameters is considered.
For each dataset, the length of the RWs is fixed in order to build positive definite covariance matrices and ensure statistical significance.
Different datasets require different lengths for the RWs. In Table~\ref{tab:data} we summarize the information on the test problems.

\begin{table}[htbp]
    \begin{small}
        \caption{Characteristics of the portfolio test problems (y = years, m = months)}\label{tab:data} 
        \centerline{
            \begin{tabular}{l|rrrr}
                \toprule
                \textbf{Problem} & \textbf{Assets}& \textbf{RW} &\textbf{Sub-periods}  & $\overline n$ \\ \midrule
                FF48-10 & 48  & 5 y & 10 y & 1632 \\
                FF48-20 & 48  & 5 y & 20 y & 3552\\
                FF48-30 & 48  & 5 y &  30 y & 5472\\
                FF100-10 &   96& 10 y& 10 y & 3264 \\
                FF100-20 &  96 & 10 y& 20 y & 7104\\
                FF100-30 &   96 & 10 y &30 y & 10,944\\
                ES50 &  50   & 1 y & 22 m & 4300\\
                FTSE100 &  83 & 1 y& 10 y & 3154\\
                SP500&  431 & 2 y& 8 y & 11,206 \\
                NASDAQC& 1203 & 1 y & 10 y & 45,714\\ \bottomrule
            \end{tabular}
        }
    \end{small}
\end{table}

We introduce some measures to evaluate the goodness of the optimal portfolios versus the benchmark one, in terms of risk, sparsity and transaction costs.
We consider as benchmark the multi-period naive portfolio, based on the strategy for which at each rebalancing date the total wealth is equally divided among the assets. We assume that the investor has one unit of wealth at the beginning of the planning horizon, 
i.e., $\xi_{\mathrm{init}}=1$, and we set as expected final wealth the one provided by the benchmark. As in~\cite{corsaro3},
we define:
\begin{equation}\label{ratio}
    ratio = \frac{w^\top_{naive}Cw_{naive}}{w^\top_{opt}Cw_{opt}},
\end{equation}
where $w_{naive}$ and $w_{opt}$ are respectively the naive portfolio and the optimal one.  
This value measures the risk reduction factor with respect to the benchmark.
We consider the number of active positions as a measure of holding costs; then the value
\begin{equation}\label{ratio_h}
    ratio_{h} = \frac{\#  \mbox{ active positions of } w_{naive}}{ \#  \mbox{ active positions of }   w_{opt}}\end{equation}
measures the reduction factor of the holding costs with respect to the benchmark.
Finally, we consider the number  of variations in the weights as a measure of transaction costs.
More precisely, if $w_j^i \neq w_{j+1}^i$ we assume that security $i$ has been bought or sold
in the period $[t^j,t^{j+1})$. Then 
we estimate the number of transactions  as:
$${\mathcal T} = trace(V^\top V),$$
where  $V\in \Re^{s\times(m-1)}$, with
$$v^{ij} = \left\{ 
\begin{array}{ll}
    1 & \mbox{if } | w_{j}^i - w_{j+1}^i |\ge \epsilon,\\
    0 & \mbox{otherwise}.\\
\end{array}\right.$$
and $ \epsilon > 0$, in order to make sense in financial terms. 
A measure of the transaction reduction factor with respect to the benchmark is given by
\begin{equation}\label{ratio_t} ratio_{t} = \frac{ {\mathcal T}_{naive}}{ {\mathcal T}_{opt}}.
\end{equation}

\par We consider a version of the presented IP-PMM algorithm in which the solution of problem~\eqref{portfolio regularized augmented system dropping} is computed by means of factorization, the parameter $\epsilon_\mathrm{drop}$ controlling the heuristic described in Section \ref{sec:DroppingStrategy} is set to $10^{-4}$, and  the constant $\xi$, which is used to ensure that the respective dual slack variable is bounded away from zero, is set to $10^{2}$. We compare IP-PMM with the Split Bregman method, which is known to be very efficient for this kind of problems. In detail, we consider the
Alternating Split Bregman algorithm used in~\cite{DesimoneAMC}, based on a further reformulation of problem \eqref{eq:SFL} as
$$
\begin{array}{rl} 
    \displaystyle \min_{w,u,d} & \displaystyle \frac{1}{2} w^\top Cw+\tau_1 \|u\|_1+\tau_2\|d\|_1, \\
    \text{s.t.}& \bar A w = \bar b,\\
    & Lw = d,\\
    & w = u.
\end{array}
$$
This algorithm splits the minimization in three parts. Given $w_k, u_k, d_k$, the $(k + 1)$-st iteration 
consists in the minimization of a quadratic function to determine $w_{k+1}$ and the application of
the soft-thresholding operator
$$
[{\mathcal S}(v,\gamma)]^i= \sign(v^i)\cdot\max(|v^i|-\gamma, 0),
$$
where $v$ is a real vector and $\gamma >0$,
to determine $u_{k+1}$ and $d_{k+1}$.
%
%
The optimal value $w_{k+1}$ can be obtained by solving the system $H w = p_{k+1}$, with
\begin{equation}\label{eq:MatH}
    H = C+\lambda_1 {\bar{A}^\top \bar{A}}+\lambda_2 L^\top L+ \lambda_3 I,
\end{equation}
where $\lambda_1, \lambda_2, \lambda_3>0$ are fixed and $p_{k+1}$ depends on the iteration.
Since $H$ is independent of the iteration and is symmetric positive definite, sparse, and banded, in~\cite{DesimoneAMC} the authors compute
its sparse Cholesky factorization only once and solve two triangular systems at each iteration. We refer to this algorithm as ASB-Chol.

\begin{table}[!ht]
\centering
\caption{IP-PMM vs ASB-Chol} \label{Portfolio-Table}
{\small
    \begin{tabular}{l|r|r|r|r|r}
        \toprule
        &                            \multicolumn{5}{c}{\textbf{IP-PMM}}                            \\ \midrule
        \textbf{Problem} & \textbf{Time (s)} & \textbf{Iters} & \textbf{$ratio$} & \textbf{$ratio_h$} & \textbf{$ratio_t$} \\ \midrule
        FF48-10          &     1.37e$-$1 &           12 &        2.32e$+$0 &          6.67e$+$0 &          1.66e$+$1 \\
        FF48-20          &     3.77e$-$1 &           16 &        2.28e$+$0 &          6.58e$+$0 &          2.13e$+$1 \\
        FF48-30          &     8.43e$-$1 &           21 &        4.64e$+$0 &          6.15e$+$0 &          1.69e$+$1 \\
        FF100-10         &     4.92e$-$1 &           12 &        1.58e$+$0 &          1.78e$+$1 &          4.36e$+$1 \\
        FF100-20         &     1.63e$+$0 &           15 &        1.81e$+$0 &          2.04e$+$1 &          4.92e$+$1 \\
        FF100-30         &     3.93e$+$0 &           21 &        5.82e$+$0 &          1.34e$+$1 &          3.60e$+$1 \\
        ES50             &     4.59e$-$1 &           14 &        2.12e$+$0 &          4.42e$+$0 &          5.75e$+$1 \\
        FTSE100          &     4.64e$-$1 &           14 &        1.85e$+$0 &          5.37e$+$1 &          6.09e$+$1 \\
        SP500            &     3.43e$+$1 &           16 &        1.57e$+$0 &          8.62e$+$1 &          1.50e$+$2 \\
        NASDAQC          &     7.05e$+$2 &           20 &        3.15e$+$0 &          2.73e$+$0 &          3.89e$+$2 \\ \toprule
        &                           \multicolumn{5}{c}{\textbf{ASB-Chol}}                           \\ \midrule
        \textbf{Problem} & \textbf{Time (s)} & \textbf{Iters} & \textbf{$ratio$} & \textbf{$ratio_h$} & \textbf{$ratio_t$} \\ \midrule
        FF48-10          &     1.67e$-$1 &         1431 &        2.33e$+$0 &          6.67e$+$0 &          1.66e$+$1 \\
        FF48-20          &     3.72e$-$1 &         1985 &        2.31e$+$0 &          7.93e$+$0 &          2.09e$+$1 \\
        FF48-30          &     1.12e$+$0 &         4125 &        4.64e$+$0 &          6.08e$+$0 &          1.66e$+$1 \\
        FF100-10         &     8.49e$-$1 &         3087 &        1.58e$+$0 &          1.78e$+$1 &          4.36e$+$1 \\
        FF100-20         &     2.09e$+$0 &         3635 &        1.80e$+$0 &          1.78e$+$1 &          4.27e$+$1 \\
        FF100-30         &     8.54e$+$0 &         9043 &        5.83e$+$0 &          1.12e$+$1 &          2.97e$+$1 \\
        ES50             &     9.70e$-$1 &         4297 &        2.05e$+$0 &          2.94e$+$0 &          4.26e$+$1 \\
        FTSE100          &     4.29e$-$1 &         1749 &        1.80e$+$0 &          5.07e$+$1 &          5.71e$+$1 \\
        SP500            &     1.98e$+$1 &         3728 &        1.74e$+$0 &          6.16e$+$1 &          1.01e$+$2 \\
        NASDAQC          &     8.84e$+$2 &        14264 &        3.15e$+$0 &          2.73e$+$0 &          3.89e$+$2 \\ \bottomrule
    \end{tabular}
}
\end{table}

The values of $\tau_1$ and $\tau_2$ in~\eqref{eq:FL} are selected to guarantee
reasonable portfolios in terms of short positions. We recall that from the financial point of view,
negative solutions correspond to transactions in which an investor sells borrowed
securities in anticipation of a price decline. In our runs we consider the smallest values of $\tau_1$ and $\tau_2$
that produce at most $2 \%$ of short positions. We set $\tau_1=\tau_2=10^{-2}$ for the FF48 and FF100 data sets, 
$\tau_1=\tau_2=10^{-3}$ for ES50 and SP500, $\tau_1=10^{-2}$ and $\tau_2=10^{-3}$ for FTSE, and 
$\tau_1=10^{-2}$ and $\tau_2=10^{-4}$ for NASDAQC.

In Table~\ref{Portfolio-Table} we present the results obtained with IP-PMM and ASB-Chol on the test problems.
The termination criteria of IP-PMM are the same as in~\cite{paper_IP_PMM}, i.e., based on the relative reduction of the primal infeasibility $||Ax - b||$ (i.e. the constraints violation), the dual infeasibility $||\nabla f(x) - A^\top y - z||$, as well as complementarity (which is controlled by $\mu$). The stopping criterion for ASB-Chol is based only on the relative reduction of the primal feasibility {$||\bar{A}x - b||$}, which is a standard choice in literature.  The relative tolerance for the two algorithms is $tol = 10^{-6}$, which guarantees that the
values of $ratio$ differ by at most $10\%$, so that both algorithms produce comparable portfolios in terms of risk. We note that
the solution computed by ASB-Chol is thresholded by setting to zero all the entries with absolute value not exceeding the same value
of $\epsilon_\mathrm{drop}$ used in the IP-PMM dropping strategy. The results show that the optimal portfolios computed
by IP-PMM and ASB-Chol outperform the benchmark ones in terms of all the metrics. 
Concerning $ratio_h$ and $ratio_t$, IP-PMM is generally able to produce greater values than ASB-Chol, which indicates a higher sparsity in the solution found by IP-PMM. IP-PMM generally performs comparably or better than ASB-Chol in terms of elapsed time. Although ASB-Chol is faster than IP-PMM on SP500 by $14.5$ seconds ($42\%$), IP-PMM is able to reach a better solution in terms of sparsity.
When applied to FF100-30 IP-PMM produces a portfolio associated with lower transaction costs and takes less than half of the time required by ABS-Chol.
When applied to NASDAQC, which is the largest problem under consideration, the two methods reach comparable solutions in terms of all the evaluation metrics,
but IP-PMM needs about $20\%$ less time ($179$ seconds) than ASB-Chol. This suggests that the use of IP-PMM can be beneficial especially when solving high-dimensional problems.

\section{Classification models for functional Magnetic Resonance Imaging data\label{sec:fMRI}}

The functional Magnetic Resonance Imaging (fMRI) technique measures brain 
spatio-temporal activity via Blood-Oxygen-Level-Dependent (BOLD) signals. Starting from the assumption
that neuronal activity is coupled with cerebral blood flow, fMRI signals have been used to identify regions
associated with functions such as speaking, vision, movement, etc.. By analyzing the different oxigenation
levels in specific areas of the brain of healthy and ill patients, in the last decades fMRI signals have
been used to investigate the effect on the brain functionality of tumors, strokes, head and brain injuries
and of cognitive disorders such as schizophrenia or Alzheimer's (see \cite{dohmatob:2014,kamitani:2005,michel:2011}
and the references therein).

In an fMRI scan, voxels representing regions of the brain of a patient are recorded at different
time instances. The temporal resolution is usually in the order of a few seconds,
while the spatial resolution generally ranges from 4-5~mm (for some full brain analyses) to 1~mm
(for analyses on specific brain regions), which may amount up to about a million voxels.
Since fMRI experiments are conducted over groups of patients, the dimensionality of the data is
further increased. Therefore, the interpretation of fMRI results requires the analysis of huge
quantities of data. To this aim, machine learning techniques are being increasingly used
in recent years, because of their capability of dealing with massive amounts of data, incorporating
also a-priori information about the problems they are targeted to
\cite{baldassarre:2012,baldassarre:2017,dohmatob:2014,dubois:2014,gramfort:2013,michel:2011,rosa:2015}.

Here we focus on the problem of training a binary linear classifier to distinguish between
different classes of patients (e.g., ill/healthy) or different kinds of stimuli (e.g., pleasant/unpleasant),
and to get information about the most significant brain areas associated with the related neural activity.
The two classes are identified by the labels $-1$ and $1$. We assume that the training set consists of
$s_{-1}$ 3-dimensional (3d) scans in class $-1$ and $s_{1} $ 3d scans in class $1$, where each 3d scan is reshaped as a
row vector of size $q = q_1\times q_2\times q_3 $, and $q_i$ is the number of voxels along
the $i$-th coordinate direction of the domain covering the brain. All the scans are stored as rows
of a matrix $D \in \R^{s \times q}$, where $s = s_{-1} + s_{1}$.

\par We use a square loss function with the aim of determining
an unbiased hyperplane in $\Re^q$ that can separate the patients in the two classes.
This leads to a minimization problem of the form:
\begin{equation} \label{eq:fMRIProb}
    \min \, \frac{1}{2s} \left\| Dw - \hat{y} \right\|^2,
\end{equation}
where $\hat{y}$ is a vector containing the labels associated with each scan. {Notice that the use of the Euclidean loss is a standard practice in the literature for the classification of fMRI data (see, e.g., \cite{baldassarre:2017,gramfort:2013,Grosenicketal:2011,jieetal:2016,lietal:2020}). Nevertheless, it should be observed that different loss functions could be employed as well (e.g. see \cite{Ryalietal:2010,yamashita:2008}), potentially leading to better classification accuracy in certain cases.}
\par Since the number of patients is usually much smaller than the size of a scan, i.e., $s \ll q$,
problem~\eqref{eq:fMRIProb} is strongly ill posed and thus requires regularization.
Recently, significant attention has been given to regularization terms encouraging the presence of
structured sparsity, where smoothly varying nonzero coefficients of the solution are
associated with small contiguous regions of the brain. This is motivated by the possibility of obtaining
more interpretable solutions than those corresponding to other regularizers that do not promote
sparsity or lead to sparse solutions without any structure {(see~\cite{baldassarre:2017,gramfort:2013,lietal:2020} and
the references therein).} 

Structured sparsity can be promoted, e.g., by using a combination of $\ell_1$ and anisotropic
Total Variation (TV) terms~\cite{argyriou:2013}, which can be regarded as a fused lasso
regularizer \cite{tibshirani:2005}. The regularized problem reads
\begin{equation} \label{eq:fMRIRegProb}
    \min_{w} \ \frac{1}{2s} \left\| Dw - \hat{y} \right\|^2 + \tau_1\, \| w \|_1 + \tau_2\, \| L w \|_1,
\end{equation}
where $ \| L w \|_1 $ is the discrete anisotropic TV of $w$, i.e., $L  = [L_x^\top \;\, L_y^\top \;\, L_z^\top]^\top \in \Re^{l \times q}$
is the matrix representing first-order forward finite differences in the $x, y, z$-directions at each voxel.
By penalizing the difference between each voxel and its neighbors in each direction, one enforces
the weights of the classification hyperplane (which share the 3d structure of the scans) to assume similar
values for contiguous regions of the brain, thus leading to identify whole regions of the brain involved in the decision process.

\par The previous problem can be reformulated by introducing the variables $u = Dw$ and $d = Lw$ and applying the splitting
$$
   w = w^+ - w^-,\ d = d^+-d^-,\qquad  (w^+,w^-,d^+,d^-) \geq (0_q,0_q,0_l,0_l).
$$ 
\noindent Let $m = l+s$ and $n = s+ 2 q + 2 l$. Using the previous variables, \eqref{eq:fMRIRegProb} can be equivalently written as:
\begin{equation}  \label{fMRI_primal_IP_PMM_format}
    \begin{split}
        \min_x & \ \ \frac{1}{2}x^\top Qx + \ c^\top x, \\
        \text{s.t.} & \ \ Ax = b,\\
        &\ x_{\mathcal{I}} \geq 0,\ x_{\mathcal{F}}\ \text{free},\ \mathcal{I} = \{s+1,\ldots,n\},\ \mathcal{F} = \{1,\ldots,s\},
    \end{split}
\end{equation}
\noindent where $b = 0_{s+l} \in \mathbb{R}^m$,
$$
   x = [u^\top,(w^+)^\top,(w^-)^\top,(d^+)^\top,(d^-)^\top]^\top, \quad
   c = [-\frac{\hat{y}^\top}{s} ,\tau_1 e_w^\top,\tau_1 e_w^\top,\tau_2 e_d^\top, \tau_2 e_d^\top]^\top \in \mathbb{R}^n,
$$
and
\begin{equation} \label{fMRI_matrices_A_Q}
    \begin{split}
        Q = \begin{bmatrix}
            \frac{1}{s}I_s &0_{s,(n-s)}\\
            0_{(n-s),s} & 0_{(n-s),(n-s)}
        \end{bmatrix} \in \mathbb{R}^{n \times n},\quad
        A = \begin{bmatrix}
            -I_s & D & -D & 0_{s,l} & 0_{s,l}\\
            0_{l,s} & L & -L & -I_l & I_l
        \end{bmatrix} \in \mathbb{R}^{m\times n}.
    \end{split}
\end{equation}


\subsection{Specialized IP-PMM for Fused Lasso Least Squares}

\noindent Notice that problem \eqref{fMRI_primal_IP_PMM_format} is in the same form as \eqref{primal problem}. In what follows, we present a specialized inexact IP-PMM, suitable for solving unconstrained fused lasso least squares problems. The proposed specialized IP-PMM is characterized by the two following implementation details. 
Firstly, instead of factorizing system \eqref{regularized augmented system}, we employ an iterative method (namely, the Preconditioned Conjugate Gradient (PCG) method \cite{iter:Hestenes_Stiefel}) to solve system \eqref{regularized normal equations}. Secondly, as suggested in Section~\ref{sec:DroppingStrategy}, we take advantage of the fact that the optimal solution of problem \eqref{fMRI_primal_IP_PMM_format} is expected to be sparse, and use the heuristic approach that allows us to drop many of the variables of the problem, when the method is close to the optimal solution. 


\subsubsection{Solving the Newton System}

\noindent We focus on solving the normal equations in \eqref{regularized normal equations}. Let $k$ denote an arbitrary iteration of IP-PMM. We re-write the matrix in \eqref{regularized normal equations} without using the succinct notation introduced earlier:
\begin{equation} \label{fMRI normal equations matrix - succinct}
    M_k = \begin{bmatrix}
        M_{1,k} & M_{2,k}^\top\\
        M_{2,k} & M_{3,k}
    \end{bmatrix},
\end{equation}
\noindent where:
\begin{equation} \label{fMRI normal equations matrix parts}
    \begin{split}
        M_{1,k} = &\ \left((\tfrac{1}{s} + \rho_k)^{-1} + \delta_k\right) I_s + D\left((\Xi_{w^+,k}+\rho_k I_q)^{-1}+(\Xi_{w^-,k}+\rho_k I_q)^{-1}\right)D^\top, \\
        M_{2,k} = &\ L(\Xi_{w^+}+\rho_k I_q)^{-1} D^\top + L(\Xi_{w^-,k}+\rho_k I_q)^{-1}D^\top, \\
        M_{3,k} = &\ L\left((\Xi_{w^+,k}+\rho_k I_q)^{-1}+ (\Xi_{w^-,k}+\rho_k I_q)^{-1}\right)L^\top \\ 
        & + \left((\Xi_{d^+,k}+\rho_k I_l)^{-1} + (\Xi_{d^-,k}+\rho_k I_l)^{-1} + \delta_k I_l\right),
    \end{split}
\end{equation}
\noindent while 
$$\Xi_k^{\mathcal{I}} = \begin{bmatrix}
    \Xi_{w^+,k} & 0_{q,q} & 0_{q,l} & 0_{q,l}\\
    0_{q,q}           & \Xi_{w^-,k} & 0_{q,l} & 0_{q,l}\\
    0_{l,q}         & 0_{l,q} & \Xi_{d^+,k} & 0_{l,l} \\
    0_{l,q}        &  0_{l,q} & 0_{l,l} & \Xi_{d^-,k}
\end{bmatrix}, $$
\noindent and $\Xi_k$ is defined as in Section \ref{sec:Interior Point-Proximal Method of Multipliers}.
\par Notice that the matrix $D$ in \eqref{eq:fMRIProb} is dense, and hence we expect $M_{1,k}$ and $M_{2,k}$ in \eqref{fMRI normal equations matrix parts} to also be dense. On the other hand, $M_{3,k}$ remains sparse, and we know that $l \gg s$. As a consequence, the Cholesky factors of the matrix in \eqref{fMRI normal equations matrix - succinct} would 
inevitably contain dense blocks. Hence, it might be prohibitively expensive to compute such a decomposition. Instead, we solve the previous system using a PCG method. In order to do so efficiently, we must find an approximation for the coefficient matrix in \eqref{fMRI normal equations matrix - succinct}. Given the fact that $M_{3,k}$ is sparse, while $M_{1,k}$ and $M_{2,k}$ are dense, we would like to find an approximation for the dense blocks. A possible approach would be to approximate $D$ by a low-rank matrix. Instead, based on the assumption $l \gg s$, we can approximate $M_k$ by the following block-diagonal preconditioner:
\begin{equation} \label{fMRI normal equations preconditioner}
    P_k = \begin{bmatrix}
        M_{1,k} & 0_{s,l}\\
        0_{l,s} & M_{3,k}
    \end{bmatrix},\ \text{where},\ P_k^{-1} = \begin{bmatrix}
        M_{1,k}^{-1} & 0_{s,l} \\
        0_{l,s} & M_{3,k}^{-1}
    \end{bmatrix}.
\end{equation}
\par {We observe that $M_{3,k}$ is a matrix of the form $LRL^\top+B$, where $R$ and $B$ are positive definite diagonal matrices and $L$ comes from stacking three first-order forward finite-difference operators. It is easy to check that $M_{3,k}$ has a $3\times 3$ block structure in which the diagonal blocks are Symmetric Diagonally Dominant M-matrices (SDDM). One could hence build a diagonal preconditioner by exploiting specialized strategies recently developed for this class of matrices \cite{chen2021rchol,KyngSachIEEE,PengSpielman2014}. 
Nevertheless, we notice that $M_{3,k}$ does have a sparse Cholesky factor, due to the sparsity displayed in the discrete anisotropic TV matrix $L$, which in our experiments guaranteed good performance.} On the other hand, the Cholesky factor of $M_{1,k}$ is dense. However, computation and storage of this dense factor is possible, as we only need to perform $O(s^3)$ operations, and store $O(s^2)$ elements.

\paragraph{\normalfont \textbf{Spectral Analysis}\label{fMRI - Spectral Analysis}}

\noindent Let us now further support the choice of the preconditioner in \eqref{fMRI normal equations preconditioner} by performing a spectral analysis of the preconditioned system $R_k = P_k^{-1}M_k$. In the following, we write $\mathcal{A} \times \mathcal{B}$ to denote a vector space whose elements are vectors $[a^\top, b^\top]^\top$ with $a \in \mathcal{A}$ and $b \in \mathcal{B}$.
\begin{theorem} \label{fMRI spectral analysis theorem}
    Let $D \in \Re^{s \times q}$ be the matrix in \eqref{eq:fMRIProb}. 
    Let also $M_k  \in \Re^{(s+l) \times (s+l)}$ be the matrix defined in \eqref{fMRI normal equations matrix - succinct} and $P_k$ the preconditioner defined in \eqref{fMRI normal equations preconditioner}. Then, the preconditioned matrix $R_k = P_k^{-1}M_k$ has $l-\rank(D)$ eigenvalues $\lambda = 1$, whose respective eigenvectors form a basis for $\{0_s\}\times\{\text{Null}(M_{2,k}^\top)\}$. All the remaining eigenvalues of the preconditioned matrix satisfy $\lambda \in (\chi,1)\cup(1,2)$, where $\chi = \frac{\delta_k \rho_k }{\sigma_{\max}^2(A)+ \rho_k \delta_k}$, $\delta_k$, $\rho_k$ are the regularization parameters of IP-PMM and $A$ is defined in \eqref{fMRI_matrices_A_Q}.
\end{theorem}
\begin{proof} Let us consider the following generalized eigenproblem:
    $$ M_k p = \lambda P_k p.$$
    \noindent We partition the eigenvector $p$ as $p = [p_s^\top,p_l^\top]^\top$. Using \eqref{fMRI normal equations matrix - succinct}, and \eqref{fMRI normal equations preconditioner}, the eigenproblem can be written as:
    \begin{equation} \label{fMRI spectral analysis - eigenproblem}
        \begin{split}
            p_s + M_{1,k}^{-1}M_{2,k}^\top p_l = \lambda p_s\\
            M_{3,k}^{-1}M_{2,k} p_s + p_l = \lambda p_l.
        \end{split}
    \end{equation}
    \noindent From \eqref{fMRI normal equations matrix parts} and \eqref{fMRI normal equations preconditioner} we know that $M_k \succ 0$ and $P_k \succ 0$ and hence $\lambda > 0$. Let us separate the analysis in two cases.
    \par \textbf{Case 1:} $\mathbf{\lambda = 1}.$
    \noindent This is the case for every $p$ such that
    $$ p \in \mathcal{S}_k = \{0_s\}\times\{\text{Null}(M_{2,k}^\top)\}, $$
    \noindent which trivially satisfies \eqref{fMRI spectral analysis - eigenproblem} for $\lambda = 1$. Upon noticing that
    $$\text{dim}(\mathcal{S}_k) = \text{dim}(\text{Null}(M_{2,k}^\top)) = l -\rank(M_{2,k}^\top),$$
    \noindent we can conclude that $R_k = P_k^{-1}M_k$ has an eigenvalue $\lambda = 1$ of multiplicity $l-\rank(M_{2,k}^\top)$. The respective eigenvectors form a basis of $\mathcal{S}_k$. Notice also, from \eqref{fMRI normal equations matrix parts}, that $\rank(M_{2,k}^\top) = \rank(D) \leq s$.
    \par \textbf{Case 2:} $\mathbf{\lambda \neq 1.}$
    \noindent There are exactly $s+\rank(D)$ such eigenvalues. In order to analyze this case, we have to consider two generalized eigenproblems. On the one hand, from the first block equation in \eqref{fMRI spectral analysis - eigenproblem}, we have that:
    $$ p_s = \frac{1}{\lambda-1}M_{1,k}^{-1} M_{2,k}^\top p_l,$$
    \noindent and hence, substituting this in the second block equation of \eqref{fMRI spectral analysis - eigenproblem} gives the following generalized eigenproblem:
    \begin{equation} \label{lemma spectral analysis: equation for G_l}
        G_{l,k} p_l = \nu M_{3,k} p_l,
    \end{equation}
    \noindent where $ G_{l,k} = M_{2,k} M_{1,k}^{-1}M_{2,k}^\top$ and $\nu = (\lambda-1)^2$. By assumption $\lambda \neq 1$, and hence $\lambda = \pm\sqrt{\nu} + 1$. However, we have $M_k \succ 0$ and hence $M_{3,k} - M_{2,k} M_{1,k}^{-1}M_{2,k}^\top \succ 0$. Let us assume that the maximum eigenvalue of $M_{3,k}^{-1}G_{l,k}$ is greater than or equal to $1$, i.e., $\nu_{\max} \geq 1$. By substituting $\nu_{\max}$ in \eqref{lemma spectral analysis: equation for G_l} and multiplying both sides of the previous inequality by $p_l^\top$, we get
    $$ p_l^\top G_{l,k} \, p_l = \nu_{\max} \, p_l^\top M_{3,k} \, p_l $$
 \noindent and hence
    $$ p_l^\top (G_{l,k} - M_{3,k}) \, p_l \geq 0. $$  
    \noindent The latter contradicts the fact that $M_{3,k} - G_{l,k} \succ 0$, and thus $\nu_{\max} < 1$. In other words, $\lambda \in (0,1)\cup(1,2)$.   
    \par Similarly, starting from the second block equation of \eqref{fMRI spectral analysis - eigenproblem}, we get
    $$ p_l = \frac{1}{\lambda-1}M_{3,k}^{-1}M_{2,k} p_s, $$
    \noindent and substituting this in the first block equation in \eqref{fMRI spectral analysis - eigenproblem} yields
    $$ G_{s,k} p_s = \nu M_{1,k} p_s, $$
    \noindent where $G_{s,k} = M_{2,k}^\top M_{3,k}^{-1}M_{2,k}$ and $\nu = (\lambda - 1)^2$. As before, we have that $\lambda = \pm\sqrt{\nu} + 1$. Using the fact that $M_{1,k} - M_{2,k}^\top M_{3,k}^{-1} M_{2,k} \succ 0$, we can mirror the previous analysis to conclude that $\nu_{\max} < 1$, and hence $\lambda \in (0,1)\cup(1,2)$. 
    \par Finally, notice that as long as the primal and dual regularization parameters of IP-PMM, $\rho_k$ and $\delta_k$ respectively, are bounded away from zero, so are the eigenvalues of $R_k = P_k^{-1}M_k$, for every iteration $k$ of the algorithm. In particular, we have that
    $$ \lambda_{\min}(R_k) \geq \frac{\lambda_{\min}(M_k)}{\lambda_{\max}(P_k)} \geq \frac{\delta_k \rho_k }{\sigma_{\max}^2(A)+ \rho_k \delta_k} = \chi,$$
    \noindent where we used the fact that $\lambda_{\min}(M_k) \geq \delta_k$ and $\lambda_{\max}(P_k) \leq \frac{\sigma^2_{\max}(A)}{\rho_k} + \delta_k$, where $A$ is defined in \eqref{fMRI_matrices_A_Q}. 
\end{proof}


\subsubsection{Dropping Primal Variables\label{sec:dropping-fmri}}

The preconditioner \eqref{fMRI normal equations preconditioner} may be computed 
(and applied) very efficiently as we expect the Cholesky factor of $M_{3,k}$ 
to preserve sparsity and $M_{1,k} \in \Re^{s \times s}$ to be relatively 
small (recall that $s \ll l$). 
We deduce from Theorem~\ref{fMRI spectral analysis theorem} that 
the preconditioner defined in \eqref{fMRI normal equations preconditioner} 
remains effective as long as the regularization parameters $\rho_k$ and $\delta_k$ 
are not too small. However, to attain convergence of IP-PMM $\rho_k$ and $\delta_k$ 
have to be reduced and then, due to the nature of IPMs, the matrix 
in \eqref{regularized normal equations} becomes increasingly ill conditioned 
as the method approaches the optimal solution. 
This implies that the preconditioner defined in \eqref{fMRI normal equations preconditioner} has only a limited applicability. 
In particular, this means that there is a limited scope for refining it 
and we may not be able to prevent degrading behaviour of PCG when IPM 
gets very close to the optimal solution. 


\par However, we notice that the optimal solution of problem \eqref{eq:fMRIRegProb} is expected to be sparse. Like in the portfolio optimization problem, in light of the reformulation \eqref{fMRI_primal_IP_PMM_format}, we know that most of the primal variables $x$ converge to zero. 
Close to optimality the presence of such variables would adversely affect 
the conditioning of the matrix in \eqref{regularized normal equations}. 
To prevent that, we employ a heuristic similar to the one introduced 
in Section~\ref{sec:DroppingStrategy} which consists of eliminating 
variables which approach zero and have an associated Lagrange multiplier 
bounded away from zero. 
%
Given $\epsilon_\mathrm{drop} > 0$, $\xi > 0$ and $\mathcal{V} = \emptyset$, 
at every iteration $k$ of IP-PMM, we add to $\mathcal{V}$ each variable 
$j \in \mathcal{I}\setminus \mathcal{V}$ satisfying 
condition~\eqref{portfolio regularized dropping condition}, and
%
%
replace \eqref{regularized normal equations} with the reduced system
\begin{multline} \label{fMRI regularized normal equations: dropped vars}
    \left(A^{\mathcal{H},\mathcal{G}} \left(Q^{\mathcal{G},\mathcal{G}} + \Xi_k^{\mathcal{G},\mathcal{G}} + \rho_k I_{|\mathcal{G}|}\right)^{-1} (A^{\mathcal{H},\mathcal{G}})^\top + \delta_k I_m \right) \Delta y \\
    = r_{2,k}+ A^{\mathcal{H},\mathcal{G}}\left( Q^{\mathcal{G},\mathcal{G}} + \Xi_k^{\mathcal{G},\mathcal{G}} + \rho_k I_{|\mathcal{G}|} \right)^{-1} r_{1,k}^{\mathcal{G}},
\end{multline}
\noindent where $\mathcal{H} = \{1,\ldots,  m\}$, $\mathcal{G} = \mathcal{F}\cup (\mathcal{I}\setminus \mathcal{V})$, and $r_{1,k}$, $r_{2,k}$ are defined in \eqref{augmented system residuals} (with constraint matrix $A^{\mathcal{H},\mathcal{G}}$).


\subsection{Computational Experience}

We consider a dataset consisting of fMRI scans for 16 male healthy US college students (age 20 to 25), with the aim of
analyzing two active conditions: viewing unpleasant and pleasant images \cite{mouraomiranda:2006}. The preprocessed
and registered data\footnote{available from \url{https://github.com/lucabaldassarre/neurosparse}} consist of $1344$ scans of size $122,\! 128$ voxels (only voxels with probability greater than 0.5 of being in the gray matter are considered), with $42$ scans considered per subject and active condition (i.e., $84$ scans per subject in total).

In order to assess the performance of the IP-PMM on this type of problems, we carry out a comparison with two state-of-the-art algorithms for the solution of problem \eqref{eq:fMRIRegProb}:
\begin{itemize}
    \item FISTA. As done for the tests in \cite{baldassarre:2017}, problem \eqref{eq:fMRIRegProb} is reformulated as
    \begin{equation*} 
        \min_{w} \; \frac{1}{2s} \left\| Dw - \hat{y} \right\|^2 + \| \hat{L} w \|_1,
    \end{equation*}
    where $\hat{L}=\left[\tau_1 I_q \ \  \tau_2 L^\top\right]^\top$, and solved by a version of FISTA \cite{beck:2009FISTA} in which the proximal operator associated with $ \| \hat{L} w \|_1$ is approximated by 10 steps of an inner FISTA cycle.
    \item ADMM. We consider the ADMM method \cite{boyd:2011admm} applied to the problem 
    \begin{equation*} 
        \begin{array}{rl}
            \displaystyle \min_{w, u, d}   & \displaystyle \frac{1}{2s} \left\| Dw - \hat{y} \right\|^2 + \tau_1\, \| u \|_1 + \tau_2\, \| d \|_1,\\
            \text{s.t.}  & w - u = 0_q,\\
            & Lw - d = 0_l,
        \end{array}
    \end{equation*}
    in which the minimization of the quadratic function associated with the update of $w$ is approximated by 10 steps of the CG algorithm.
\end{itemize}

In Table~\ref{tab:fMRI_classification_scores} we show the results obtained by applying the algorithms to the solution of the fMRI data classification problem. For each choice of the pair of regularization parameters $(\tau_1,\tau_2)$, we report the average results obtained in a Leave-One-Subject-Out (LOSO) cross-validation test over the full dataset of patients. This consists in using the data concerning 1 patient as the validation set and the data concerning the remaining patients as the training set. Because of this setting, for each problem the size of $w$ is $q = 122,\! 128$, 
the number of rows of $D$ is $s = 1260$, and the dimension of $d = Lw$ is $l = 339,\! 553$.

By preliminary experiments the choice $\tau_1 = \tau_2$ appeared the most appropriate. Furthermore, for the IP-PMM, the parameters $\epsilon_\mathrm{drop}$ and $\xi$ controlling the heuristic described in Section~\ref{sec:dropping-fmri} are set to $10^{-6}$ and $10^{2}$, respectively. To perform a fair comparison between the three algorithms, we consider a stopping criterion based on the execution time, which, after some preliminary tests, is fixed to 30 minutes. The solution of the normal equations system \eqref{fMRI regularized normal equations: dropped vars} is computed by the MATLAB \texttt{pcg} function, for which we set the maximum number of iterations to 2000 and the tolerance as
$$
tol = \left\{\begin{array}{ll}
    10^{-4} & \mbox{if } \|r_{y,k}\|<1, \\
    \max \left\{10^{-8},\;\frac{10^{-4}}{\|r_{y,k}\|} \right\} & \mbox{otherwise,}
\end{array}\right.
$$
where $r_{y,k}$ is the right-hand side of equation \eqref{fMRI regularized normal equations: dropped vars}.

For each algorithm tested, we report the mean and the standard deviation for three quality measures of the solution: \textit{classification accuracy} (ACC), \textit{solution density} (DEN) and \textit{corrected pairwise overlap} (CORR OVR) (see \cite[Section~2.3.3]{baldassarre:2017}). Let $N_f$ be the number of folders in the cross validation setting and let $w_i$ be a given approximate solution to the problem associated with the $i$-th folder. For each $w_i$ we define the accuracy (ACC) as the percentage of test vectors correctly classified by the linear model identified by $w_i$. Given a vector $v\in\Re^q$, we define $\mathcal{Z}(v)$ as the set of indices corresponding to the nonzero components in $v$ and $\mathcal{D}(v) = |\mathcal{Z}(v)|\big/ q$ as the density of $v$. Hence, for each $w_i$ the density (DEN) is computed as $\mathcal{D}(w_i)$. Finally, given any pair of indices $i,j\in\{1,\ldots,N_f\}$, the corrected pairwise overlap is defined as

$$ \mathcal{O}^c_{i,j} = \frac{|\mathcal{Z}(w_i)\cap\mathcal{Z}(w_j)| - E}{\max\{|\mathcal{Z}(w_i)|,\,|\mathcal{Z}(w_j)|\}},$$

\noindent where $E$ is the expected overlap between the support of two
random vectors with density equal to $\mathcal{D}(w_i)$ and $\mathcal{D}(w_j)$, respectively, which is given by $E = q\,\mathcal{D}(w_i)\,\mathcal{D}(w_j)$.
We observe that the corrected pairwise overlap, which may be the less common in the field of machine learning, is meant to measure the ``stability'' of the voxel selection. The three metrics are computed after thresholding the solution, as in \cite{baldassarre:2017}: after sorting the entries by their increasing magnitude, we set to zero the entries contributing at most to $0.01\%$ of the $\ell_1$-norm of the solution.

\begin{table}[!ht]
    \centering
    \caption{Comparison of IP-PMM, FISTA and ADMM in terms of the LOSO cross-validation scores\label{tab:fMRI_classification_scores}}
    {\small
    \begin{tabular}{l|r|r|r|r}
        \toprule
        \textbf{Algorithm}           &   $ \tau_1=\tau_2 $ & \multicolumn{1}{c|}{\textbf{ACC}} &  \multicolumn{1}{c|}{\textbf{DEN}} &  \multicolumn{1}{c}{\textbf{CORR OVR}} \\ \midrule
        \multirow{3}{*}\textbf{IP-PMM} & $ 10^{-2} $ & $86.16 \pm  7.11$ & $20.56 \pm  6.63$ & $43.47 \pm\hspace{5.5pt}  9.09$ \\
        & $ 5 \cdot 10^{-2} $ & $84.90 \pm  4.80$ & $ 3.77 \pm  0.84$ & $62.70 \pm 10.39$ \\
        & $ 10^{-1} $ & $82.29 \pm  6.22$ & $ 2.49 \pm  0.34$ & $82.60 \pm\hspace{5pt}  9.24$ \\ \midrule
         \multirow{3}{*}\textbf{FISTA}  & $ 10^{-2} $ & $86.90 \pm  5.01$ & $88.97 \pm  0.71$ & $ 5.43 \pm\hspace{5.5pt}  0.43$ \\
        & $ 5 \cdot 10^{-2} $ & $84.15 \pm  5.92$ & $19.36 \pm  0.86$ & $65.50 \pm\hspace{5.5pt}  2.68$ \\
        & $ 10^{-1} $ & $81.62 \pm  7.58$ & $ 5.14 \pm  0.44$ & $80.44 \pm\hspace{5.5pt}  5.72$ \\ \midrule
        \multirow{3}{*}\textbf{ADMM}   & $ 10^{-2} $ & $86.46 \pm  6.91$ & $98.70 \pm  0.03$ & $ 0.03 \pm\hspace{5.5pt}  0.01$ \\
        & $ 5 \cdot 10^{-2} $ & $85.57 \pm  5.37$ & $97.97 \pm  0.05$ & $ 0.15 \pm\hspace{5.5pt}  0.04$ \\
        & $ 10^{-1} $ & $82.07 \pm  6.51$ & $97.50 \pm  0.19$ & $ 0.26 \pm\hspace{5.5pt}  0.13$ \\ \bottomrule
    \end{tabular}
    }
\end{table}

By looking at Table~\ref{tab:fMRI_classification_scores}, one can see that IP-PMM appears to be generally better than the other algorithms in enforcing the structured sparsity of the solution, presenting a good level of sparsity and overlap. It is worth noting that, because of its definition, the corrected pairwise overlap tends to zero as the density goes towards $100\%$. Hence, for ADMM, which seems to be unable to enforce sparsity in the solution, the overlap is close to zero. As suggested in \cite{baldassarre:2017}, one can evaluate the results in terms of the distance of the pair (ACC, CORR OVR) from the pair $ (100,100) $ (the smaller the distance, the better the results). For the tests reported in the table, we can see that the best scores are obtained by IP-PMM with regularization parameters $\tau_1=\tau_2= 10^{-1}$, for which the average accuracy is $82.3\%$ and the corrected overlap is $82.6\%$ with an average solution density of $2.5\%$. 

\begin{figure}[h!]
    \centering
    \includegraphics[width=0.5\textwidth]{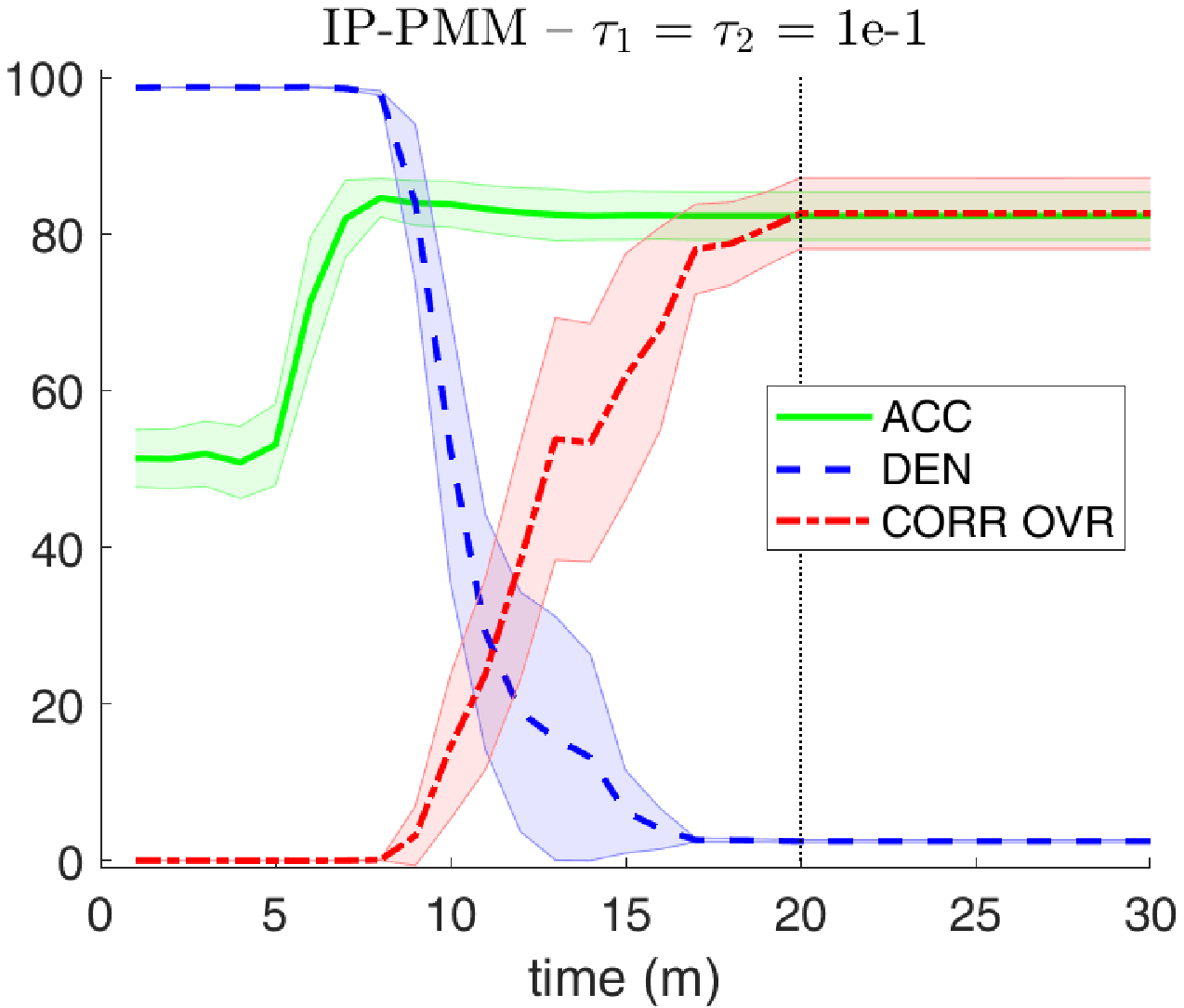}\includegraphics[width=0.5\textwidth]{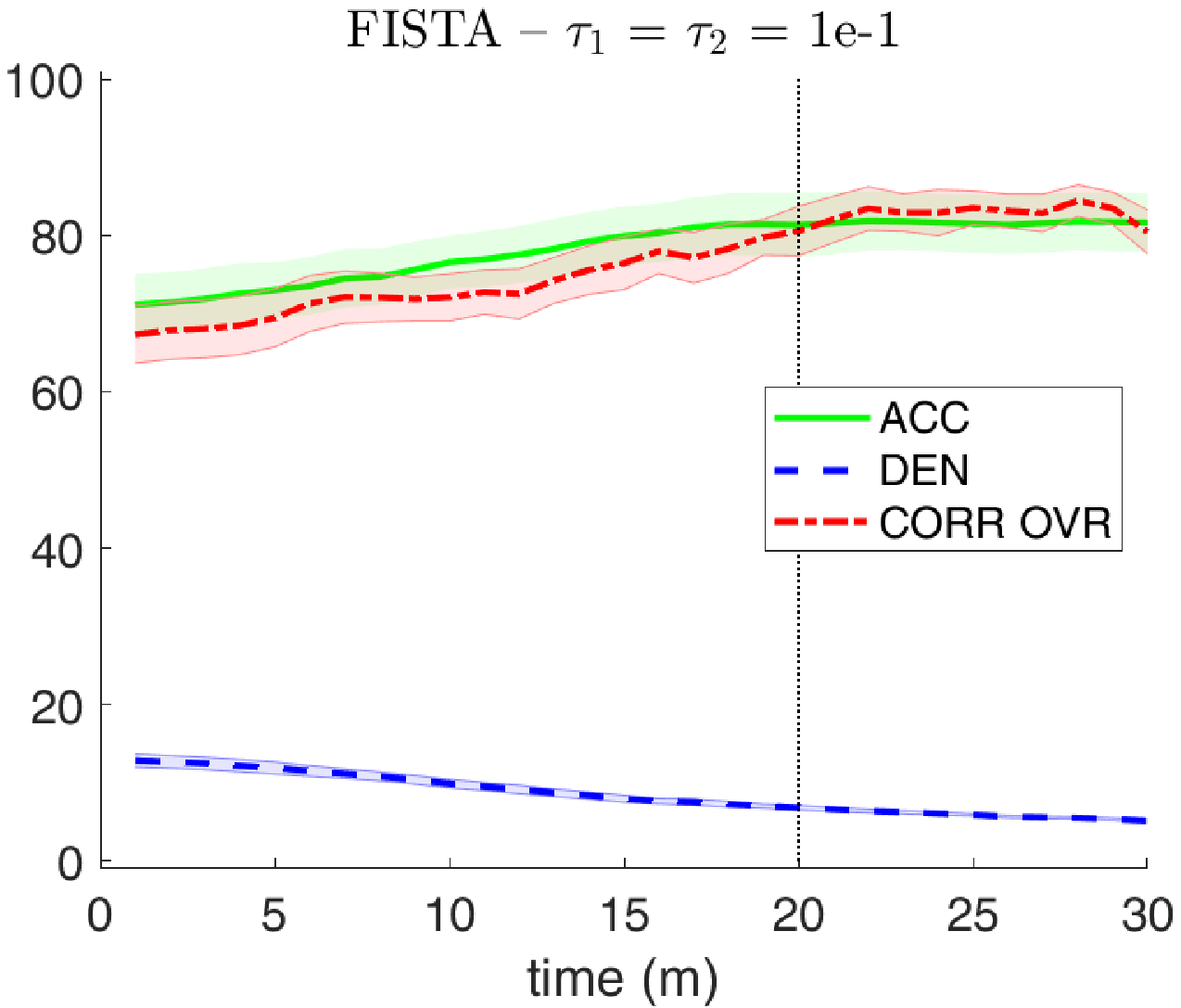}
    \vskip -9pt
    \caption{History of classification accuracy, solution density and corrected pairwise overlap for IP-PMM (\textit{left}) and FISTA (\textit{right}), in the case $\tau_1=\tau_2=10^{-1}$. For the three quantities we report average measures with $95\%$ confidence intervals. \label{fig:fMRI_time_comparison}}
\end{figure}
To further evaluate the efficiency of IP-PMM in the solution of this class of problems, we compare its performance in terms of elapsed time against the performance of FISTA on the problem where the two methods reach the best scores, i.e., with $\tau_1=\tau_2=10^{-1}$. For all the 16 instances of the LOSO cross validation, we store the current solution of each algorithm after every minute and, at the end of the execution, we compute the three quality measures for such intermediate solutions. The results are shown in Figure~\ref{fig:fMRI_time_comparison} in terms of history of the mean values (lines) together with their $95\%$ confidence intervals (shaded regions). From the plots we can see that while FISTA reaches the measures reported in Table~\ref{tab:fMRI_classification_scores} at the end of the 30-minute run, the performance of IP-PMM stabilizes after about 20 minutes. At the 20 minutes mark we observe that for IP-PMM the value of each of the three metrics is the same as the one reported in Table~\ref{tab:fMRI_classification_scores}. For FISTA, while the accuracy ($81.32\%$) and overlap ($80.54\%$) have similar values as those reported in the table, we observe a larger density ($6.83\%$).


\section{TV-based Poisson Image Restoration\label{sec:ImgRestor}}
\par Next we consider the restoration of images corrupted by Poisson noise, which arises in many applications, such as fluorescence microscopy, computed tomography (CT) and astronomical imaging (see, e.g., \cite{DiSerafinoLandiViola_TV_Poisson_restor} and the references therein). 
In the discrete formulation of the restoration problem, the object to be restored is represented by a vector
$w \in \Re^n$ and the measured data are assumed to be a vector $g \in \mathbb{N}_0^m$, whose entries $g^j$
are samples from $m$ independent Poisson random variables $G^j$ with probability
$$
    P(G^j = g^j) = \frac{e^{-(D w + a)^j}\left[(D w + a)^{j}\right]^{g^j}}{g^j!},
$$
where $a \in \Re_+^{m}$ models the background radiation detected by the sensors. The matrix $D = (d^{ij}) \in \Re^{m \times n}$ models the functioning of the imaging system and satisfies 
\begin{equation*}
     d^{ij} \geq 0 \mbox{ for all }  i, j, \qquad \sum_{i=1}^m d^{ij} = 1 \mbox{ for all } j.
\end{equation*}
Here we assume that $D$ represents a convolution operator with periodic boundary conditions, which implies that $D$ has a Block-Circulant structure with Circulant Blocks (BCCB). Hence, $D w$ is computed expeditiously using the 2-dimensional Fast Fourier Transform (FFT).
The maximum-likelihood approach \cite{bertero:2009} for the estimation of $u$ leads to the minimization of the \textit{Kullback-Leibler (KL) divergence} of $Dw+a$ from $g$:
\begin{equation} \label{eq:kl}
    D_{KL}(w) \equiv  D_{KL}(D w + a, g) = 
    \sum_{j=1}^m \left( g^j \ln \frac{g^j}{(D w + a)^j} + (D w + a)^j - g^j \right),
\end{equation}
where we set $g^j \ln (g^j / (D w + a)^j) = 0 $ if $g^j = 0$ (we implicitly assume that $g$ has been converted into a real vector with entries ranging in the same interval as the entries of $w$). Since the estimation problem is highly ill conditioned, a regularization term is added to \eqref{eq:kl}. We consider the Total Variation (TV)~\cite{rudin:1992}, which has received considerable attention because of its ability of preserving edges and smoothing flat areas of the images. Notice that, while it may introduce staircase artifacts, TV is still applied in many medical and biological applications (see, e.g., \cite{Barnard:2018, Mota:2016, Zhang:2018} and J. Huang's webpage\footnote{\url{http://ranger.uta.edu/\textasciitilde huang/R_CSMRI.htm}}).
The feasible set of the problem is defined by non-negativity constraints on the image intensity and the linear constraint $\sum_{i=1}^{n} w^i = \sum_{j=1}^{m} (g^j-a^j) \equiv r$ which guarantees preservation of the total intensity of the image.

The resulting model is
\begin{equation} \label{eqn:poispb}
    \begin{array}{cl}
        \displaystyle \min_{w} & \displaystyle D_{KL}(w)  + \lambda \| L w \|_1 \\
        \text{s.t.} & e_n^\top w =  r, \\
        & w \ge 0,
    \end{array}
\end{equation}
where $L \in \Re^{l \times n}$ is the matrix arising from the discretization of the TV functional (as in \cite{Chambolle_TV_discretization}).


\subsection{Specialized IP-PMM for Image Restoration Problems}
\par By employing the splitting strategy used in the previous sections, we can transform problem \eqref{eqn:poispb} to the following equivalent form:
\begin{equation} \label{eqn:smooth_poispb}
    \begin{array}{cl}
        \displaystyle \min_{x} & \displaystyle f(x) \equiv D_{KL}(w)  + c^\top u, \\
        \mbox{s.t.} & A x = b,\\
        & x \geq 0,  \\
    \end{array}
\end{equation}
\noindent where, after introducing the additional constraint $d=Lw$, and letting $\overline{m} = l+1$, $\overline{n} = n + 2l$, we set $x = [w^\top,\ u^\top]^\top \in \Re^{\overline{n}}$, $u = [(d^+)^\top,\ (d^-)^\top]^\top \in \Re^{2l}$, $c = \lambda\,e_{2l}$, $b = [r, \ 0_{l}^\top]^\top \in \mathbb{R}^{\overline{m}}$,  and 
\begin{equation*}
    A = \begin{bmatrix}
        e_n^\top & 0_l^\top & 0_l^\top\\
        L & - I_l & I_l
    \end{bmatrix} \in \mathbb{R}^{\overline{m}\times\overline{n}}.
\end{equation*}
\par We solve problem \eqref{eqn:smooth_poispb} by using IP-PMM combined with a perturbed composite Newton method \cite{SIAMJOpt:TapZhaSalWei}. 
Following the presentation in Section \ref{sec:Interior Point-Proximal Method of Multipliers}, we know that at the $k$-th iteration of the method we have to solve two linear systems of the form of \eqref{regularized augmented system}. In order to avoid factorizations, every such system is solved using the preconditioned MINimal RESidual (MINRES) method \cite{iter:PS}. In order to accelerate the convergence of MINRES, we employ a block-diagonal preconditioner, which uses a diagonal approximation of $\nabla^2 f(x)$. More specifically, at iteration $k$ of IP-PMM, 
we have the following coefficient matrix:
\[M_{k} = \begin{bmatrix}
    -H_{k} & A^\top \\
    A & \delta_k I_{\overline{m}} \\
\end{bmatrix}, \]
\noindent where $H_{k} = (\nabla^2 f(x_{k}) + \Theta_{k}^{-1} + \rho_k I_{\overline{n}})$, and we precondition it using the matrix
\begin{equation} \label{poisson_block_diag_preconditioner}
   \widetilde{M}_{k} = \begin{bmatrix}
        \widetilde{H}_{k} & 0_{\overline{n},\overline{m}} \\
        0_{\overline{m},\overline{n}} & A \,\widetilde{H}_{k}^{-1} A^\top + \delta_k I_{\overline{m}}\\
    \end{bmatrix},
\end{equation}
\noindent where $\widetilde{H}_{k}$ is a diagonal approximation of $H_k$. In order to analyze the spectral properties of the preconditioned matrix, we follow the developments in \cite{paper_IP_PMM_inexact}. More specifically, we define $\widehat{H}_{k} \coloneqq\widetilde{H}^{-\frac{1}{2}}_{k} H_{k}\widetilde{H}^{\frac{1}{2}}_{k}$, and let:
\[\alpha_H = \lambda_{\min}(\widehat{H}_{k}),\qquad \beta_H = \lambda_{\max}(\widehat{H}_{k}), \qquad \kappa_H = \frac{\beta_H}{\alpha_H}. \]
\noindent Using this notation, we know that an arbitrary element of the numerical range of this matrix is represented as $\gamma_{H} \in W(\widehat{H}_{k}) = [\alpha_H,\beta_H]$. Furthermore, we observe that in the special case where $\widetilde{H}_k = \diag(H_k)$, we have  $\alpha_H \leq 1 \leq \beta_H$ since 
\[\frac{1}{n+2l} \sum_{i=1}^{n+2l} \lambda_i (\widehat{H}_{k}^{-1}H_{k,j}) = \frac{1}{n+2l} \, \mathrm{Tr}(\hat{H}_{k}^{-1}H_{k}) = 1.\]
\begin{theorem} \label{thm:spectral_analysis_block_diagonal_preconditioner}
    Let $k$ be an arbitrary IP-PMM iteration. Then, the eigenvalues of $\widetilde{M}_{k}^{-1} M_{k}$ lie in the union of the following intervals: 
    \[I_{-} = \bigg[-\beta_H - 1, -\alpha_H \bigg], \qquad I_{+} = \bigg[ \frac{1}{1+\beta_H}, 1\bigg]. \]
\end{theorem}
\begin{proof}
    The proof follows exactly the developments in \cite[Theorem 3.3]{paper_IP_PMM_inexact}.
\end{proof}

In problem \eqref{eqn:smooth_poispb}, $f(x) = D_{KL}(w) + c^\top u$ and hence
\[\nabla f(x) = \begin{bmatrix}
    \nabla D_{KL} (w)\\
    c
\end{bmatrix}, \qquad \nabla^2 f(x) = \begin{bmatrix}
    \nabla^2 D_{KL}(w) & 0_{n,2l}\\
    0_{2l,n} & 0_{2l,2l}
\end{bmatrix}, \]
\noindent where 
\[\nabla D_{KL}(w) = D^\top \left( e_m - \frac{g}{Dw+a} \right), \qquad \nabla^2 D_{KL}(w) = D^\top U(w)^2 D,\]
\noindent with $U(w) = \diag \left( \frac{\sqrt{g}}{Dw+a} \right)$. Here the ratios and the square root are assumed to be component-wise. Notice that $D$ might be dense; however, as previously noted, its action can be computed via the FFT. Unfortunately, $D^\top U(w)^2 D$ is not expected to be close to multilevel circulant. Even if it could be well-approximated by a multilevel circulant matrix, the scaling matrix of IP-PMM would destroy this structure. In other words, we use the structure of $D$ only when applying it to a vector. As a result, we only store the first column of $D$ and we use the FFT to apply this matrix to a vector. This allows us to compute the action of the Hessian easily.

\begin{remark}
    The obvious choice would be to employ the approximation $\widetilde{H}_k = \diag(H_k)$, but the structure of the problem makes this choice rather expensive. A more efficient alternative is to use $\widetilde{H}_k = U(w_k)^2$,  which is easier to compute and, as we will see in the following section, leads to good reconstruction results in practice.
\end{remark}

\subsection{Computational Experience\label{sec:esperiments_poisson_restoration}}
To evaluate the performance of the IP-PMM on this class of problems, we consider a set of three $256\times 256$ grayscale images, which are presented in Figure~\ref{fig:kltv_original_images}.
\begin{figure}[h!]
    \centering
    \includegraphics[width=0.3\textwidth]{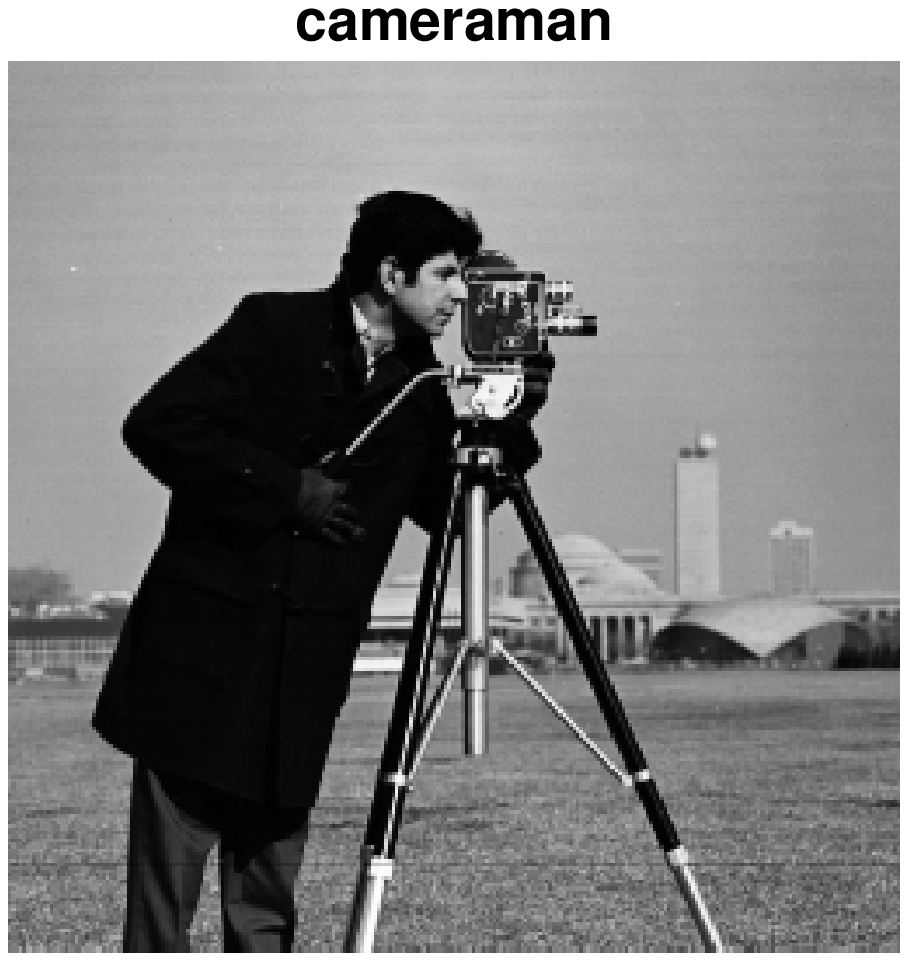} \hfill \includegraphics[width=0.3\textwidth]{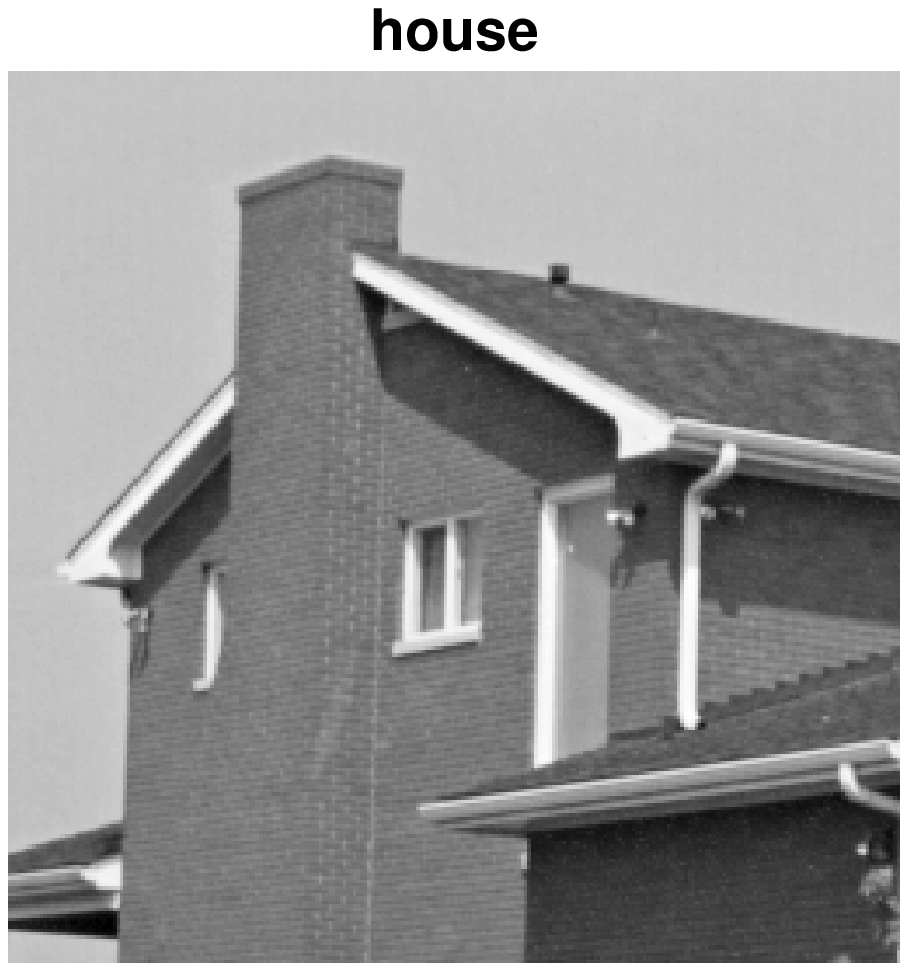} \hfill
    \includegraphics[width=0.3\textwidth]{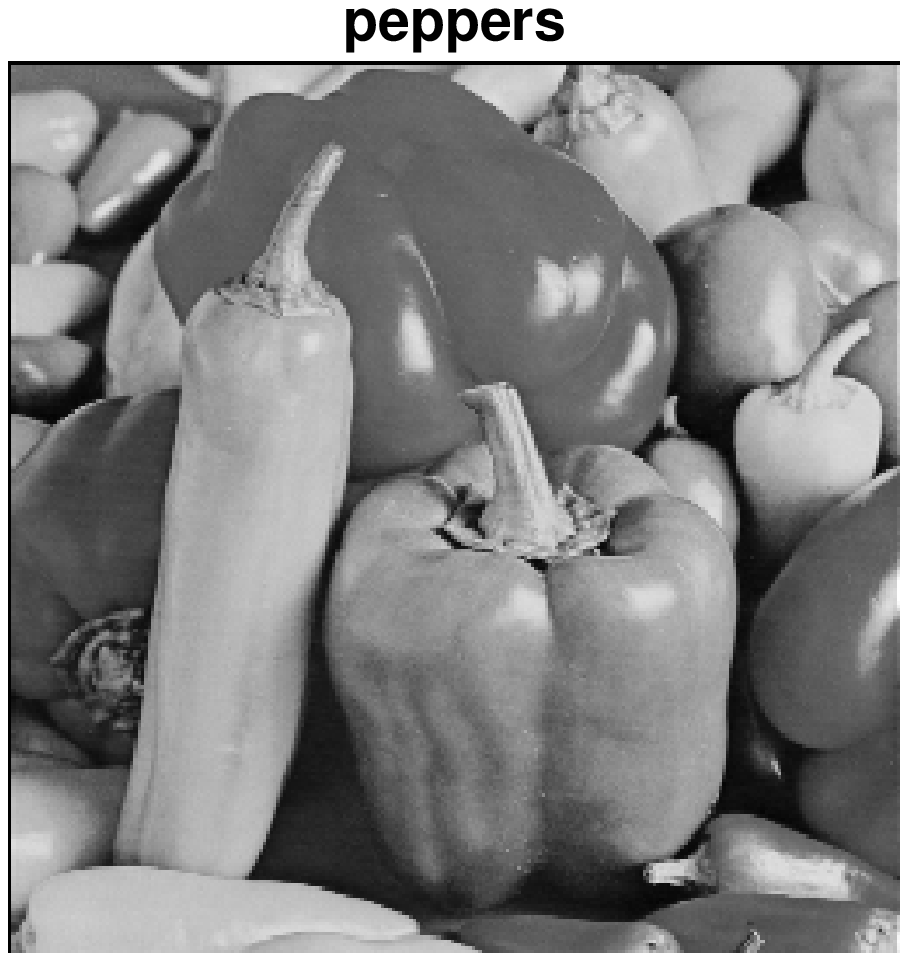}
    \caption{The three $256\times 256$ grayscale images of the image restoration tests. \label{fig:kltv_original_images}}
\end{figure}
For each of the three images we set up three restoration tests, where the images are corrupted by Poisson noise and $D$ represents one of the following blurs: Gaussian blur (GB), motion blur (MB), and out-of-focus blur (OF) (see, e.g., \cite{hansen:2006} for further details).

We compare the proposed method with the state-of-the-art Primal-Dual Algorithm with Linesearch (PDAL) proposed in \cite{malitsky:2018}. By following the example of \cite[Algorithm~2]{wen:2016}, problem \eqref{eqn:poispb} is reformulated as
\begin{equation}
    \min_w \,\max_{p,y} \; g^\top \ln(1+y) - y^\top (D w+ a) - \lambda \, w^\top L^\top p + \chi_\infty(p) + \chi_\mathcal{C}(w),
    \label{eqn:kltv_pdal_reformulation}
\end{equation}
where $\chi_\infty$ is the characteristic function of the $\infty$-norm unit ball and $\chi_\mathcal{C}$ the characteristic function of the feasible set $\mathcal{C}$ of problem \eqref{eqn:poispb}. It is worth noting that the PDAL algorithm for the solution of problem \eqref{eqn:kltv_pdal_reformulation} requires at each step a projection on the feasible set $\mathcal{C}$, which is performed here by using the secant algorithm proposed by Dai and Fletcher in~\cite{dai:2006a}. Concerning the parameters of PDAL, we use the same notation and tuning as in~\cite{malitsky:2018}. Following~Section~6 of that paper, we set $ \mu=0.7$, $\delta=0.99$ and $\beta = 25$. The initial steplength is $ \tau = \sqrt{1/\omega} $, where $\omega$ is an estimate of $\|M^\top M\|$ and $M = \left[ D^\top \; L^\top \right]^\top $ is the matrix linking the primal and dual variables. In the IP-PMM, we use the MINRES code by Michael Saunders and co-workers\footnote{available from \url{https://web.stanford.edu/group/SOL/software/minres/}} for which we set the relative tolerance $tol=10^{-4}$ and the maximum number of iterations at each call equal to 20. The regularization parameter $\lambda$ is determined by trial and error to minimize the Root Mean Square Error (RMSE) obtained by IP-PMM. We recall that, denoting the original image as $\bar{w} \in \Re^n$, for any given approximate solution $w \in \Re^n$ we have that
$$ \mathrm{RMSE}(w) = \frac{1}{\sqrt{n}}\|w-\bar{w}\|_2.$$
For all the problems, the starting point is chosen to be the noisy and blurry image, i.e., $g$.

\begin{figure}[h!]
    \centering
    \includegraphics[width=0.31\textwidth]{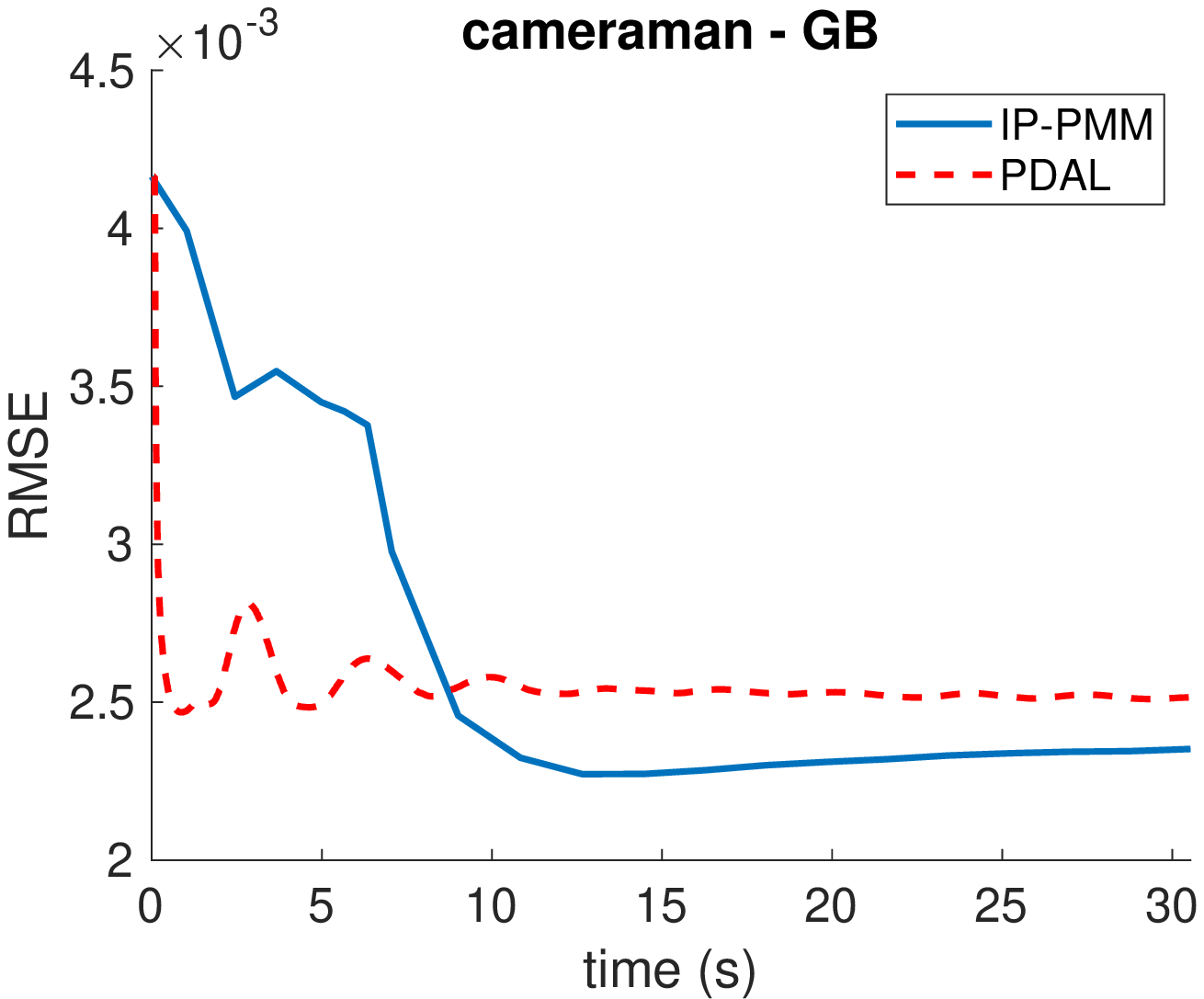}\hfill
    \includegraphics[width=0.31\textwidth]{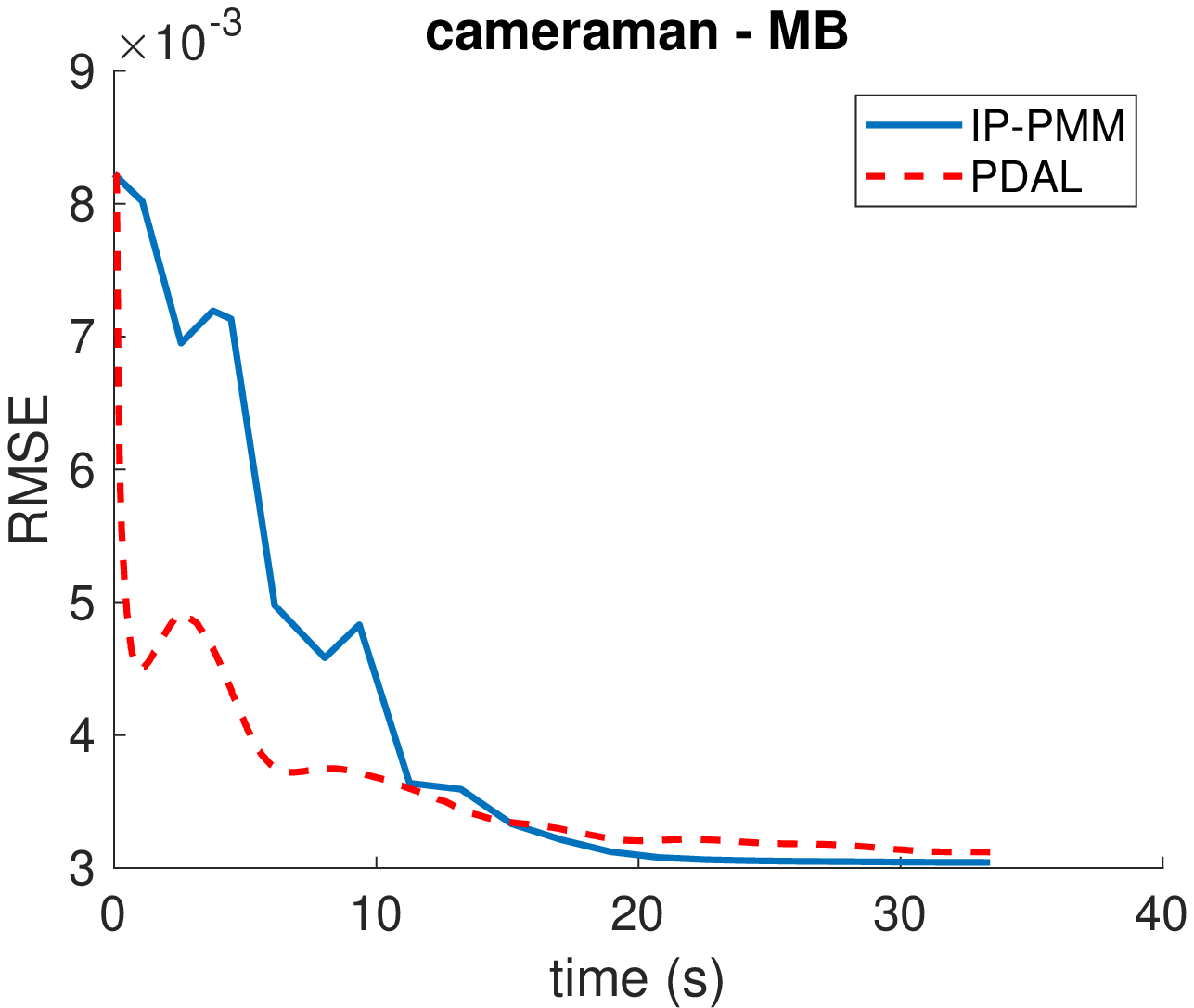}\hfill
    \includegraphics[width=0.31\textwidth]{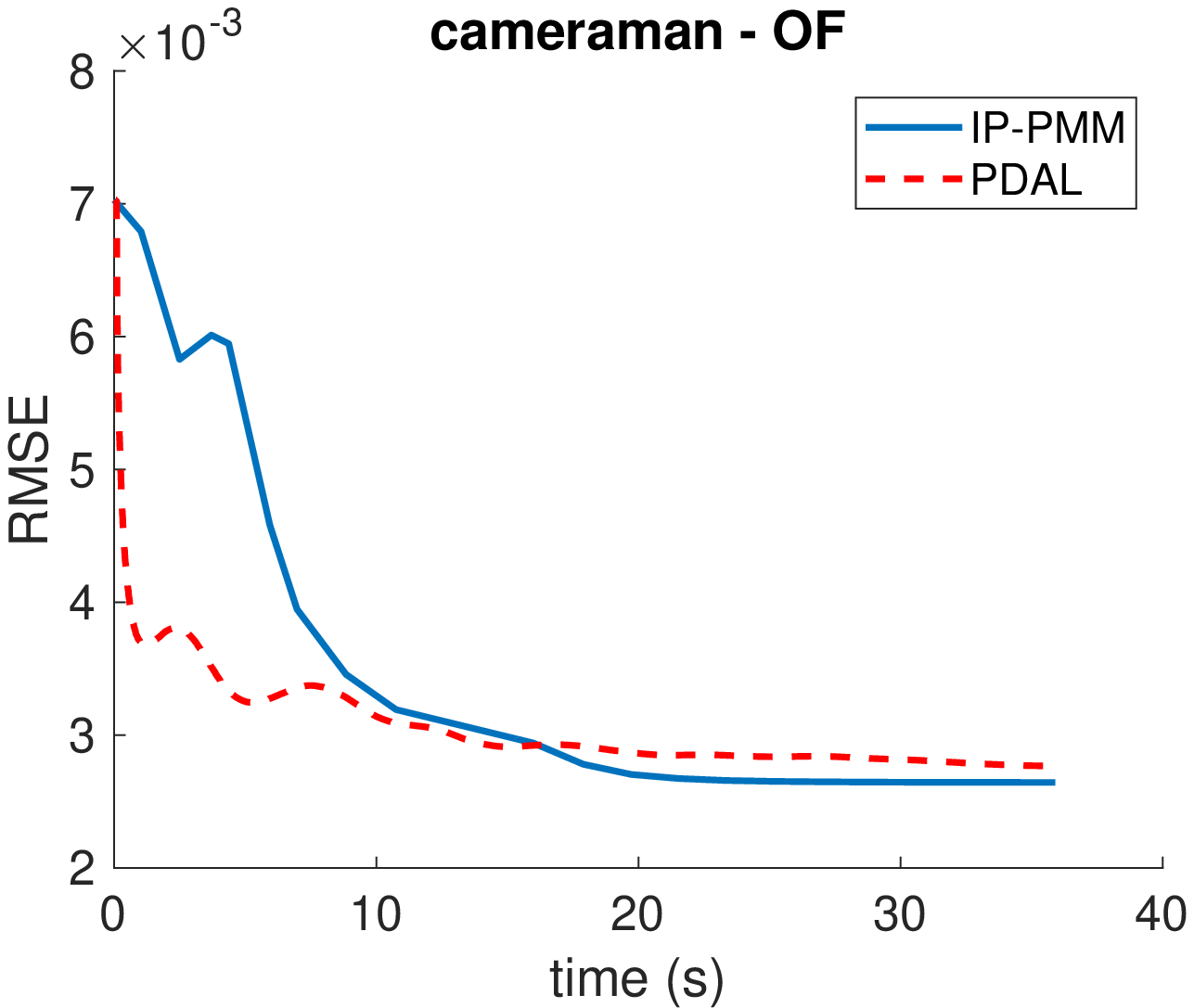}\\[1mm]
    \includegraphics[width=0.31\textwidth]{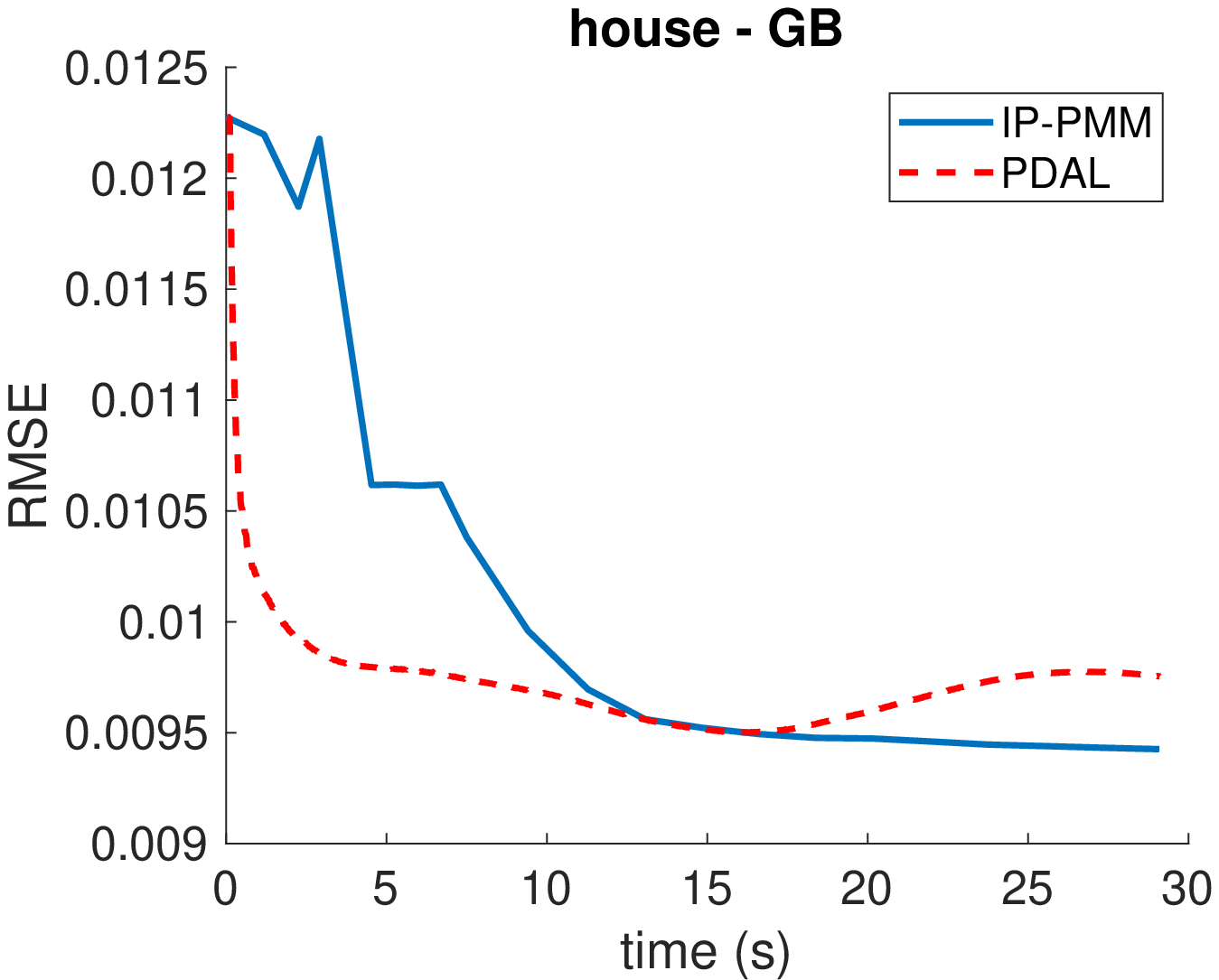}\hfill
    \includegraphics[width=0.31\textwidth]{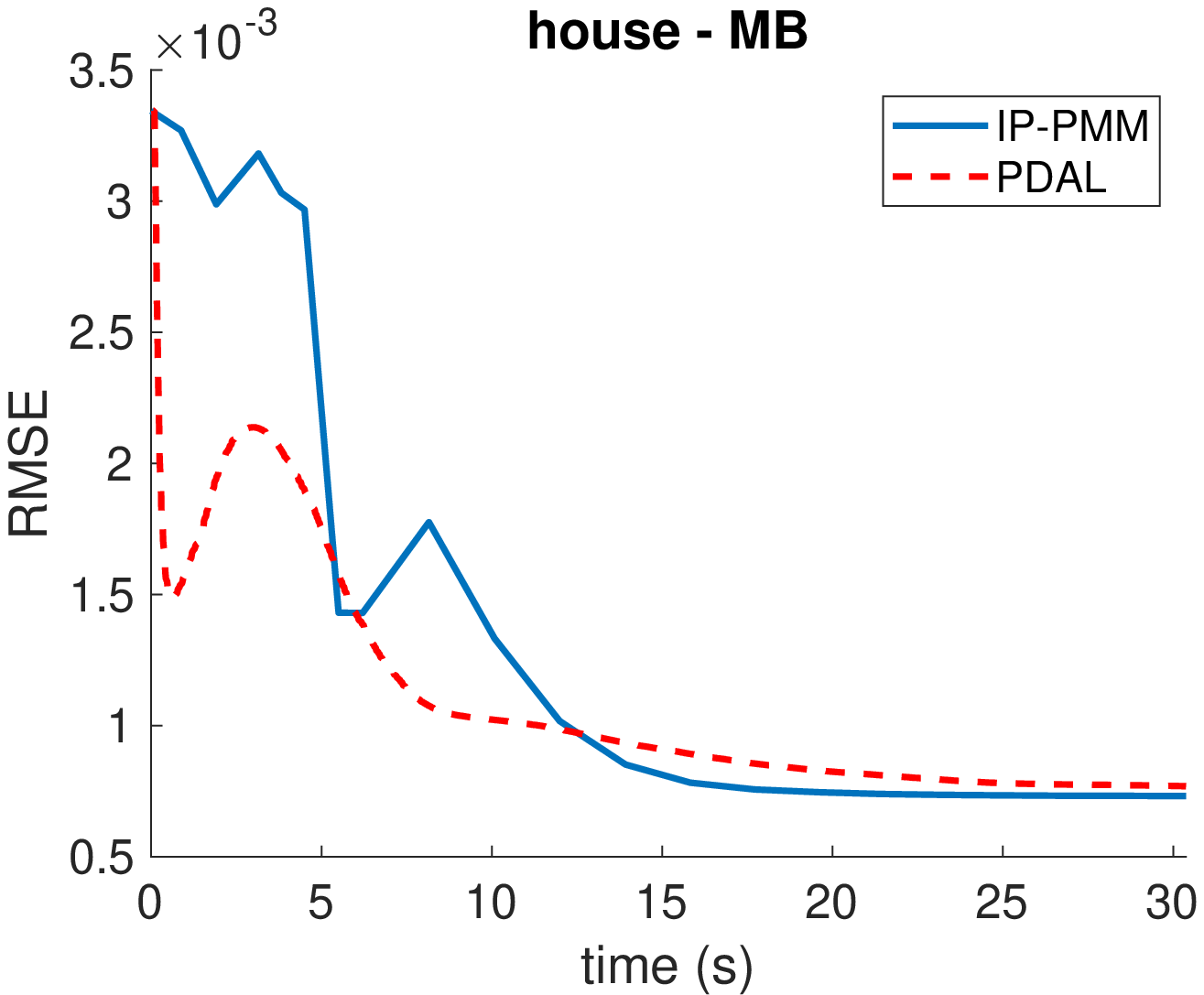}\hfill
    \includegraphics[width=0.31\textwidth]{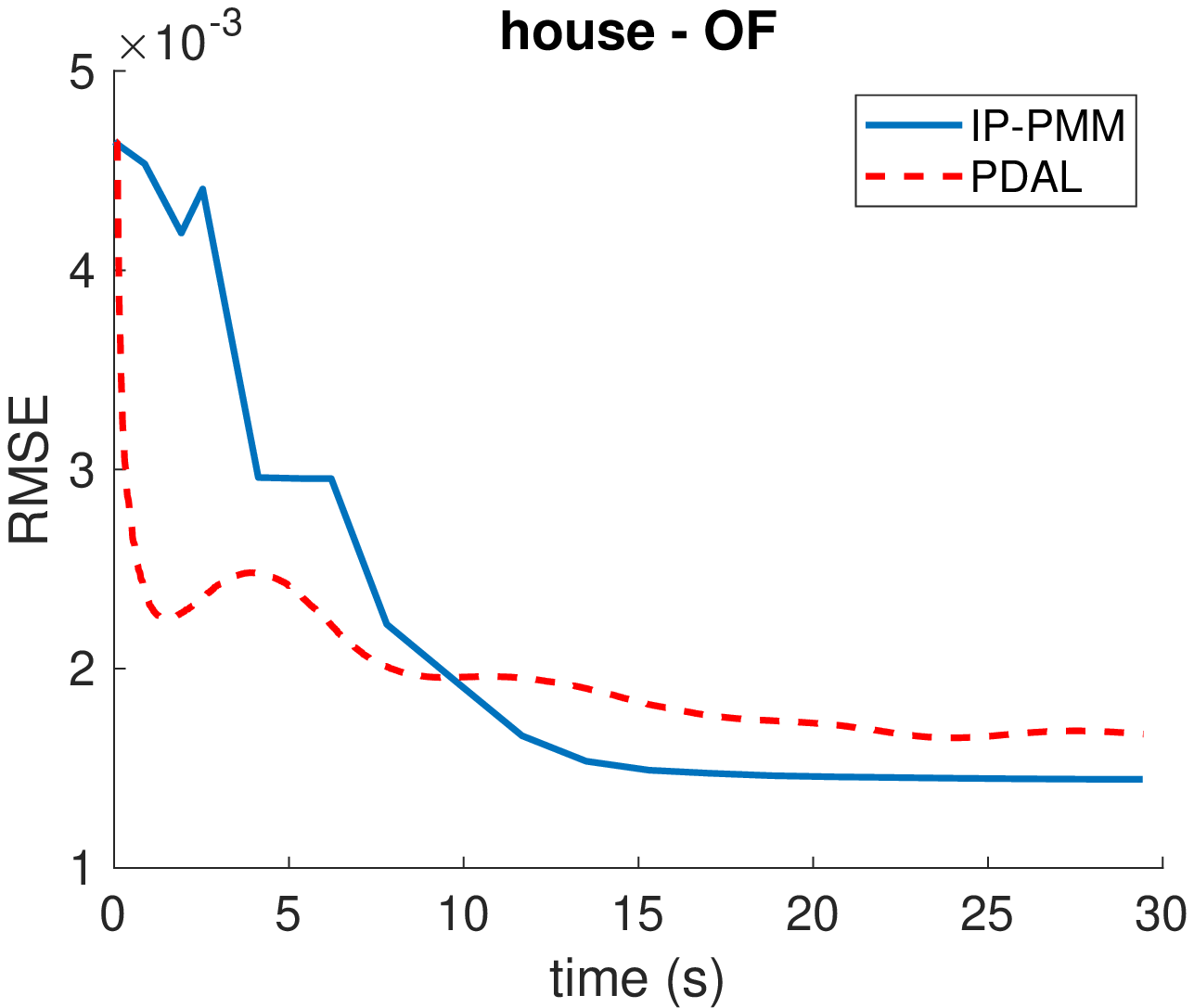}\\[1mm]
    \includegraphics[width=0.31\textwidth]{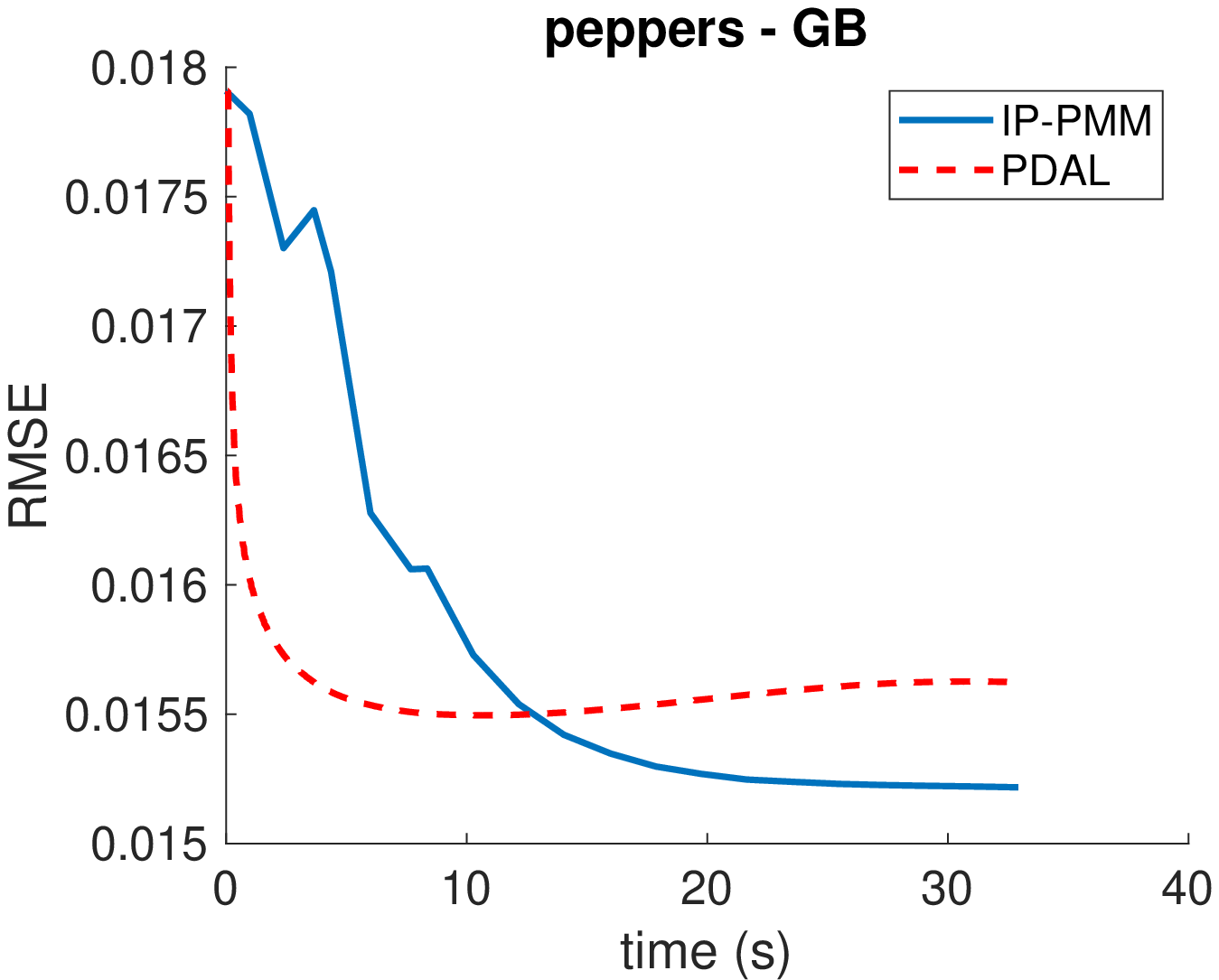}\hfill \includegraphics[width=0.31\textwidth]{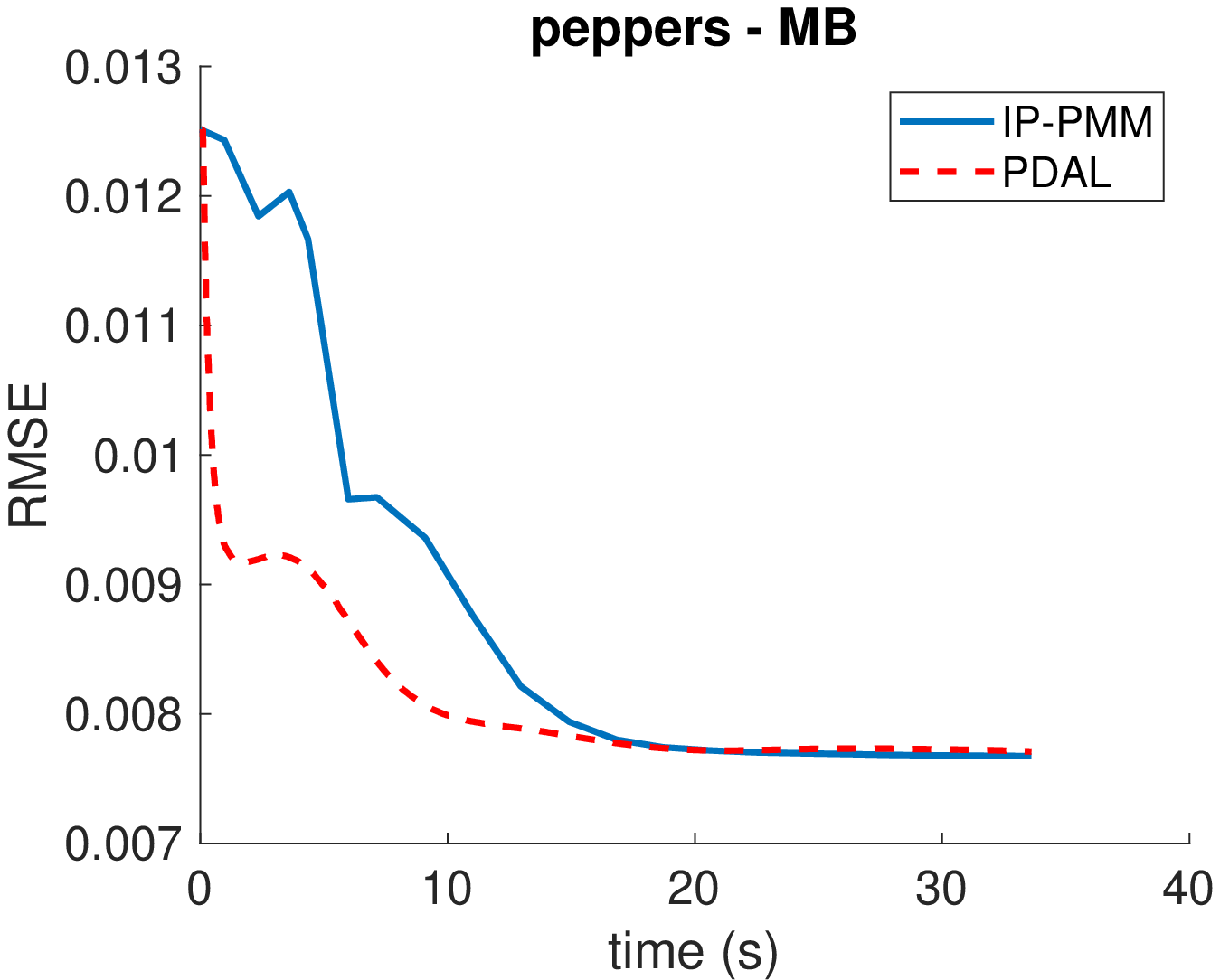}\hfill
    \includegraphics[width=0.31\textwidth]{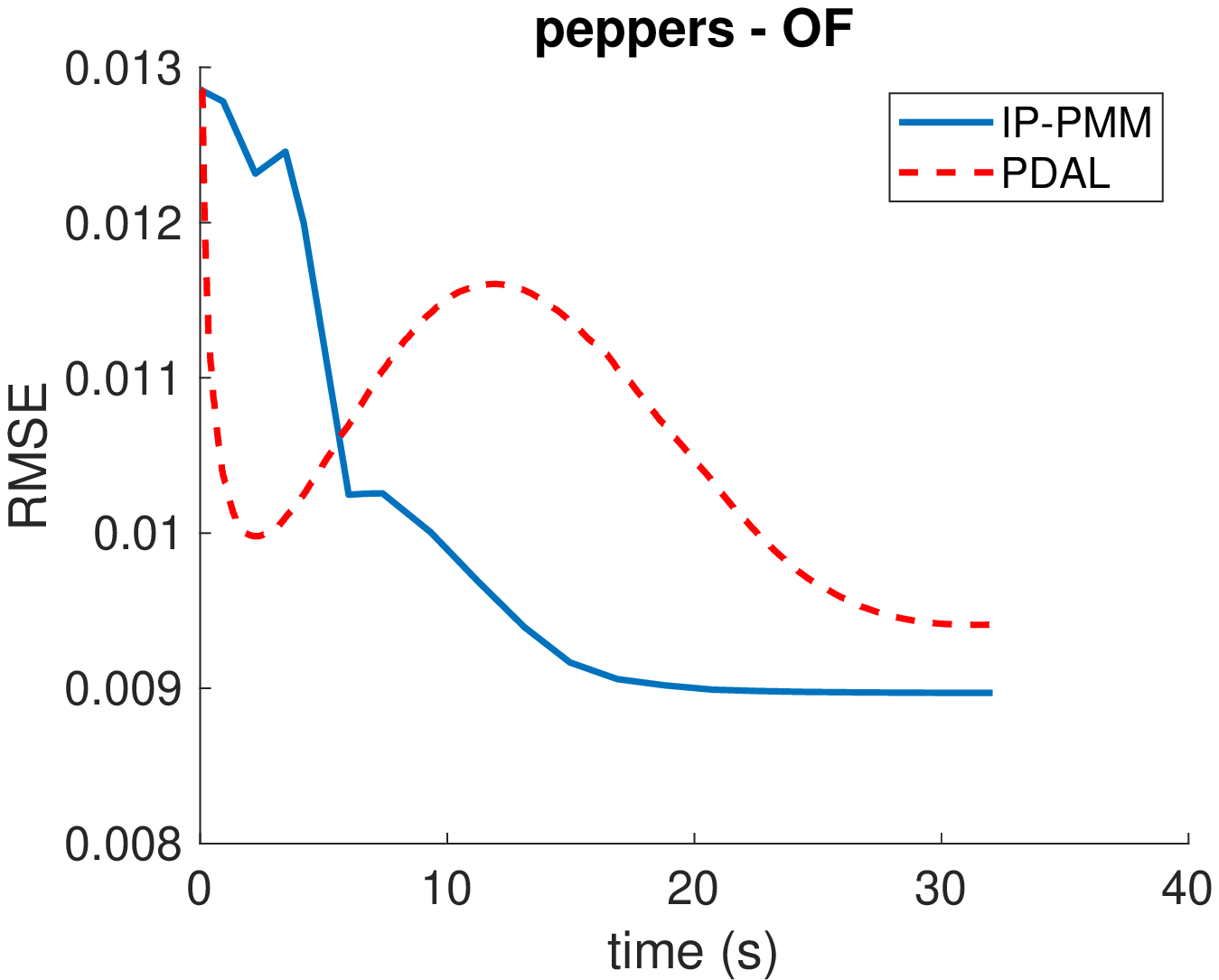}
    \caption{Comparison between IP-PMM and PDAL in terms of Root Mean Square Error (RMSE) vs execution time in the solution of the 9 image restoration problems. From top to bottom, the rows refer to the cameraman instances, the house instances and the peppers instances, respectively. From left to right, the columns refer to the GB, MB and OF, respectively.\label{fig:kltv_mse_plots}}
\end{figure}    

For all 9 tests we run 20 iterations of the IP-PMM method and let PDAL run for the same amount of time. 
In Figure~\ref{fig:kltv_mse_plots} we report a comparison between the two algorithms in terms of elapsed time versus Root Mean Square Error (RMSE) in the solution of the 9 instances described above. As can be seen from the plots, the IP-PMM clearly outperforms PDAL on the instances with GB and OF (columns 1 and 3, respectively, of Figure~\ref{fig:kltv_mse_plots}), while on the instances characterized by MB the two algorithms perform comparably.

\begin{table}[h!]
    \centering
    \caption{Comparison between IP-PMM and PDAL in terms of RMSE, PSNR and MSSIM computed at the solutions provided by the two algorithms. \label{tab:kltv_scores}}
    \footnotesize
    \begin{tabular}{l|c|c|c|c|c|c}
    	\toprule
    	                 &        \multicolumn{3}{c|}{\textbf{IP-PMM}}         &          \multicolumn{3}{c}{\textbf{PDAL}}         \\ \midrule
    	\textbf{Problem} & \textbf{RMSE} & \textbf{PSNR} & \textbf{MSSIM} & \textbf{RMSE} & \textbf{PSNR} & \textbf{MSSIM} \\ \midrule
    	cameraman - GB   &  4.85e$-$2   &  2.63e$+$1   &   8.33e$-$1   &  5.02e$-$2   &  2.60e$+$1   &   8.22e$-$1   \\
    	cameraman - MB   &  5.52e$-$2   &  2.52e$+$1   &   8.11e$-$1   &  5.59e$-$2   &  2.51e$+$1   &   7.77e$-$1   \\
    	cameraman - OF   &  5.14e$-$2   &  2.58e$+$1   &   7.98e$-$1   &  5.26e$-$2   &  2.56e$+$1   &   7.62e$-$1   \\ \midrule
    	house - GB       &  9.71e$-$2   &  2.03e$+$1   &   7.51e$-$1   &  9.88e$-$2   &  2.01e$+$1   &   6.92e$-$1   \\
    	house - MB       &  2.70e$-$2   &  3.14e$+$1   &   8.67e$-$1   &  2.77e$-$2   &  3.11e$+$1   &   8.43e$-$1   \\
    	house - OF       &  3.80e$-$2   &  2.84e$+$1   &   8.33e$-$1   &  4.09e$-$2   &  2.78e$+$1   &   7.70e$-$1   \\ \midrule
    	peppers - GB     &  1.23e$-$1   &  1.82e$+$1   &   7.46e$-$1   &  1.25e$-$1   &  1.81e$+$1   &   6.57e$-$1   \\
    	peppers - MB     &  8.76e$-$2   &  2.12e$+$1   &   8.90e$-$1   &  8.78e$-$2   &  2.11e$+$1   &   8.72e$-$1   \\
    	peppers - OF     &  9.47e$-$2   &  2.05e$+$1   &   8.01e$-$1   &  9.70e$-$2   &  2.03e$+$1   &   6.60e$-$1   \\ \bottomrule
    \end{tabular}%
\end{table}%

To better analyze the difference between the solutions provided by the two algorithms, one can look at Table~\ref{tab:kltv_scores}, where we report the value of three scores: RMSE, Peak Signal-to-Noise Ratio (PSNR), which is defined as
$$\mathrm{PSNR}(w) = 20 \log_{10}\frac{\max_i{\bar{w}^i}}{\mathrm{RMSE}(w)},$$
and Mean Structural SIMilarity (MSSIM), which is a structural similarity measure related to the perceived visual quality of the image (see \cite{wang:2004} for a detailed definition). It is worth noting that for RMSE smaller values are better, while for PSNR and MSSIM, higher values indicate better noise removal and perceived similarity between the restored and original image, respectively. From the table it is clear that in all the considered cases IP-PMM is able to produce a better restored image than PDAL, having always a larger MSSIM, also when the RMSE and PSNR values are comparable.

For the sake of space, we now restrict the comparison to the cases where the two algorithms seem to have reached equivalent solutions in terms of RMSE, to understand the differences in the restored images. We focus on the three instances in which $D$ represents MB (second column of Figure~\ref{fig:kltv_mse_plots}). In Figure~\ref{fig:mb_images} we report the results for cameraman, house and peppers with MB. By looking at the images one can see that those reconstructed by IP-PMM appear to be smoother (look, for example, at the sky in cameraman and house), which somehow indicates that the IP-PMM is better than PDAL in enforcing the TV regularization. Observe that this ``visual'' difference is reflected by the higher values of MSSIM reported for IP-PMM in Table~\ref{tab:kltv_scores}.
\begin{figure}[h!]
    \centering
    \includegraphics[width=0.3\textwidth]{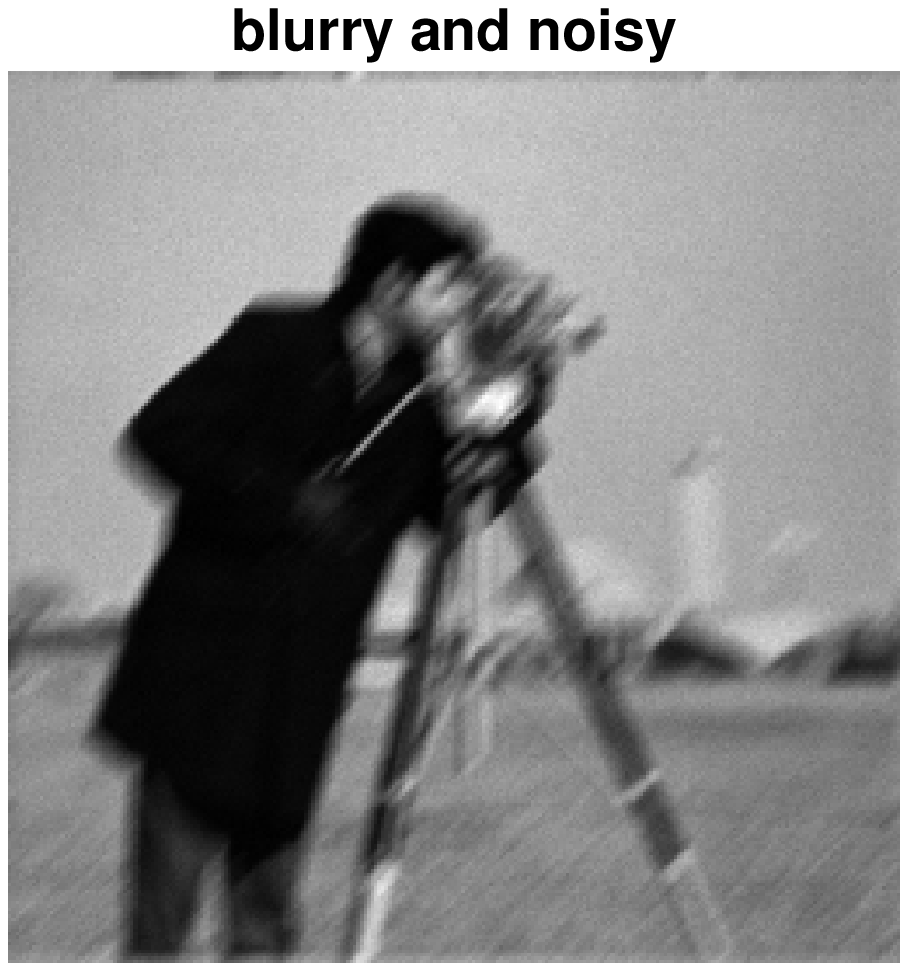} \hfill \includegraphics[width=0.3\textwidth]{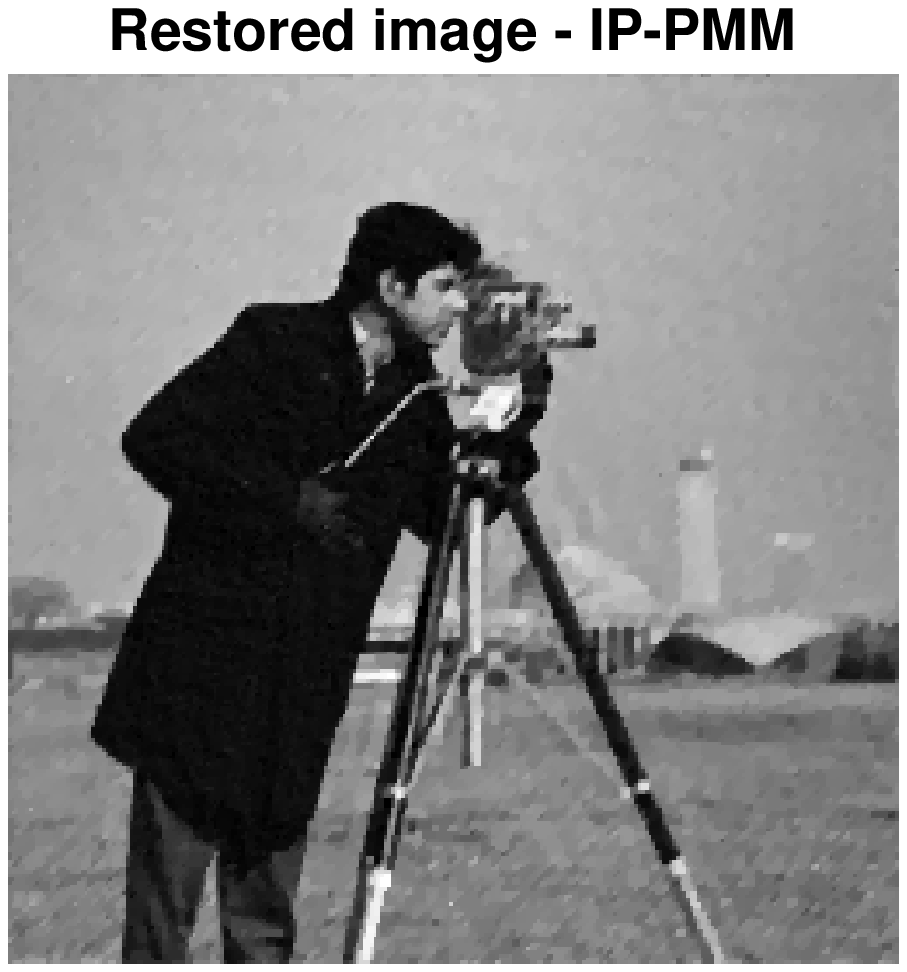} \hfill
    \includegraphics[width=0.3\textwidth]{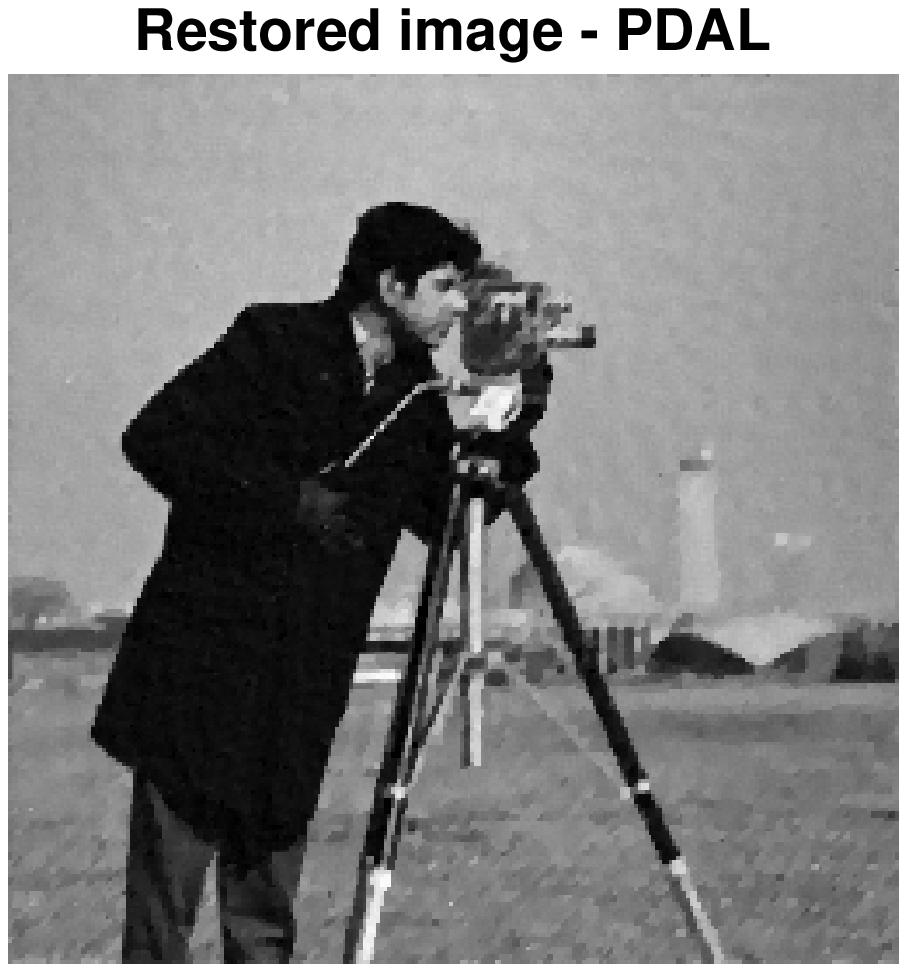} \\[5pt]
    \includegraphics[width=0.3\textwidth]{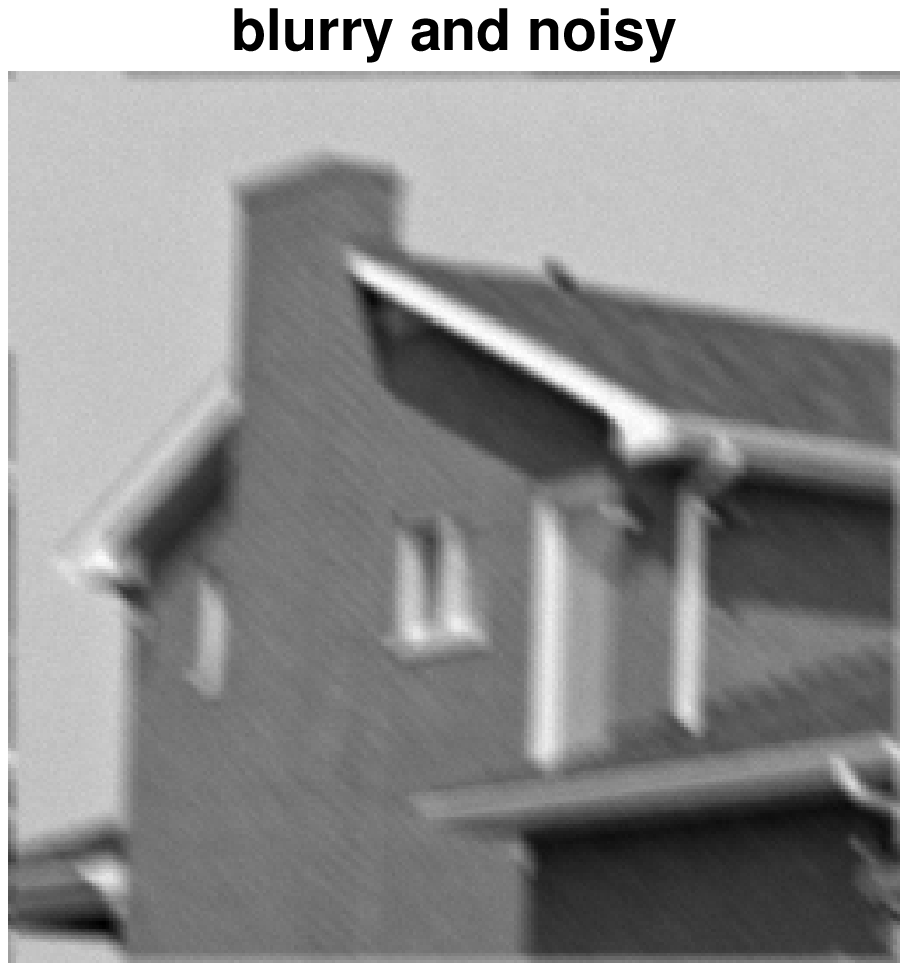} \hfill \includegraphics[width=0.3\textwidth]{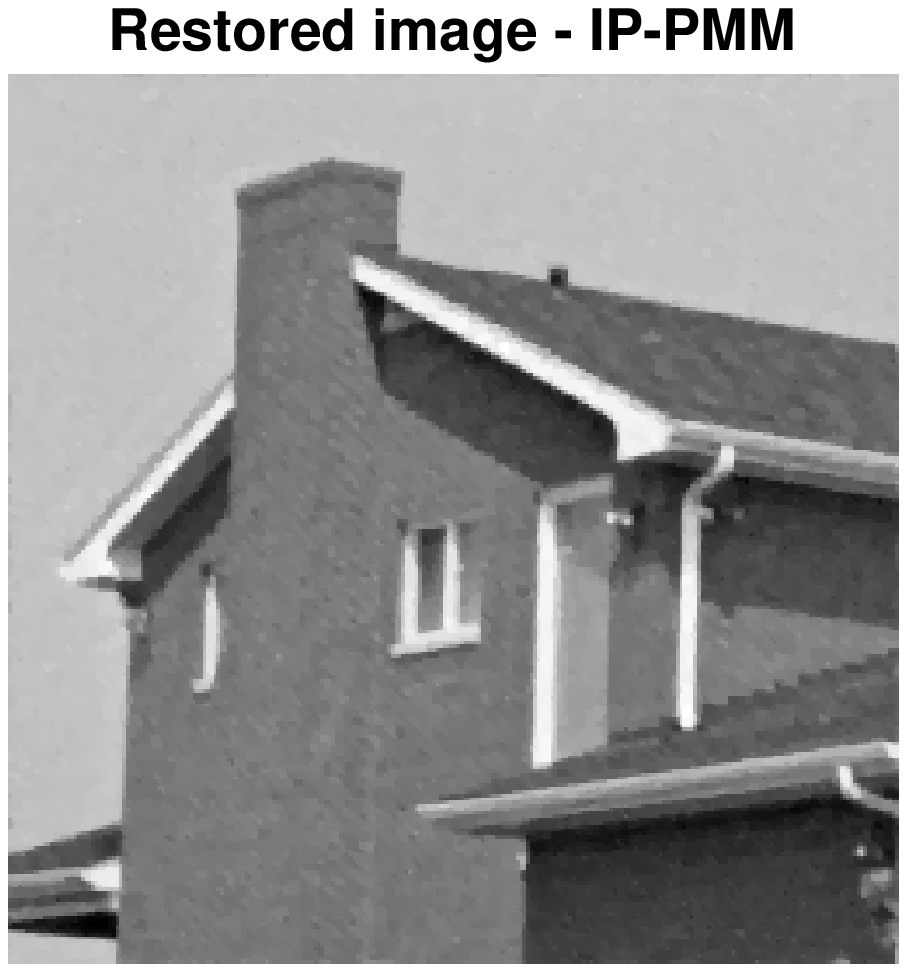} \hfill
    \includegraphics[width=0.3\textwidth]{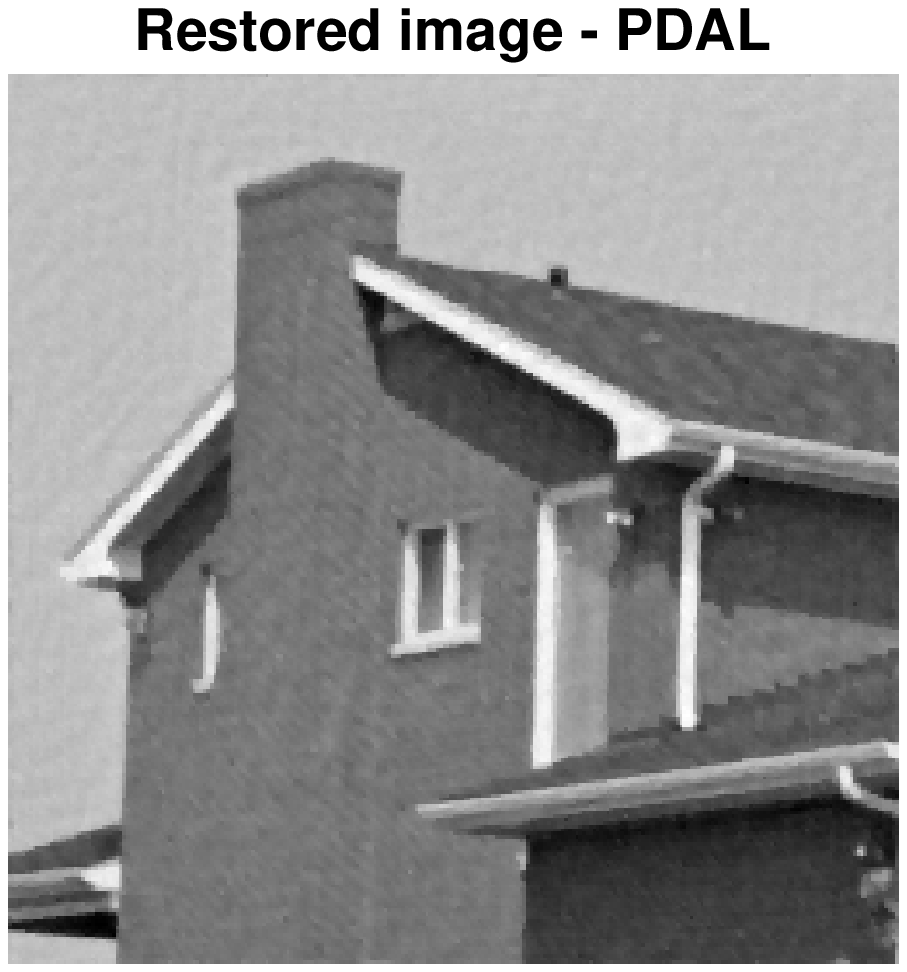} \\[5pt]
    \includegraphics[width=0.3\textwidth]{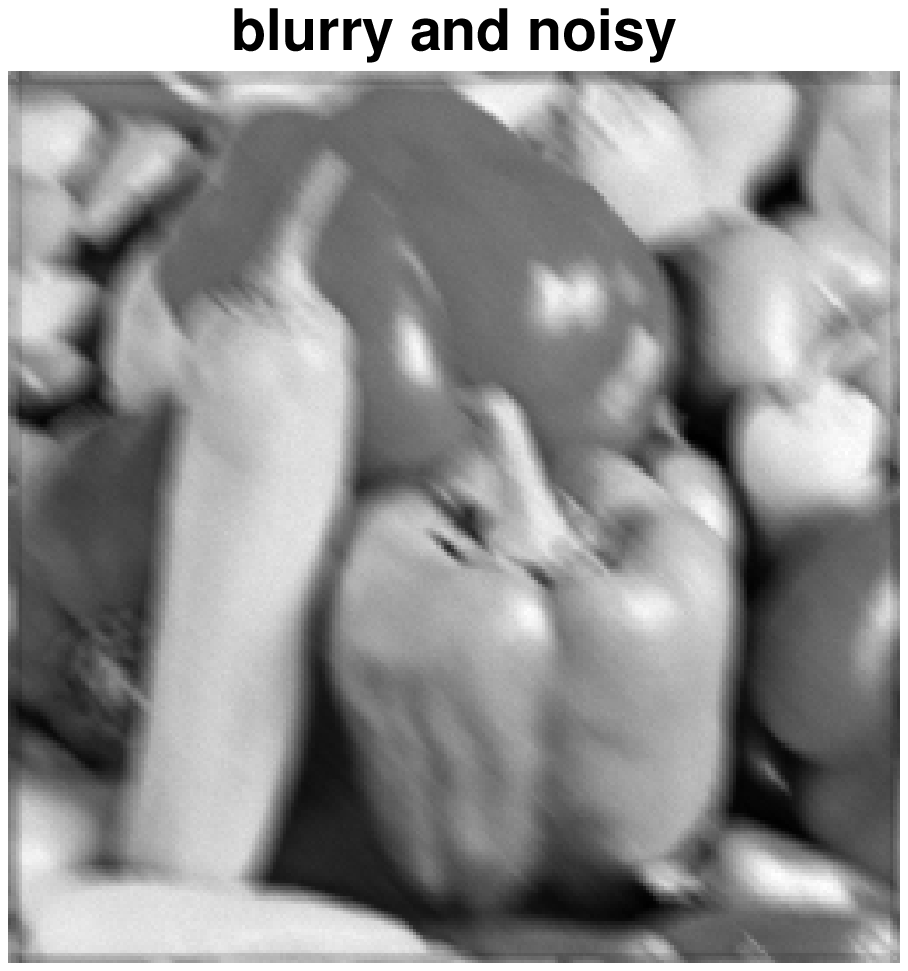} \hfill \includegraphics[width=0.3\textwidth]{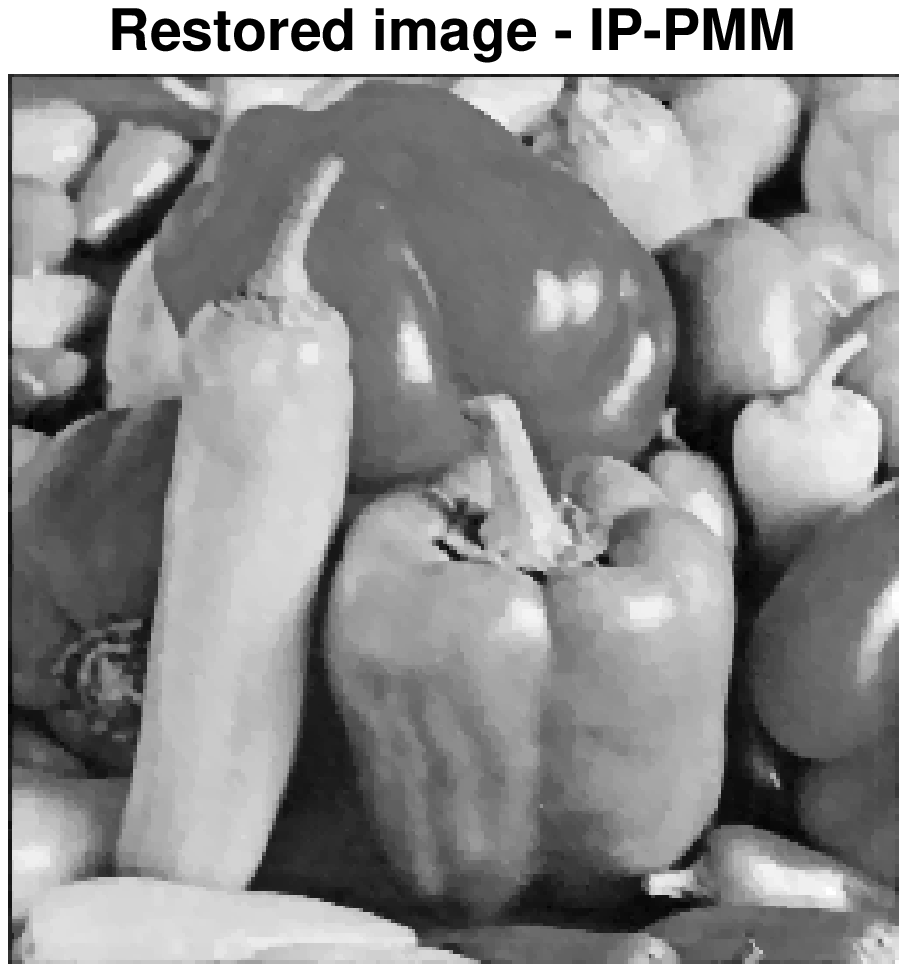} \hfill
    \includegraphics[width=0.3\textwidth]{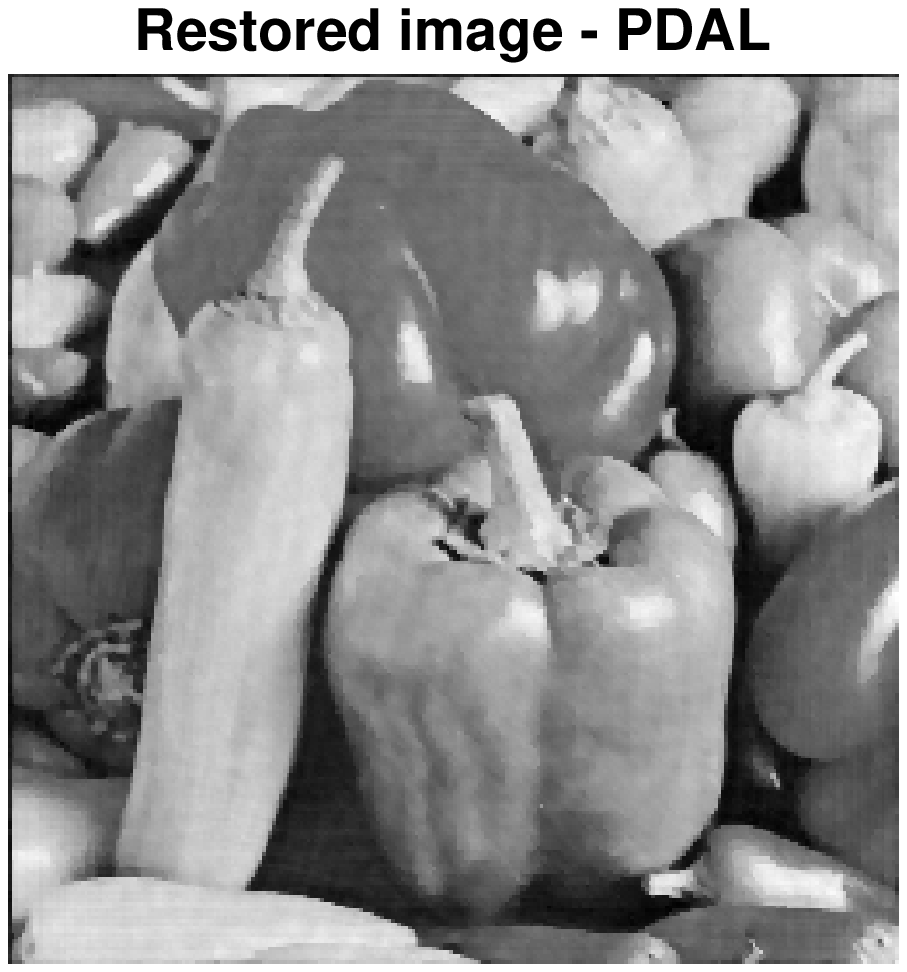}
    \caption{Results on cameraman, house and peppers with MB: noisy and blurry images (\textit{left}), images restored by IP-PMM (\textit{center}), images restored by PDAL (\textit{right}).\label{fig:mb_images}}
\end{figure}
%
%


\section{Linear Classification via Regularized Logistic Regression\label{sec:ClassificationRLR}}

Finally, we deal with the problem of training a linear binary classifier. Let us consider a matrix $D \in \Re^{n \times s}$ whose rows $(d^i)^\top$, with $i\in\{1,\ldots,n\}$, represent the training points, and a vector of labels $g \in \{-1,\,1\}^{n}$. In other words, we have a training set with $n$ binary-labeled samples and $s$ features. According to the logistic model, the conditional probability of having the label $g^i$ given the point $d_i$ has the form
$$
   p_{log}(w)_i = P(g^i | d^i) = \frac{1}{1+e^{-g^i\, w^\top d^i}},
$$
where $w\in\Re^s$ is the vector of weights determining the unbiased linear model under consideration.
By following the maximum-likelihood approach, the weight vector $w$ can be obtained by maximizing the log-likelihood function or, equivalently, by minimizing the \textit{logistic loss function}, i.e., by solving
$$ \min_{w} \  \phi(w) \equiv \frac{1}{n} \sum_{i=1}^{n} \phi_i(w), \quad \ \phi_i(w)=\log\left(1+ e^{-g^i\, w^\top d^i}\right). $$
To cope with the inherent ill-posedness of the estimation process, a regularization term is usually added to the previous model. For large-scale instances, where the features tend to be redundant, an $\ell^1$-regularization term is usually introduced to enforce sparsity in the solution, thus embedding feature selection in the training process.
This results in the well-studied $\ell^1$-regularized logistic regression model:
\begin{equation} \label{eqn:logistic_regression}
    \min_{w} \ \phi(w)  + \tau \|w\|_1,
\end{equation}
where $\tau > 0$.

As done in the previous sections, we can replace the nonsmooth model~\eqref{eqn:logistic_regression} with an equivalent smooth convex programming problem, i.e.,
\begin{equation} \label{smooth logistic regression problem}
    \begin{split}
        \min_{x} & \ f(x) \equiv \phi(w) + c^\top u,\\
        \text{s.t.} &\ Ax = b,\\
        & \ u \geq 0,
    \end{split}
\end{equation}
\noindent where, after introducing the additional constraint 
$u=w$, with $u = [(d^+)^\top,\ (d^-)^\top]^\top \in \Re^{2s}$, and letting $\overline{m} = s$, $\overline{n} = 3s$, we set $x = [w^\top, \ u^\top]^\top \in \Re^{\overline{n}}$, $c = \tau\, e_{2s}$, $b = 0_{\overline{m}}$, and $A \in \Re^{\overline{m} \times \overline{n}}$ defined as $A = [I_s \ \ -I_s \ \ I_s]$. The version of IP-PMM solving problem~\eqref{smooth logistic regression problem} is very similar to the one used to solve~\eqref{eqn:smooth_poispb}. The only difference here lies in the preconditioner. In particular, when solving problems of the form~\eqref{smooth logistic regression problem}, we use the preconditioner defined in~\eqref{poisson_block_diag_preconditioner} (and subsequently analyzed in Theorem~\ref{thm:spectral_analysis_block_diagonal_preconditioner}), but we set $\widetilde{H}_k = \diag(H_k)$.  

\subsection{Computational Experience}

To illustrate the performance of the IP-PMM on this class of problems, we consider a set of three linear classification problems from the LIBSVM dataset for binary classification\footnote{available from \url{https://www.csie.ntu.edu.tw/\textasciitilde cjlin/libsvmtools/datasets/binary.html}}. The names of the datasets, together with their number of features, training points and testing points are summarized in Table~\ref{tab:logregl1_datasets}. For real-sim there is no predetermined separation of data between train and test, hence we apply a hold-out strategy keeping $30\%$ of the data for testing. 

\begin{table}[htbp]
    \begin{small}
        \caption{Characteristics of the $\ell^1$-regularized logistic regression problems}\label{tab:logregl1_datasets} 
        \centerline{
            \begin{tabular}{l|r|r|r}
            	\toprule
            	\textbf{Problem} & \textbf{Features} & \textbf{Train pts} & \textbf{Test pts} \\ \midrule
            	gisette          &              5000 &                6000 &               1000 \\
            	rcv1             &             47,236 &               20,242 &             677,399 \\
            	real-sim         &             20,958 &               50,617 &              21,692 \\ \bottomrule
            \end{tabular}
        }
    \end{small}
\end{table}

To overcome the absence of the hyperplane bias in model \eqref{eqn:logistic_regression}, we add to the data matrices a further column with all ones, hence the resulting size of the problems is equal to $s+1$. For all the problems we set $\tau=\frac{1}{n}$, which is a standard choice in the literature. 

To assess the effectiveness and efficiency of the proposed method we compare it with two state-of-the-art methods:
\begin{itemize}
    \item an ADMM \cite{boyd:2011admm}\footnote{available from \url{http://www.stanford.edu/~boyd/papers/distr_opt_stat_learning_admm.html}};
    \item a MATLAB implementation of the newGLMNET method~\cite{yuan:2012} used in LIBSVM, developed by the authors of~\cite{yue:2019}\footnote{available from \url{https://github.com/ZiruiZhou/IRPN}}.
\end{itemize}
\noindent As in the tests presented in Section~\ref{sec:esperiments_poisson_restoration}, the solution of the augmented system in IP-PMM is performed by means of the MINRES implementation by Michael Saunders' team, with maximum number of iterations equal to 20 and tolerance $tol=10^{-4}$.

We compare the three algorithms in terms of objective function value and classification error versus execution time, on runs lasting $15$ seconds. The plots are reported in Figure~\ref{fig:logreg_l1_plots}. The IP-PMM is comparable with newGLMNET on the gisette instance, characterized by a very dense ($>99\%$) training data matrix, and both IP-PMM and newGLMNET clearly outperform ADMM. On the rcv1 and real-sim instances the IP-PMM method sightly outperforms newGLMNET in terms of classification error, and it is noticeably better in terms of the objective function value.

\begin{figure}[h!]
    \centering
    \includegraphics[width=0.47\textwidth]{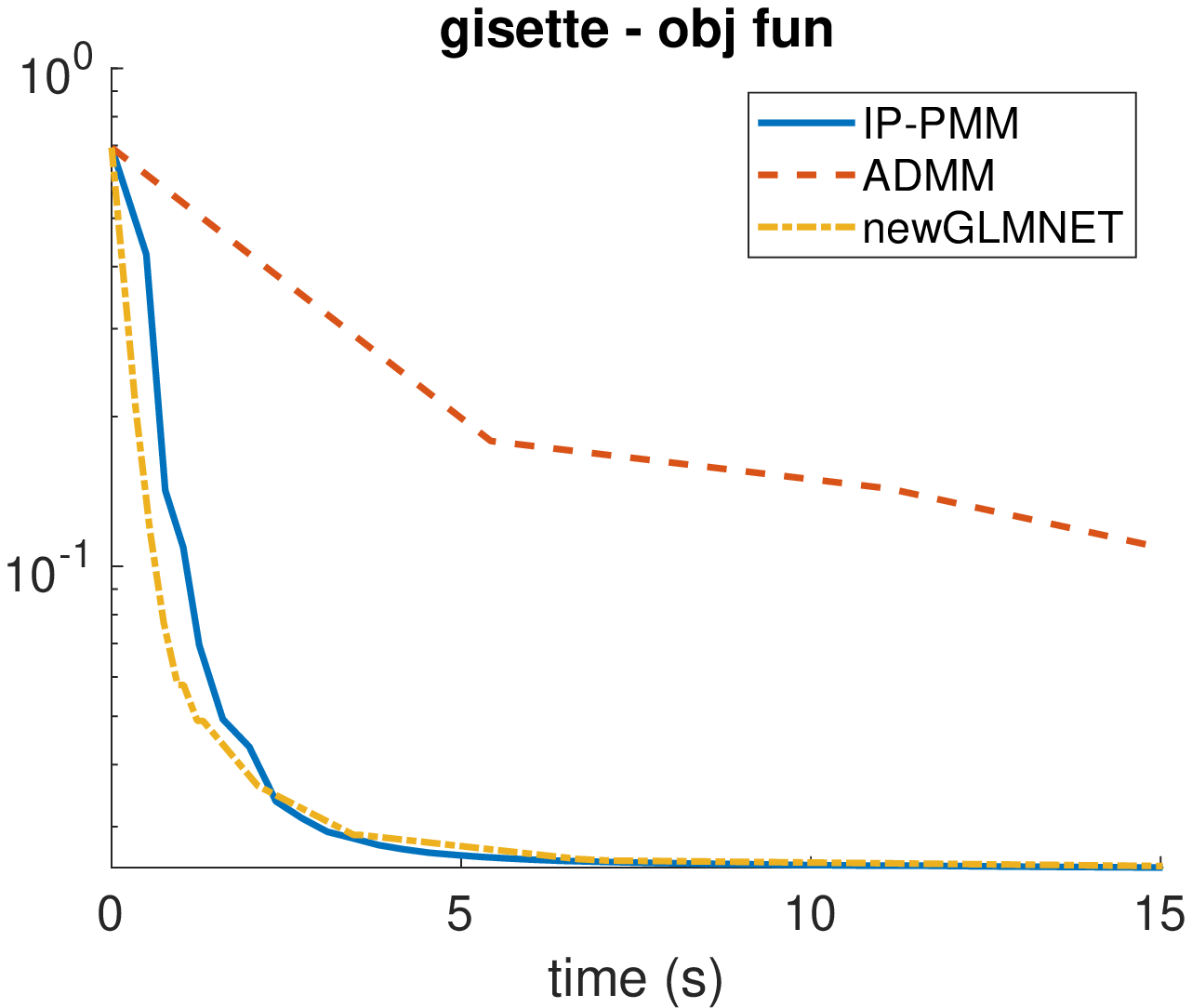} \hfill \includegraphics[width=0.47\textwidth]{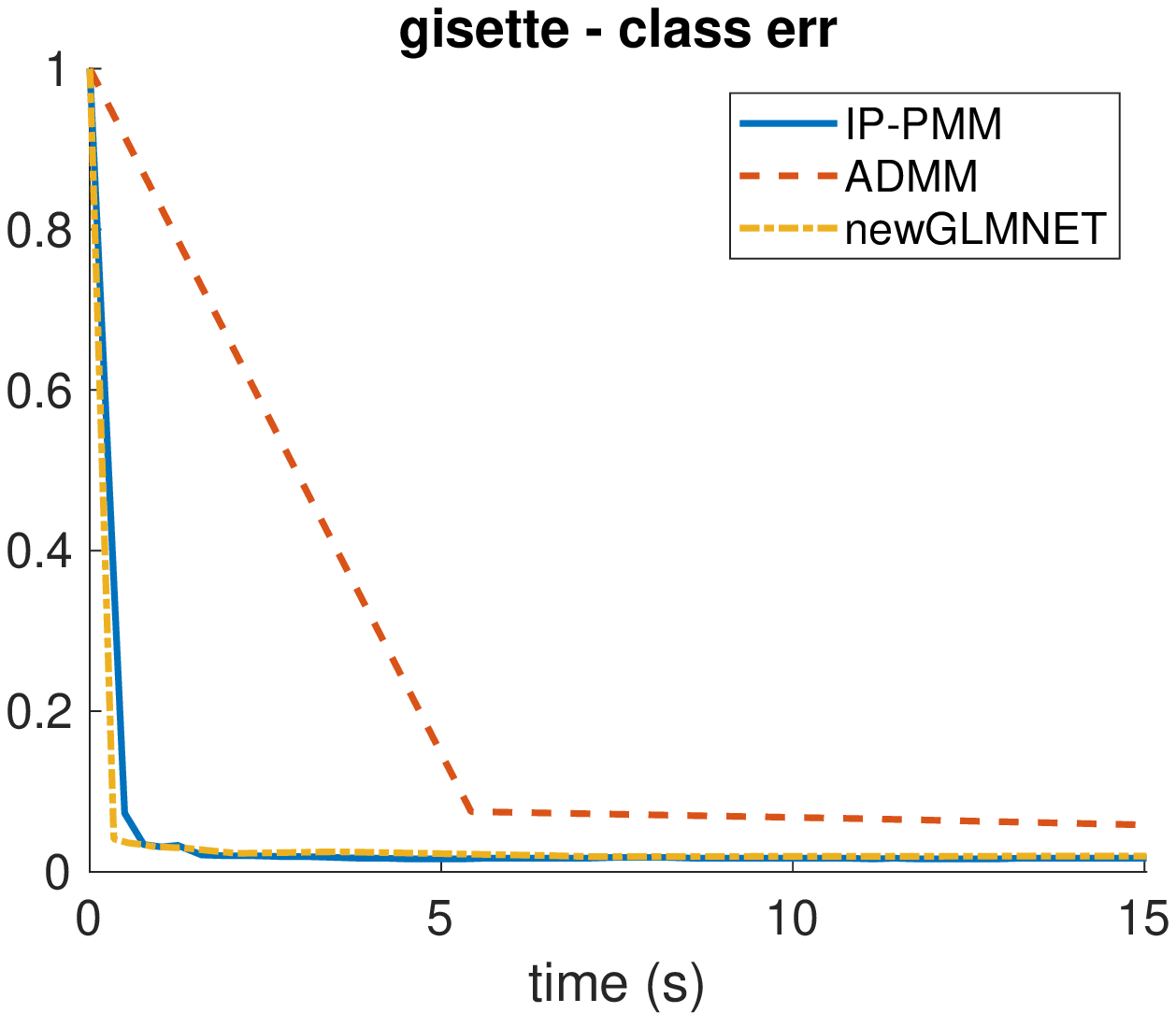}\\
    \includegraphics[width=0.47\textwidth]{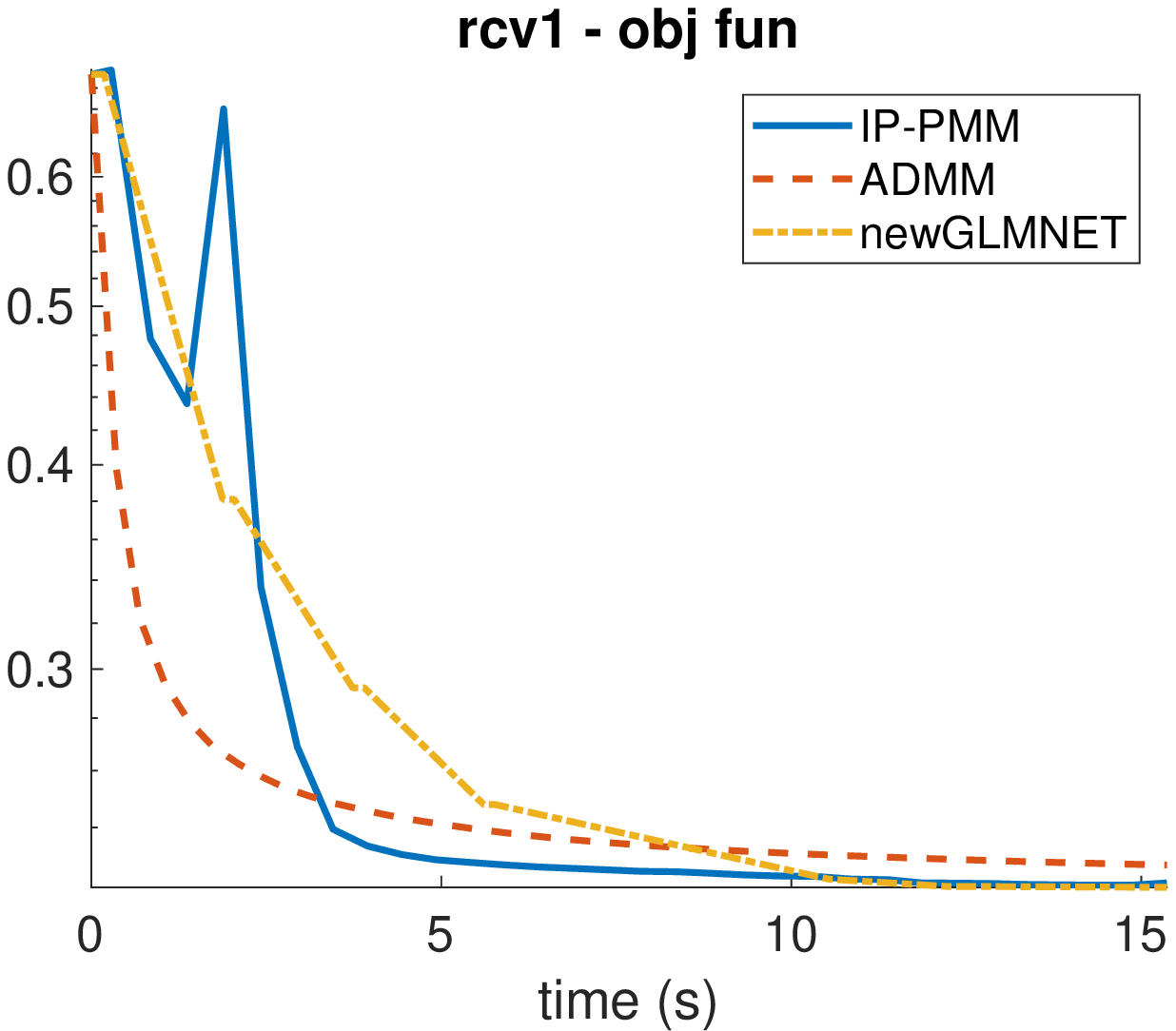} \hfill \includegraphics[width=0.47\textwidth]{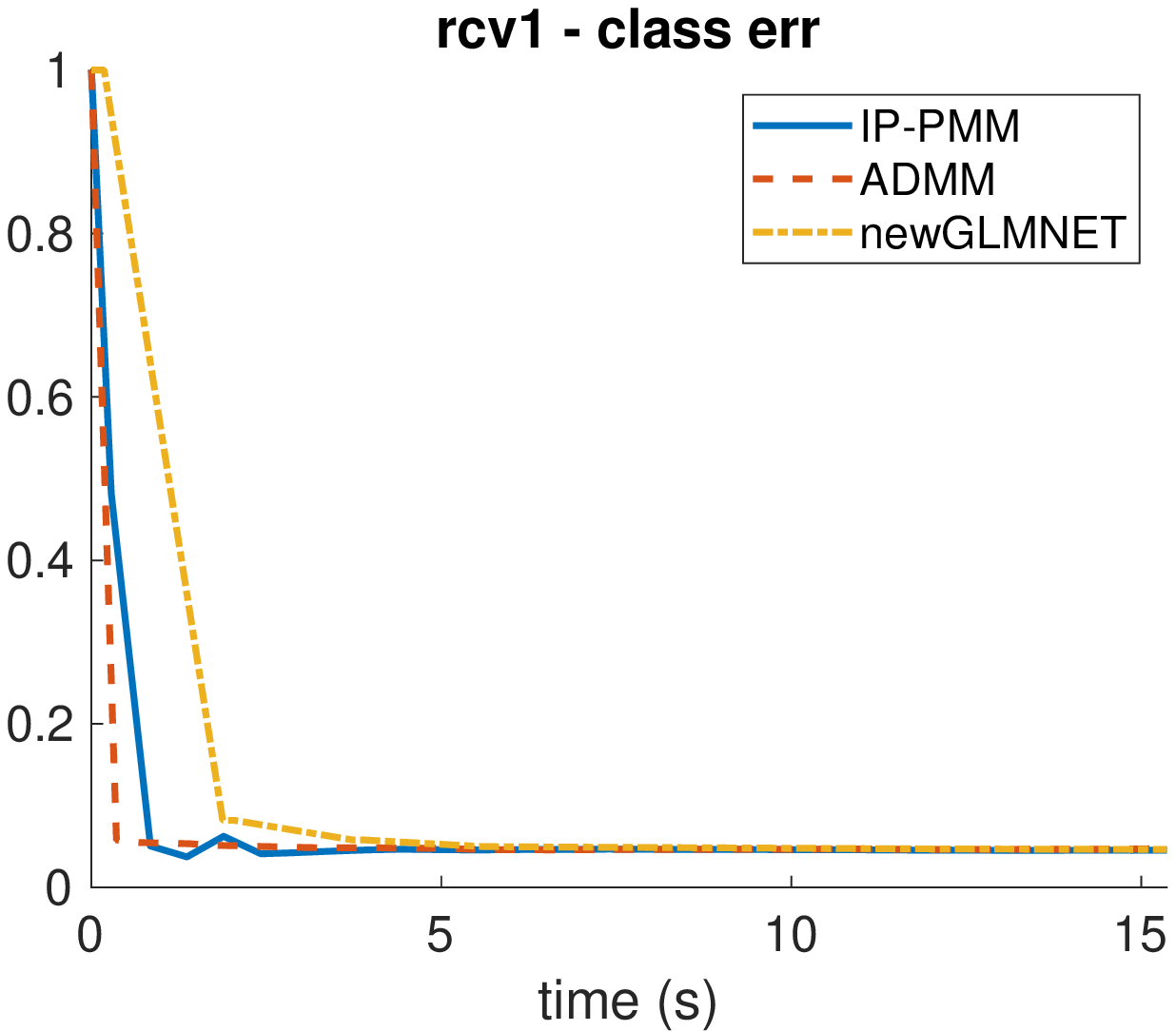}\\
    \includegraphics[width=0.47\textwidth]{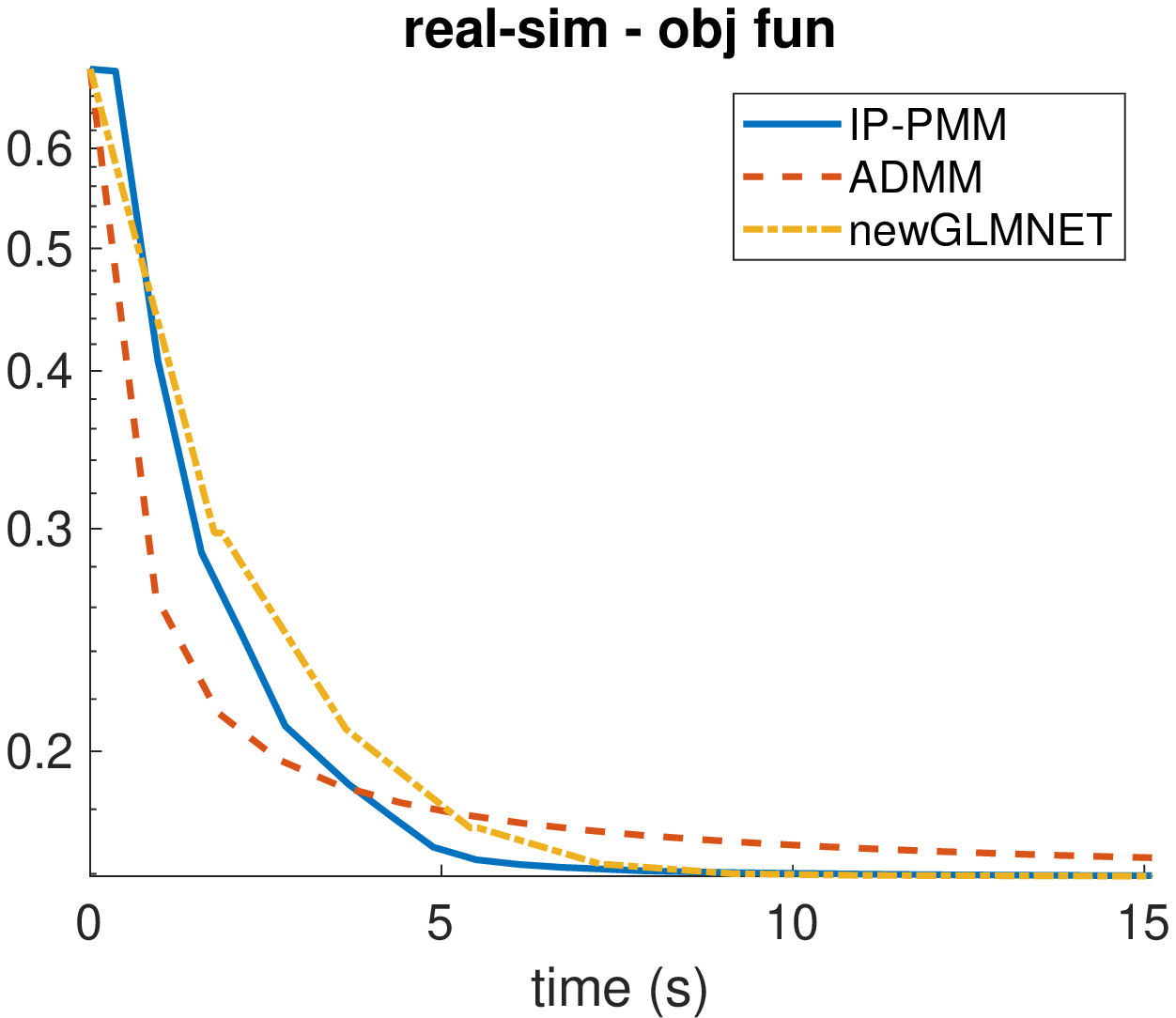} \hfill \includegraphics[width=0.47\textwidth]{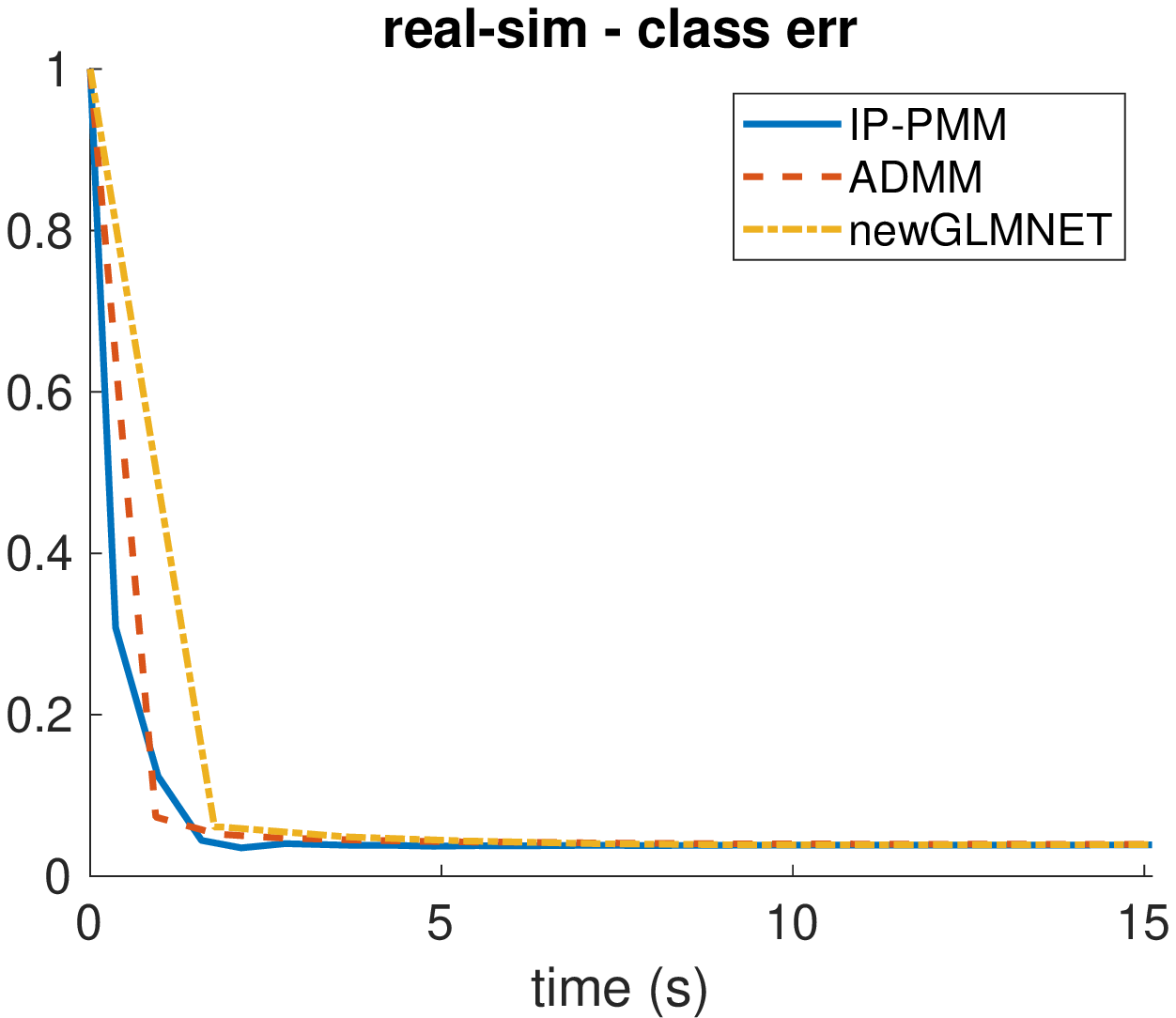}\\ \vskip -9pt
    \caption{Results on the three $\ell^1$-regularized logistic regression problems (objective function value and classification error versus execution time).\label{fig:logreg_l1_plots}}
\end{figure}

\begin{remark}
Let us notice that in the presented experiments, in order to invert each associated preconditioner, we needed to perform a Cholesky decomposition of an approximate normal equations matrix $A \widetilde{H}_k^{-1} A^\top + \delta_k I_{\overline m}$ (or a sub-matrix of it; e.g., as in Section \ref{sec:fMRI}), where $\widetilde{H}_k \approx \nabla^2 f(x_k) + \Theta_k^{-1} + \rho_k I_{\overline n}$. In certain cases, if $A$ has full row-rank, one could instead employ an approximation based on a random sketching strategy, presented in \cite{ChowdhuryLondonAvronDrineasRandomSketch,ChowdhuryYangDrineasPCML}. We should mention however, that this forces one to employ a singular value decomposition to invert the resulting matrix instead of a Cholesky decomposition (which is expected to be faster on the sparse problems under consideration). Furthermore, in the case of rank-deficient matrix $A$ this would create certain computational issues, as then the dropping heuristic presented in Section \ref{sec:DroppingStrategy} would be very expensive to employ (see the discussion in \cite[Section 5]{ChowdhuryLondonAvronDrineasRandomSketch}). Nevertheless, we should mention that in certain applications for which either most of the singular values of $A$ are close to zero, or the Cholesky decomposition of the approximated Schur complement is expensive, such an approach could prove advantageous.
\end{remark}

\section{Conclusions\label{sec:Conclusions}}

We have presented specialized IPMs for quadratic and general convex nonlinear optimization problems
that model various sparse approximation instances. We have shown that by a proper choice of linear
algebra solvers, which are a key issue in IPMs, we are able to efficiently solve the larger but smooth
optimization problems coming from a standard reformulation of the original ones. This confirms
the ability of IPMs to handle large sets of linear equality and non-negativity constraints. Computational experiments
have been performed on diverse applications: multi-period portfolio selection, classification of fMRI data,
restoration of blurry images corrupted by Poisson noise, and linear binary classification via regularized
logistic regression. Comparisons with state-of-the-art first-order methods, which are widely used to tackle
sparse approximation problems, have provided evidence that the presented IPM approach can offer a noticeable advantage
over those methods, especially when dealing with not-so-well conditioned problems.

We also believe that the results presented in this work may provide a basis for an in-depth analysis of the application
of IPMs to many sparse approximation problems, and we plan to work in that direction in the future.

\section*{Acknowledgments}
We thank the anonymous reviewers for their careful reading of the manuscript and their insightful remarks and suggestions,
which allowed us to improve the quality of our work.

\bibliographystyle{siamplain}
\bibliography{gonsar,itersolve,bibportfolio,bib_pougk,biblio_classif_imag}
\end{document}